\documentclass[12 pt]{article}
\usepackage[pdftex]{graphicx}
\usepackage{amsmath}
\usepackage{mathtools}
\usepackage{hyperref}
\usepackage{amsthm}
\usepackage{amsfonts}
\usepackage{amssymb}
\usepackage[margin=1.0in]{geometry}
\usepackage{tikz-cd}
\usetikzlibrary{cd}
\usepackage{enumitem}

\usepackage{tkz-graph}
\usepackage{array,tikz-qtree,ifthen,cancel}

\newtheorem{theorem}{Theorem}[section]
\newtheorem{cor}[theorem]{Corollary}
\newtheorem{lemma}[theorem]{Lemma}
\newtheorem{prop}[theorem]{Proposition}

\newtheorem*{thm7.2}{Theorem \ref{Thm:Ramsey_Space}}
\newtheorem*{thm7.3}{Theorem \ref{Thm:TRS}}
\newtheorem*{thm7.4}{Theorem \ref{Thm:General_GP_Silver}}
\newtheorem*{thm6.16}{Theorem \ref{thm.ARTIdeal}}

\newtheorem*{theorem*}{Theorem}

\theoremstyle{definition}
\newtheorem{defin}[theorem]{Definition}
\newtheorem{fact}[theorem]{Fact}

\newtheorem{que}[theorem]{Question}
\newtheorem{notation}[theorem]{Notation}

\theoremstyle{remark}
\newtheorem*{rem}{Remark}
\newtheorem*{claim}{Claim}

\newcommand{\flim}{\mathrm{Flim}}
\newcommand{\age}{\mathrm{Age}}
\newcommand{\fin}{\mathrm{Fin}}
\newcommand{\fr}{Fra\"iss\'e }

\renewcommand{\phi}{\varphi}

\newcommand{\emb}{\mathrm{Emb}}

\newcommand{\aut}{\mathrm{Aut}}

\newcommand{\dom}{\mathrm{dom}}

\newcommand{\im}{\mathrm{Im}}

\newcommand{\cupdots}{\cup\cdots\cup}

\newcommand{\la}{\langle}
\newcommand{\ra}{\rangle}

\renewcommand{\c}[1]{\mathcal{#1}}
\renewcommand{\cal}[1]{\mathcal{#1}}
\newcommand{\bb}[1]{\mathbb{#1}}
\renewcommand{\rm}[1]{\mathrm{#1}}
\renewcommand{\sf}[1]{\mathsf{#1}}
\renewcommand{\frak}[1]{\mathfrak{#1}}

\newcommand{\bA}{\mathbf{A}}

\newcommand{\cA}{\mathcal{A}}

\newcommand{\bB}{\mathbf{B}}

\newcommand{\cB}{\mathcal{B}}

\newcommand{\bC}{\mathbf{C}}

\newcommand{\cC}{\mathcal{C}}

\newcommand{\bD}{\mathbf{D}}

\newcommand{\cD}{\mathcal{D}}

\newcommand{\rmE}{\mathrm{E}}

\newcommand{\bF}{\mathbf{F}}

\newcommand{\cF}{\mathcal{F}}

\newcommand{\bI}{\mathbf{I}}

\newcommand{\cI}{\mathcal{I}}

\newcommand{\bK}{\mathbf{K}}

\newcommand{\cK}{\mathcal{K}}

\newcommand{\cL}{\mathcal{L}}

\newcommand{\bM}{\mathbf{M}}

\newcommand{\bbN}{\mathbb{N}}

\newcommand{\cO}{\mathcal{O}}

\newcommand{\bbP}{\mathbb{P}}
\newcommand{\rmP}{\mathrm{P}}

\newcommand{\bbQ}{\mathbb{Q}}

\newcommand{\cR}{\mathcal{R}}

\newcommand{\cS}{\mathcal{S}}

\newcommand{\cU}{\mathcal{U}}
\newcommand{\sfU}{\mathsf{U}}

\newcommand{\bX}{\mathbf{X}}

\newcommand{\cX}{\mathcal{X}}

\newcommand{\frakX}{\mathfrak{X}}

\newcommand{\frakY}{\mathfrak{Y}}

\newcommand{\wt}{\widetilde}
\newcommand{\wh}{\widehat}
\newcommand{\ol}{\overline}

\newcommand{\code}{\mathrm{Code}}
\newcommand{\crit}{\mathrm{Crit}}
\newcommand{\Left}{\mathrm{Left}}
\newcommand{\argpath}{\mathrm{Path}}
\newcommand{\econ}{\mathrm{ECon}}
\newcommand{\con}{\mathrm{Con}}
\renewcommand{\sp}{\mathrm{Sp}}
\newcommand{\spnd}{\mathrm{SpNd}}
\newcommand{\cd}{\mathrm{Cd}}
\newcommand{\cdnd}{\mathrm{CdNd}}
\newcommand{\ac}{\mathrm{AC}}
\newcommand{\StEmb}{\mathrm{StEmb}}

\renewcommand{\succ}{\mathrm{Succ}}
\newcommand{\is}{\mathrm{IS}}

\newcommand{\str}{\mathbf{Str}}

\newcommand{\aemb}{\mathrm{AEmb}}

\newcommand{\demb}{\mathrm{DEmb}}

\newcommand{\Path}{\mathsf{Path}}

\newcommand{\sort}{\mathrm{Sort}}
\newcommand{\lex}{\preceq_{lex}}
\newcommand{\MP}{\mathsf{MP}}
\newcommand{\SL}{\mathrm{SL}}
\newcommand{\axiom}{\textbf{A.3(2)}}
\newcommand{\depth}{\mathrm{depth}}

\newcommand{\full}{\mathsf{Full}}

\begin{document}
	\title{Infinite-dimensional Ramsey theory for binary free amalgamation classes}
	\author{Natasha Dobrinen and Andy Zucker}
    \date{}
	\maketitle

 \begin{abstract}
 We develop infinite-dimensional Ramsey theory for 
\fr limits of 
finitely constrained free 
amalgamation classes in finite binary languages.
We show that our approach is optimal and in particular, recovers the
exact big Ramsey degrees proved in
 \cite{BCDHKVZ} for these structures.
A crucial step in the work   develops the new notion of an \axiom-ideal and 
shows  that  Todorcevic's Abstract Ramsey Theorem  \cite{stevo_book}
holds when  Axiom {\bf A.3(2)} is replaced by the weaker assumption of an {\bf A.3(2)}-ideal.
\let\thefootnote\relax
\let\thefootnote\relax
\let\thefootnote\relax\footnote{The first author was supported by NSF grants DMS-1901753 and 2300896, and the second author was supported by NSF grant DMS-2054302 and NSERC grants RGPIN-2023-03269 and DGECR-2023-00412.}
 \end{abstract}

\section{Introduction}

Ramsey theory
was initiated  by
 the following celebrated result:
Given   positive integers $m,r$
and a
coloring 
of $[\mathbb{N}]^m = \{F\subseteq \bbN: |F| = m\}$ 
 into  $r$ colors, there is an infinite  subset  $M\subseteq\mathbb{N}$ such that  $[M]^m$ is monochromatic
\cite{Ramsey}.
Ramsey's theorem is considered  ``finite-dimensional" as the subsets of $\bbN$ being colored are of 
 the fixed finite size $m$.

Infinite-dimensional Ramsey theory is concerned with 
coloring the \emph{infinite} subsets of $\bbN$. Immediately we see that this is more difficult; whereas $[\bbN]^m$ is just a countable set, $[\bbN]^\infty$ is most naturally thought of as a topological space homeomorphic to the Baire space, where the typical basic clopen subset of $[\bbN]^\infty$ has the form $\{M\in [\bbN]^\infty: M\cap n = F\}$ for some positive integer $n$ and some $F\subseteq n$. Using the axiom of choice, Erd\H{o}s and Rado produced a coloring $\chi\colon [\bbN]^\infty\to 2$ with the property that no set of the form $[M]^\infty$ is monochromatic \cite{ErdosRado52}; 
 however, restrictions 
to `definable' colorings have yielded  the desired infinite-dimensional Ramsey theorems.
A subset $\frakX \subseteq [\mathbb{N}]^{\infty}$
 is called {\em Ramsey} if there
is  an infinite subset $M\subseteq \bbN$ with $[M]^\infty\subseteq \frakX$ or $[M]^\infty\cap \frakX = \emptyset$. An intermediate step in 
the progression from finite-dimensional to infinite-dimensional Ramsey theory was the   Nash-Williams
 Theorem \cite{Nash-Williams}, which implies that clopen subsets of the Baire space are Ramsey.
 In unpublished work, Galvin showed that open sets are Ramsey.
Galvin and Prikry \cite{GP} substantially 
 advanced the theory  by showing that  Borel sets are Ramsey, and 
 soon after, Silver \cite{Silver} extended their result to all analytic sets.
 This line of work culminated in the
 Ellentuck Theorem \cite{Ellentuck},
 which characterises
  those sets which are 
 (completely)
  Ramsey
  as exactly those with the property of Baire in the \emph{Ellentuck topology}. In particular, this implies that every \emph{Souslin measurable} subset of $[\bbN]^\infty$ is Ramsey, where a subset of a Polish space $X$ is \emph{Souslin measurable} if it is in the smallest $\sigma$-algebra generated by the topology and which is closed under the \emph{Souslin operation}.

In order to discuss the extensions of both finite and infinite-dimensional Ramsey theory to countable first-order structures, we briefly review some notation and terminology. All structures considered in this paper are relational. Recall that a \emph{relational language} is a set $\cL = \{R_i: i\in I\}$ of relational symbols, where each $R_i$ is equipped with a natural number $n_i\geq 1$ called the \emph{arity} of $R_i$. An \emph{$\cL$-structure} $\bA$ is a set $A$ (the \emph{universe} or \emph{underlying set} of $\bA$) along with a distinguished subset $R_i^{\bA}\subseteq A^{n_i}$ for each $i\in I$. We typically denote $\cL$-structures in bold letters and use the un-bolded letter to denote the underlying set, i.e.\ $A, B, C$ are the underlying sets of $\bA$, $\bB$, $\bC$, etc. If $\bA$ and $\bB$ are $\cL$-structures, an \emph{embedding} $f\colon \bA\to \bB$ is an injection from $A$ to $B$ such that for every $i\in I$, we have $(a_0,...,a_{n_i-1})\in R_i^\bA$ iff $(f(a_0),...,f(a_{n_i-1}))\in R_i^\bB$. Write $\emb(\bA, \bB)$ for the set of embeddings of $\bA$ into $\bB$. A \emph{copy} of $\bA$ in $\bB$ is the image of an embedding of $\bA$ into $\bB$, and we write $\binom{\bB}{\bA}$ for the set of copies of $\bA$ in $\bB$. We write $\bA\leq \bB$ iff $\emb(\bA, \bB)\neq\emptyset$ iff $\binom{\bB}{\bA}\neq \emptyset$. Given $\bA\leq \bB\leq \bC$ and positive integers $t< r$, we write
$$\bC\xrightarrow{copy}(\bB)^\bA_{r, t}$$
if for any coloring $\chi\colon \binom{\bC}{\bA}\to r$, there is $\bB'\in \binom{\bC}{\bB}$ so that $|\chi[\binom{\bB'}{\bA}]|\leq t$. When $t = 1$, we omit it from the notation. Similarly, we write
$$\bC\to (\bB)^\bA_{r, t}$$
if for any $\chi\colon \emb(\bA, \bC)\to r$, there is $g\in \emb(\bB, \bC)$ with $|\chi[g\circ \emb(\bA, \bB)]|\leq t$, once again omitting $t$ from the subscript when $t = 1$. 

Fix a structure $\bK$ with underlying set $\bbN$. Given a structure $\bA$ with $|\bA| = m\in \bbN\cup \{\infty\}$, the set of copies $\binom{\bK}{\bA}$ is then a subset of $[\bbN]^m$. The \emph{age} of $\bK$, denoted $\age(\bK)$, is the class of finite structures $\bA$ with $\emb(\bA, \bK)\neq \emptyset$.  Thus an analog of the finite-dimensional Ramsey theorem for $\bK$ would be as follows: given any $\bA\in \age(\bK)$ and any $r< \omega$, we have $\bK\xrightarrow{copy} (\bK)^\bA_r$. However, due to a result of Hjorth
\cite{Hjorth},
 exact analogues of Ramsey's Theorem usually fail
  for all but vertex colorings 
 when $\bK$ is ultrahomogeneous. An example of this was already  seen  in 
work of 
Sierpi\'{n}ski  in 1933, who constructed a 
 coloring of pairs in the  rational linear order into two colors so that both colors persist in any subset of the rationals forming a dense linear order \cite{Sierpinski}.
Thus we weaken our desired Ramsey theorem by instead asking for the least $t$, if it exists, for which $\bK\xrightarrow{copy} (\bK)^\bA_{r, t}$ for every $r> t$ (equivalently, for just $r = t+1$). This $t$ is called the {\em big Ramsey degree for copies} of $\bA$ in $\bK$ and is denoted $\rm{BRD}^{copy}(\bA, \bK)$. Similarly, one can ask for the least $t$ satisfying $\bK\to (\bK)^\bA_{r, t}$ for every $r> t$, and if it exists, this $t$ is called the \emph{big Ramsey degree} of $\bA$ in $\bK$ and denoted $\rm{BRD}(\bA, \bK)$. The relationship between the two versions is straightforward; we have $\rm{BRD}(\bA, \bK) = \rm{BRD}^{copy}(\bA, \bK)\cdot |\aut(\bA)|$ (for a proof of a similar result, see Proposition~4.4 of \cite{Zucker_Metr_UMF}). In particular, one value is finite iff the other is. For example, when $\bK$ is the rational linear order and $\bA$ is the $2$-element linear order, Sierpi\'{n}ski's coloring is the worst possible; Galvin \cite{Galvin} shows that $\rm{BRD}^{copy}(\bA, \bK) = \rm{BRD}(\bA, \bK) = 2$. We say that a countably infinite structure $\bK$ has \emph{finite big Ramsey degrees} if $\rm{BRD}(\bA, \bK)$ is finite for every $\bA\in \age(\bK)$. 

A rich source of countably infinite structures to consider are \emph{\fr structures}. Recall that given a relational language $\cL$, a \emph{\fr class} of $\cL$-structures is any class $\cK$ of finite $\cL$-structures which is closed under isomorphism, contains arbitrarily large finite structures, and satisfies the following three key properties.
\begin{itemize}
    \item 
    $\cK$ has the \emph{hereditary property} (HP): Whenever $\bB\in \cK$ and $\bA\leq \bB$, then $\bA\in \cK$.
    \item 
    $\cK$ has the \emph{joint embedding property} (JEP): Whenever $\bA, \bB \in \cK$, there is $\bC\in \cK$ with both $\bA\leq \bC$ and $\bB\leq \bC$.
    \item 
    $\cK$ has the \emph{amalgamation property} (AP): Whenever $\bA, \bB, \bC\in \cK$, $f\in \emb(\bA, \bB)$, and $g\in \emb(\bA, \bC)$, there are $\bD\in \cK$, $r\in \emb(\bB, \bD)$, and $s\in \emb(\bC, \bD)$ with $r\circ f = s\circ g$.
\end{itemize}
\fr \cite{fraisse_1954} proves that given a \fr class $\cK$, there is up to isomorphism a unique countably infinite structure $\bK$ with $\age(\bK) = \cK$ and with $\bK$ \emph{ultrahomogeneous}, i.e.\ for any finite $\bA\subseteq \bK$ and any $f\in \emb(\bA, \bK)$, there is $g\in \aut(\bK)$ with $g|_\bA = f$. This unique $\bK$ is called the \emph{\fr limit} of $\cK$ and is sometimes written as $\flim(\cK)$. Conversely, whenever $\bK$ is countably infinite and ultrahomogeneous, $\age(\bK)$ is a \fr class. Given a \fr class $\cK$, we say that $\cK$ has \emph{finite big Ramsey degrees} if $\rm{BRD}(\bA, \flim(\cK))$ is finite for every $\bA\in \cK$. We take a moment to emphasize the difference between the finite-dimensional Ramsey theory of \fr limits and the Ramsey theory of classes of finite structures. A class $\cK$ of finite structures is said to have the \emph{Ramsey property} if for any $\bA\leq \bB\in \cK$, there is $\bC\in \cK$ with $\bB\leq \bC$ so that $\bC\to (\bB)^\bA_2$. Ramsey theory on finite structures  has been  developed  with wide success
from the 1970's through the present
(see, for instance, \cite{HubickaNesetril19}).
An early highlight in this area is  the 
Ne\v{s}et\v{r}il--R\"{o}dl Theorem  that all finite ordered free amalgamation classes in finite languages have the Ramsey property
\cite{NesetrilRodl77, NesetrilRodl83}.
In particular, the classes of finite linear orders, finite ordered graphs, and finite ordered $k$-clique-free graphs all have the Ramsey property. Compare this to the discussion above, where we saw that the $2$-element linear order has big Ramsey degree $2$ in the rational linear order.

Until quite recently, 
finite-dimensional structural Ramsey theory on infinite structures  
has progressed  sporadically. The first non-trivial infinite structure for which the finite-dimensional Ramsey theory has been fully understood
is the rationals as a dense linear order.
In 1979,  Devlin 
characterized the big Ramsey degrees of the rationals in \cite{Devlin}, 
building on unpublished work of Laver establishing upper bounds.
The next complete result only appeared in 2006:
The big Ramsey degrees of the Rado graph were fully characterized by Laflamme, Sauer, and Vuksanovic in \cite{LSV}, showing that the 
upper bounds due to  Sauer \cite{Sauer06} were exact. 
Other work shortly after this giving complete characterizations of big Ramsey degrees include  \cite{NVT08}, for classes of finite-distance ultrametric spaces, and 
 \cite{LNS} for various unary expansions of the class of linear orders as well as the circular directed graph $\mathbf{S}(2)$.

 Characterizing big Ramsey degrees in finitely constrained binary free amalgamation classes took another fifteen years. Recall that an $\cL$-structure $\bF$ is \emph{irreducible} if every pair $a\neq b\in F$ is contained in some non-trivial relation. If $\cF$ is a set of finite irreducible $\cL$-structures, then $\rm{Forb}(\cF)$ denotes the class of finite $\cL$-structures which \emph{forbid}, i.e.\ do not embed,
all $\bF\in \cF$. Every free amalgamation class has the form $\rm{Forb}(\cF)$ for some set $\cF$ of finite irreducible structures. If $\cL$ is a finite relational language with symbols of arity at most $2$ and $\cF$ is finite, we call $\rm{Forb}(\cF)$ a \emph{finitely constrained binary free amalgamation class}.
The class of finite triangle-free graphs was shown to have finite big Ramsey degrees using new methods 
developed by the first author  in  
\cite{Dobrinen_3Free}, and later, these methods were extended to the class of $k$-clique-free graphs for any $k\geq 3$ in \cite{Dobrinen_Henson}. These advances inspired extremely rapid progress on big Ramsey degrees; see for instance \cite{BalkoPO, BalkoForbCycles, BCHKV_3Unif, Barbosa, Braunfeld_etal, CDP_SDAP1, CDP_SDAP2, Hubicka_CS20, Masulovic18}. In
\cite{Zucker_BR_Upper_Bound}, the second author
simplified and abstracted the methods from \cite{Dobrinen_3Free} and \cite{Dobrinen_Henson}
to 
produce a streamlined proof that every finitely constrained binary relational free amalgamation class has finite big Ramsey degrees. The  big Ramsey degrees for  these   classes of structures were completely characterized 
 in \cite{BCDHKVZ}; we will discuss this characterization in more detail later in the introduction. For further  background and  overview regarding the finite-dimensional Ramsey theory of countable structures, we refer the reader to \cite{BCDHKVZ} and \cite{Dobrinen_ICM}.

The problem of developing infinite-dimensional structural Ramsey theory was brought into focus in the paper \cite{KPT} of Kechris, Pestov, and Todorcevic. 
When considering the infinite-dimensional Ramsey theory of a countably infinite structure $\bK$ with $K = \bbN$, we can identify $\binom{\bK}{\bK}$ with a closed subspace of the Baire space. Now, along with the same topological considerations that arose when considering colorings of $[\bbN]^\infty$, we must also consider the effect of big Ramsey degrees. For instance, if $\bK$ is the \fr limit of a \fr class with strong amalgamation and $\bA\in \age(\bK)$ has big Ramsey degree for copies in $\bK$ at least $t$, fix an unavoidable coloring $\chi_\bA\colon \binom{\bK}{\bA}\to t$. We can equip $\binom{\bK}{\bA}$ with a lexicographic well-order induced from $\bbN$, and then define $\chi\colon \binom{\bK}{\bK}\to t$ by setting $\chi(\bM) = \chi_\bA(\min\binom{\bM}{\bA})$. This is a clopen $t$-coloring of $\binom{\bK}{\bK}$, and our assumptions on $\bK$ imply that each color class is unavoidable. Thus the subspaces of  $\binom{\bK}{\bK}$ that we consider  must be further restricted  to produce any viable Ramsey theory. The first result along these lines is due to the first author for the example where $\bK$ is the Rado graph, exhibiting  subspaces of $\binom{\bK}{\bK}$ satisfying an analogue of the Galvin--Prikry Theorem \cite{Dobrinen_Rado}. Of course, a priori there are many subspaces of $\binom{\bK}{\bK}$ one might choose to investigate. Ideally, an infinite-dimensional Ramsey theorem for a given subspace should yield information about the finite-dimensional Ramsey theory of $\bK$; one might even hope to exactly recover the big Ramsey degrees of every $\bA\in \age(\bK)$. The approach we take in this paper is to consider subspaces of $\binom{\bK}{\bK}$ given by \emph{big Ramsey structures}, first defined by the second author in \cite{Zucker_BR_Dynamics}. These are expansions of $\bK$ which simultaneously encode the exact big Ramsey degrees of every finite $\bA\subseteq \bK$. Among the different big Ramsey structures a structure $\bK$ might admit, we will see that the \emph{recurrent} ones (Definition~\ref{Def:Recurrent}), should they exist, are canonical. While the questions of whether having finite big Ramsey degrees implies admitting a big Ramsey structure or whether admitting a big Ramsey structure implies admitting a recurrent big Ramsey structure
are still open, all known examples of \fr classes with finite big Ramsey degrees do admit recurrent big Ramsey structures.

In \cite{Dobrinen_SDAP}, 
the first author developed infinite-dimensional Ramsey theory for \fr structures in finite binary relational languages  with a certain property called SDAP$^+$ (as well as a  weaker property LSDAP$^+$),
building on work in \cite{CDP_SDAP1, CDP_SDAP2}. This property
essentially implies that any ``diagonal antichain" coding the \fr limit is a recurrent big
Ramsey structure for it. 
There are two classes of results in \cite{Dobrinen_SDAP}: an Ellentuck-style theorem for the rationals and more general $\mathbb{Q}$-like structures, and a Galvin--Prikry-style theorem for the  
unconstrained\footnote{
As well,  SDAP$^+$ classes  in finitely many relations of arity at most two include  the \fr classes of  $k$-partite  graphs and more generally, free amalgamation classes with with finitely many forbidden irreducible substructures each of size  two.}
free amalgamation classes considered in this paper,
possibly with an additional rational linear order, each theorem recovering exact big Ramsey degrees.

However, 
the following substantial problems remained open:
\begin{quote}
Do all finitely constrained binary free amalgamation classes carry infinite-dimensional Ramsey theory?
\end{quote}
\begin{quote}
For such classes, even the unconstrained ones,
exactly which topological properties characterize those sets which are Ramsey?
\end{quote}

We answer these questions by developing the infinite-dimensional Ramsey theory for finitely constrained binary free amalgamation classes; these are exactly the classes of structures considered in \cite{Zucker_BR_Upper_Bound} and \cite{BCDHKVZ}.
Examples of such classes are  the class of finite graphs, the class of finite $k$-clique-free graphs ($k\geq 3$), the class of finite $k$-partite graphs with a labeled $k$-partition ($k\geq 2$), and the class of finite graphs with two types of edges (say $R_0$ and $R_1$)
which do not embed triangles with two $R_0$-edges and one $R_1$-edge.
Fix a finitely constrained binary free amalgamation class $\cK$ with limit $\bK$. In \cite{BCDHKVZ}, it is shown that the big Ramsey degree for copies of any $\bA\in \cK$ is exactly the number of \emph{diaries} (Definition~\ref{Def:DiagonalDiary}) which encode the structure $\bA$. Furthermore, it is shown that any two diaries which encode $\bK$ are bi-embeddable,  implying that each such diary is a recurrent big Ramsey structure for $\bK$. As any diary is completely determined by its behavior on substructures of a fixed finite size, it follows that in fact, \emph{every} recurrent big Ramsey structure for $\bK$ is bi-interpretable with some diary. Diaries coding $\flim(\cK)$
then necessarily 
form the starting point for our development of the infinite-dimensional structural Ramsey theory 
of these classes. Theorem~\ref{Thm:General_GP_Silver} is one of the main theorems of this paper; ``$\rm{DEmb}$" refers to ``diary embedding" (Definition~\ref{Def:EmbDiary}).

\begin{thm7.4}
    Fix a finitely constrained binary free amalgamation class $\cK$. Let $\Gamma$ and $\Theta$ be any diaries with $\str(\Gamma)\cong \str(\Theta)\cong \flim(\cK)$. Then for any finite Souslin-measurable coloring of $\demb(\Theta, \Gamma)$, there is $h\in \demb(\Gamma, \Gamma)$ with $h\circ \demb(\Theta, \Gamma)$ monochromatic. 
\end{thm7.4}

We show that Theorem~\ref{Thm:General_GP_Silver} is optimal in a number of ways, completely answering a question of Kechris, Pestov, and Todorcevic \cite{KPT} and recovering exact big Ramsey degrees via clopen colorings.

Theorem~\ref{Thm:General_GP_Silver} is proven by first constructing a fixed diary $\Delta$ coding $\bK$ with desirable properties. An infinite number of levels of $\Delta$ are designated as \emph{strong} (Definition~\ref{Def:StrongDiary}), and given another diary $\Theta$, the \emph{strong embeddings} of $\Theta$ into $\Delta$ are embeddings which interact nicely with the strong levels of $\Delta$ (Definition~\ref{Def:StrongEmbs}). Let $\cR\subseteq \demb(\Theta, \Delta)$ denote the set of strong embeddings of $\Theta$ into $\Delta$. We also consider embeddings of $\Delta$ into $\Delta$ which are the identity up to a strong level, then above this are strong in the above sense; let $\cS\subseteq \demb(\Delta, \Delta)$ denote these embeddings. On $\cS$, let $\leq$ denote the partial order where $\sigma\leq \phi$ iff $\sigma\in \phi\cdot\cS$, and let $\leq_\cR\,\subseteq \cR\times\cS$ be defined so that $\eta \leq_\cR \phi$ iff $\eta\in \phi\cdot \cR$. Then $(\cR, \cS, \leq_\cR, \leq)$ is a monoid-based containment space (Definitions~\ref{Def:ContainmentSpace} and \ref{Def:MonoidBased}), and $(\cS, \leq)$ is a monoid-based topological containment space. In the course of proving Theorems~\ref{Thm:Ramsey_Space} and \ref{Thm:TRS}, 
we make a contribution to abstract Ramsey space theory by showing that Todorcevic's Abstract Ramsey Theorem (Theorem 4.27 of \cite{stevo_book}) still holds upon weakening his axiom \axiom\hphantom{} (see Definition~\ref{Def:A32Ideal}).
\begin{thm7.2}
    With notation as in Definition~\ref{Def:StrongEmbs}, $(\cR, \cS, \leq_\cR, \leq)$ is a  Ramsey  space. More precisely, $(\cR, \cS, \leq_\cR, \leq)$ satisfies \textbf{A.4} and admits an \axiom-ideal. 
\end{thm7.2}
\begin{thm7.3}
    With notation as in Definition~\ref{Def:StrongEmbs}, $(\cS, \leq)$ is a topological Ramsey space. More precisely, $(\cS, \leq)$ satisfies \textbf{A.4} and admits an \axiom-ideal.
\end{thm7.3}

 Sections \ref{Sec:Background on gluings}
and \ref{Sec:DiagonalDiaries} review gluings, age-classes, and diaries,  simplifying some of the work in \cite{BCDHKVZ}. Strong diaries are introduced in 
Section \ref{Sec:Strong diaries}.
Spaces of  \emph{strong embeddings} of diaries into strong diaries
form the foundation  of our infinite-dimensional Ramsey theory. 
In Section \ref{Sec:Ramsey spaces}
we review Ramsey spaces  and  the four axioms which Todorcevic showed suffice to prove  the Abstract Ramsey  Theorem \ref{thm.RS}. 
An obstacle to applying this theory, however,
is that our spaces of embeddings do not satisfy axiom \axiom. 
Indeed, this failure is a property intrinsic to  free amalgamation classes,
not of our set-up, and was already noticed in \cite{Dobrinen_Rado}.
We overcome this obstacle by introducing the new notion of an \emph{\axiom-ideal} (Definition 
\ref{Def:A32Ideal}). 
\begin{thm6.16}
    Suppose $\Omega = (\cR,\cS,\le,\le_{\cR})$ is a containment space which is perfect, respects branching, satisfies axiom {\bf A.4}, and admits an {\bf A.3(2)}-ideal.
Then the conclusion of the Abstract Ramsey Theorem holds.
\end{thm6.16}
The assumptions that $\Omega$ be perfect and respect branching are mild technical assumptions that always hold for monoid-based containment spaces.

Section \ref{Sec:Main Theorem} contains the main theorems.  We first use Theorem~\ref{Thm:Ramsey_Space} along with the bi-embeddability of diaries coding $\bK$ 
 to prove Theorem~\ref{Thm:General_GP_Silver}.
Then, in
Subsection~\ref{Subsec:Ellentuck}, we discuss what Theorem~\ref{Thm:TRS} implies about $\demb(\Delta, \Delta)$. Given a monoid $M$ equipped with a left-invariant metric, the \emph{Ellentuck topology} on $M$ is the coarsest left-invariant topology containing the metric topology (see subsection~\ref{Subsec:Ellentuck}).
Theorem \ref{Thm:TRS} implies
 that for  any positive integer $r$ and coloring
$\chi:\cS\rightarrow  
 r$ which has the Baire-property with respect to the Ellentuck topology on $\cS$,
  there is a strong embedding $h:\Delta\rightarrow\Delta$ with $\chi\circ h$ constant. 
Thus, the analogue of the Ellentuck Theorem holds for the space $\cS$. This yields Corollary~\ref{Cor:Ellentuck}, a strengthening of Theorem~\ref{Thm:General_GP_Silver} in the case $\Gamma = \Theta = \Delta$ which allows all finite colorings of $\demb(\Delta, \Delta)$ which are \emph{Ellentuck Souslin measurable}, i.e.\ in the smallest $\sigma$-algebra containing the Ellentuck topology on $\demb(\Delta, \Delta)$ and closed under the Souslin operation. 
Subsections~\ref{Subsec:A32GenericEmbs} and \ref{Subsec:A4} prove the two main parts of Theorem~\ref{Thm:Ramsey_Space}, namely that $\cS$ contains an \axiom-ideal for $\cR$ and that \textbf{A.4} holds. Slight modifications of these arguments yield the analogous parts of Theorem~\ref{Thm:TRS}. We remark that the proof that \textbf{A.4} holds in our setting amounts to proving a very tight analogue of the Halpern-L\"auchli theorem \cite{halpern_lauchli}.

Section \ref{sec.optimality} shows the optimality of our results.
Theorem \ref{Thm:LargestSemigroup}
shows that for any diary $\Theta$ coding $\bK$, the semigroup $S:=\demb(\Theta,\Theta)$ is as large as possible in the following sense:
If $T\subseteq \emb(\bK,\bK)$ is a Polish subsemigroup of $\emb(\bK,\bK)$ with $S\subsetneq T$, then $T$ cannot satisfy any meaningful infinite-dimensional Ramsey theorem.
Subsection \ref{Subsec:7.2} shows  the impossibility of infinite-dimensional Ramsey theory on the whole space $\emb(\bK,\bK)$  by showing there is a clopen coloring into finitely, or even infinitely many colors such that each color is recurrent no matter how we zoom in to a subcopy of $\bK$. 
We also show the  impossibility of  obtaining an analogue of Ellentuck's Theorem on arbitrary diaries in Subsection \ref{Subsec:7.3}.

To conclude, we point out that 
while soft arguments using  previous work in \cite{Zucker_BR_Upper_Bound} and \cite{BCDHKVZ} are able to  show   that clopen colorings  are Ramsey, here we extend this to very complex colorings.
Moreover, we obtain direct (hard) arguments 
for colorings of diaries, both finite and infinite.
This yields a direct recovery of exact big Ramsey degrees, namely one color per each finite diary coding a given finite structure, 
whereas previously for finitely constrained binary FAP classes, 
big Ramsey degrees were characterized by arguments  going back and forth between the Ramsey theorems for coding trees and the lower bound arguments.


\section{Recurrent big Ramsey structures}

This section gives some background from \cite{Zucker_BR_Dynamics} and \cite{BCDHKVZ} on recurrent big Ramsey structures, making explicit some ideas which only appear implicitly in those works. One new result presented here is Proposition~\ref{Prop:Recover_BRD}, which demonstrates how Ramsey theorems for suitable subspaces of $\binom{\bK}{\bK}$ can recover the exact big Ramsey degrees of $\bK$.  

\begin{defin}
    \label{Def:Expansion}
    Given relational languages $\cL^*\supseteq \cL$ and an $\cL^*$-structure $\bM^*$, the \emph{$\cL$-reduct} $\bM|_\cL$ is the $\cL$-structure on the same underlying set as $\bM^*$ and with the same interpretations of symbols from $\cL$ as in $\bM$. Conversely, if $\bM$ is an $\cL$-structure, an \emph{$\cL^*$-expansion} of $\bM$ is some $\cL^*$-structure $\bM^*$ on underlying set $M$ with $\bM^*|_\cL = \bM$. Given an $\cL$-structure $\bM$, $\bB\leq \bM$, and an $\cL^*$-expansion $\bM^*$ of $\bM$, we set
$$\bM^*(\bB):= \{\bB^*: \bB^* \text{ is an $\cL^*$-expansion of $\bB$ with } \bB^*\leq \bM^*\}.$$ 
We call $\bM^*$ a \emph{precompact} expansion of $\bM$ if $\bM^*(\bB)$ is finite for every finite $\bB\leq \bM$. If $f\in \emb(\bB, \bM)$, we write $\bM^*{\cdot}f$ for the unique $\bB^*\in \bM^*(\bB)$ with $f\in \emb(\bB^*, \bM^*)$. \qed 
\end{defin}

\begin{defin}[\cite{Zucker_BR_Dynamics}]
    \label{Def:BRS}
    If $\cK$ has finite big Ramsey degrees, a \emph{big Ramsey structure} for $\cK$ is an $\cL^*$-expansion $\bK^*$ of $\bK$ for some first-order language $\cL^*\supseteq \cL$ satisfying the following:
\begin{itemize}
    \item
    For every $\bA\in \age(\bK)$, we have $|\bK^*(\bA)| = \rm{BRD}(\bA, \cK)$.
    \item
    On $\emb(\bA, \bK)$, the coloring $f\to \bK^*{\cdot}f$ witnesses that $\rm{BRD}(\bA, \cK)\geq |\bK^*(\bA)|$.
\end{itemize}
\end{defin}

Thus a big Ramsey structure for $\bK$ is an expansion of $\bK$ which simultaneously describes the big Ramsey degrees of every finite substructure of $\bK$. Note that if $\bK^*$ is a big Ramsey structure for $\bK$ and $\eta\in \emb(\bK, \bK)$, then $\bK^*{\cdot}\eta$ is also a big Ramsey structure for $\bK$. While it is an open question whether every \fr structure with finite big Ramsey degrees admits a big Ramsey structure, every known example does.

In fact, for every known example of a \fr structure $\bK$ with finite big Ramsey degrees, we can find big Ramsey structures satisfying a much stronger property. 

\begin{defin}
    \label{Def:Recurrent}
    Given relational languages $\cL^*\supseteq \cL$ and an infinite $\cL$-structure $\bM$,  we call an $\cL^*$-expansion $\bM^*$ of $\bM$  \emph{recurrent} if $\emb(\bM^*, \bM^*{\cdot}\eta)\neq \emptyset$ for every $\eta\in \emb(\bM, \bM)$. Equivalently, $\bM^*$ is a recurrent expansion of $\bM$  iff $\emb(\bM^*, \bM^*)\subseteq \emb(\bM, \bM)$ meets every right ideal of $\emb(\bM, \bM)$. Similarly, if $\bA\subseteq \bM$ is finite, then a coloring of $\emb(\bA, \bM)$ is \emph{recurrent} if upon encoding this coloring as an expansion of $\bM$ (for instance, by introducing a new relational symbol for each color), this expansion is recurrent.

    If $\bK$ is a \fr $\cL$-structure with finite big Ramsey degrees, we call an $\cL^*$-expansion $\bK^*$ a \emph{recurrent big Ramsey structure for $\bK$} if it is both a big Ramsey structure for $\bK$ and a recurrent expansion of $\bK$.
\end{defin}
Note that if $\bM^*$ is a recurrent expansion of $\bM$ and $\eta\in \emb(\bM, \bM)$, then $\bM^*{\cdot}\eta$ is also a recurrent expansion of $\bM$. In particular, if $\bK^*$ is a recurrent big Ramsey structure for $\bK$ and $\eta\in \emb(\bK, \bK)$, then $\bK^*{\cdot}\eta$ is also a recurrent big Ramsey structure for $\bK$.

The definition of a recurrent big Ramsey structure is implicit in \cite{BCDHKVZ}. That paper considered \emph{strong} big Ramsey structures, which are big Ramsey structures satisfying $\bK^*\to (\bK^*)^{\bA^{\!*}}_2$ for every finite $\bA^*\subseteq \bK^*$. By Theorem~5.1.6 from \cite{BCDHKVZ} and the remark that follows, every recurrent big Ramsey structure is strong, and any strong big Ramsey structure in a finite language is recurrent.

The following proposition shows that a recurrent big Ramsey structure for $\bK$ can interpret a wide variety of expansions of $\bK$.  

\begin{prop}
    \label{Prop:Recurrent_BRS}
    Fix $\cL$ a finite relational language and $\bK = \flim(\cK)$ a \fr $\cL$-structure. Suppose $\bK^*$ is a recurrent big Ramsey structure for $\bK$ in a relational language $\cL^* \supseteq \cL$. Suppose $\bK'$ is another expansion of $\bK$ in a finite relational language $\cL'\supseteq \cL$. Then $\bK^*$ interprets an $\cL'$-structure isomorphic to some $\bM'\subseteq \bK'$ with $\bM'|_\cL\cong \bK$.
\end{prop}

\begin{proof}
    Let $\{\bA_i: i< n\}$ list up to isomorphism all members of $\cK$ of size at most the largest possible arity among relation symbols in $\cL'$.  Let $\chi_i'$ and $\chi_i^*$ denote the finite colorings of $\emb(\bA_i, \bK)$ given by $f\to \bK'{\cdot}f$ and $f\to \bK^*{\cdot}f$, respectively. Let $\chi_i = \chi_i'\times \chi_i^*$ denote the product coloring. Find $\eta\in \emb(\bK, \bK)$ so that for every $i< n$, the coloring $\chi_i\circ \eta$ takes on at most $\rm{BRD}(\bA_i, \bK)$ colors. In particular, since $\chi_i^*$ is an unavoidable 
    $\rm{BRD}(\bA_i, \bK)$-coloring of $\emb(\bA_i, \bK)$, we have that for each $i< n$ and $f\in \emb(\bA_i, \bK)$, the expansion $\bK'{\cdot}(\eta\circ f)$ depends only on the expansion $\bK^*{\cdot}(\eta\circ f)$. As $\bK^*$ is recurrent, find $\theta\in \emb(\bK, \bK)$ with $\bK^*{\cdot} (\eta\circ \theta) = \bK^*$. It follows that $\bK^*$ interprets an $\cL'$-structure isomorphic to $\bK'{\cdot}(\eta\circ\theta)$.  
\end{proof}

Proposition~\ref{Prop:Recurrent_BRS} has two key corollaries. The first shows that under mild assumptions, recurrent big Ramsey structures interpret $\omega$-orders of the underlying set.

\begin{cor}
    \label{Cor:Recurrent_BRS_Enum}
    If $\cK$ is a \fr class in a finite relational language, $\bK = \flim(\cK)$, and $\bK^*$ is a recurrent big Ramsey structure for $\bK$, then $\bK^*$ interprets an ordering of $K$ in order type $\omega$. Thus we can identify $\binom{\bK^*}{\bK^*}$ and $\emb(\bK^*, \bK^*)$.
\end{cor}
    
\begin{rem}
    To obtain Corollary~\ref{Cor:Recurrent_BRS_Enum}, one only needs to assume that $\cK$ contains finitely many isomorphism types of structures of size $2$.
\end{rem}

The second corollary shows, again under some finite language assumptions, that recurrent big Ramsey structures are canonical up to bi-embeddability and bi-interpretation. 

\begin{cor}
    \label{Cor:Recurrent_BRS_Canonical}
    If $\cK$ is a \fr class in a finite relational language, $\bK = \flim(\cK)$, and $\bK^*, \bK'$ are recurrent big Ramsey structures for $\bK$  in finite relational languages $\cL^*, \cL'\supseteq \cL$, respectively, then each is bi-interpretable with a substructure of the other.   
\end{cor}

The last proposition of this section shows that finding an expansion of a \fr structure $\bK$ which satisfies the infinite Ramsey theorem recovers exact big Ramsey degrees for $\bK$. 

\begin{prop}
    \label{Prop:Recover_BRD}
    Let $\cL^*\supseteq \cL$ be relational languages, and let $\bK$ be a \fr $\cL$-structure. Suppose $\bK^*$ is a recurrent, precompact expansion of $\bK$ satisfying $\bK^*\to (\bK^*)^{\bA^{\!*}}_2$ for every finite $\bA^{\!*}\subseteq \bK^*$. Then $\bK$ has finite big Ramsey degrees, and $\bK^*$ is a (recurrent) big Ramsey structure for $\bK$. 
\end{prop}

\begin{proof}
    Fix a finite $\bA\subseteq \bK$, towards showing that $\rm{BRD}(\bA, \bK) = |\bK^*(\bA)|$. First, as $\bK^*$ is a recurrent expansion of $\bK^*$, the coloring of $\emb(\bA, \bK)$ given by $f\to \bK^*{\cdot}f$ is unavoidable, showing that $\rm{BRD}(\bA, \bK)\geq |\bK^*(\bA)|$. For the other inequality, fix a finite coloring $\chi\colon \emb(\bA, \bK)\to r$. Note that $\emb(\bA, \bK) = \bigsqcup_{\bA^{\!*}\in \bK^*(\bA)} \emb(\bA^{\!*}, \bK^*)$. Repeatedly using that $\bK^*$ satisfies the infinite Ramsey theorem, find $\eta\in \emb(\bK^*, \bK^*)\subseteq \emb(\bK, \bK)$ such that for each $\bA^{\!*}\in \bK^*$, $\chi\circ \eta$ is constant on $\emb(\bA^{\!*}, \bK^*)$. Hence $|\im(\chi\circ\eta)|\leq |\bK^*(\bA)|$. Once we have $\rm{BRD}(\bA, \bK) = |\bK^*(\bA)|$, it follows that $\bK^*$ is a big Ramsey structure for $\bK$. 
\end{proof}

\section{Background on gluings and age-classes}\label{Sec:Background on gluings}

All background comes from \cite{BCDHKVZ}, but we include some of it here to keep this paper somewhat self-contained. Set-theoretic notation is standard, with $\omega = \bbN$, and given $n< \omega$, we identify $n$ with $\{0,..., n-1\}$.

Throughout, $\cL$ is a finite relational language with symbols of arity at most two, $\cF$ is a finite set of finite irreducible $\cL$-structures, $\cK = \rm{Forb}(\cF)$, and $\bK = \flim(\cK)$. Let $\|\cF\| = \max\{|\bF|: \bF\in \cF\}$. We can arrange, by changing $\cL$ and $\cF$ as necessary, that all of the following hold for any $\bA\in \cK$:
\begin{itemize}
    \item 
    For any $a\in A$, there is exactly one unary predicate $U\in \cL$ so that $U^\bA(a)$ holds.
    \item
    For any $a\in A$ and any binary $R\in \cL$, we have $\neg R^\bA(a, a)$.
    \item
    For any $a\neq b\in A$, there is at most one $R\in \cL$ with $R^\bA(a, b)$.
    \item
    There is a map $\rm{Flip}: \cL\to \cL$ so that for $a\neq b\in A$, $R^\bA(a, b)$ iff $\rm{Flip}(R)^\bA(b, a)$.
\end{itemize}
In particular, fixing $k, \sfU< \omega$ and writing $\cL^\sf{u} = \{U_i: i< \sfU\}$ and $\cL^\sf{b} = \{R_i: i< k\}$ for the unary and binary predicates, respectively, we can treat $U^\bA\colon A\to \sfU$ and $R^\bA\colon A^2\setminus \{(a, a): a\in A\}\to k$ as functions, exactly as was done in \cite{BCDHKVZ}.

We let $\{V_n: n< \omega\}$ be new unary symbols not in $\cL^\sf{u}$, and we set $\cL_d := \cL^\sf{b}\cup \{V_i: i< d\}$. We always assume that the following items hold for any $\cL_d$-structure $\bB$ that we discuss: 
\begin{itemize}
    \item 
    $B\cap \omega = \emptyset$.
    \item
    For any $a\in B$, there is exactly one $i< d$ so that $V_i^\bB(a)$ holds. 
    \item
    Conventions regarding the binary symbols are exactly the same as those for $\cL$-structures. 
\end{itemize}
We treat $V^\bB\colon B\to d$ as a function, exactly as was done in \cite{BCDHKVZ}. We write $\fin(\cL_d)$ for the class of finite $\cL_d$-structures.

\begin{defin}
    \label{Def:Gluings}
    A \emph{gluing} is a triple $\gamma:= (\bX, \rho, \eta)$ where:
\begin{itemize}
    \item
    $\bX$ is a finite $\cL$-structure with $X\subseteq \omega$. We call $X$ the \emph{underlying set} of $\gamma$ and $\bX$ the \emph{structure} of $\gamma$.
    \item 
    There is $d< \omega$ so that $\rho\colon d\to \sfU$ is a function. We call $d$ the \emph{rank} of $\gamma$ and $\rho$ the \emph{sort} of $\gamma$. More generally, we can call any function from a natural number to $\sfU$ a \emph{sort}.
    \item
    $\eta\colon d\times X \to  k$ is a function called the \emph{attachment map} of $\gamma$.
\end{itemize}

Given a sort $\rho\colon d\to \sfU$, we let $\rm{Glue}(\rho)$ be the set of gluings with sort $\rho$, and we let $\tilde{\rho}\in \rm{Glue}(\rho)$ be the gluing $(\emptyset, \rho, \emptyset)$.

Given $\bB$ an $\cL_d$-structure and $\gamma = (\bX, \rho, \eta)$ a rank $d$ gluing, the $\cL$-structure $\gamma(\bB)$ is formed on underlying set $X\cup B$ so that:
\begin{itemize}
    \item 
    $\gamma(\bB)|_X = \bX$.
    \item
    The $\cL^\sf{b}$-part of $\gamma(\bB)$ on $B$ is induced from $\bB$.
    \item
    If $b\in B$, we have $U^{\gamma(\bB)}(b) = \rho\circ V^\bB(b)$.
    \item
    If $b\in B$ and $x\in X$, we have $R^{\gamma(\bB)}(b, x) = \eta(V^\bB(b), x)$.
\end{itemize} 
\end{defin}

As the (abuse of) notation suggests, the gluing $\gamma$ gives rise to a map from the class of $\cL_d$-structures to the class of $\cL$-structures. In this spirit, we adopt the following convention we adopt when considering classes of structures.
\vspace{2 mm}

\noindent
\textbf{Convention}: Given a class $\cC$ of $\cL'$-structures (here $\cL'$ will be $\cL$ or $\cL_d$), we identify $\cC$ with its characteristic function, i.e.\ $\cC(\bA) = 1$ iff $\bA\in \cC$. If $\phi$ is a (class) function whose outputs are $\cL'$-structures, we let $\cC\cdot \phi$ denote the composition $\cC\circ \phi$. In particular, if $\dom(\phi)$ is $\fin(\cL'')$ (again with $\cL'' = \cL$ or $\cL_d$), then we treat $\cC{\cdot}\phi$ as the class of $\cL''$-structures $\{\bB\in \fin(\cL''): \phi(\bB)\in\cC\}$; this is exactly the class whose characteristic function is $\cC{\cdot}\phi$. 

\begin{defin}
 \label{Def:Poset_Of_Classes}
Given a gluing $\gamma$ of rank $d$, we define
$$\cK{\cdot}\gamma := \{\bB\in \fin(\cL_d): \gamma(\bB)\in \cK\}$$
By treating $\cK$ and $\cK{\cdot}\gamma$ as characteristic functions, the notation becomes quite suggestive, i.e.\ $(\cK{\cdot}\gamma)(\bB) = \cK(\gamma(\bB))$. 
   
A class $\cA$ of finite $\cL_d$-structures \emph{contains all unaries} if $\cA$ contains every singleton $\cL_d$-structure. A rank $d$ gluing $\gamma$ \emph{admits all unaries} if $\cK{\cdot}\gamma$ contains all unaries.

    Given a sort $\rho\colon d\to \sfU$, we set
    $$P(\rho):= \{\cK{\cdot}\gamma: \gamma\in \rm{Glue}(\rho)\text{ admits all unaries}\}$$
    We treat $P(\rho)$ as a partial order under inclusion. In the case $\sfU = 1$, there is a unique sort with domain $d$, so we just write $P(d)$. \qed
\end{defin}

\begin{prop}
    \label{Prop:Finite_Posets}
    $P(\rho)$ is finite.
\end{prop}

\begin{proof}
    This is Proposition 2.1.4 from \cite{BCDHKVZ}.
\end{proof}

\begin{prop}
    \label{Prop:Intersections}
    $P(\rho)$ is closed under intersections.
\end{prop}

\begin{proof}
    This is Proposition 2.1.7 from \cite{BCDHKVZ}.
\end{proof}

We now establish some notation allowing us to manipulate gluings and classes of finite $\cL_d$-structures.
\vspace{2 mm}

\noindent
\textbf{Convention}: When discussing partial functions $e\colon d_0\rightharpoonup d$ with $d_0, d<\omega$, we think of $e$ as equipped with the knowledge of what $d_0$ is, even if $\dom(e)\subsetneq d_0$, and of what $d$ is, even if $e$ is not surjective. 

\begin{defin}
    \label{Def:Maps_and_Gluings}
    Let $d_0, d< \omega$ and $e\colon d_0\rightharpoonup d$ be a partial function. If $\bB$ is an $\cL_{d_0}$-structure, then $e{\cdot}\bB$ is the $\cL_d$-structure on the underlying set $\{x\in \bB: V^\bB(x)\in \dom(e)\}$, with $\cL^\sf{b}$-part induced from $\bB$, and with $V^{e\cdot \bB}(x) = e(V^\bB(x))$. In other words, we take $\bB$, throw away the points whose unary is not in $\dom(e)$, then relabel the unaries according to the function $e$. If $\cA$ is a class of finite $\cL_d$-structures, we set
    $$\cA{\cdot}e := \{\bB\in \fin(\cL_{d_0}): e{\cdot}\bB\in \cA\}.$$
    Treating $\cA$ and $\cA{\cdot}e$ as characteristic functions, we have $(\cA{\cdot}e)(\bB) = \cA(e{\cdot}\bB)$.

\end{defin}

    \begin{notation}  
    \label{Notation:Manipulate_Classes}
    By abusing notation, we can identify a finite 
    set
    $S\subseteq \omega$ with its increasing enumeration. When doing this, we interpret $S$ as a total function with domain $|S|$ and with codomain given from context, i.e.\ writing $S\subseteq d$ indicates that the codomain is $d$. If $S\subseteq d$, $\cA$ is a class of finite $\cL_d$-structures, and keeping in mind the convention before Definition~\ref{Def:Maps_and_Gluings}, then we can write $\cA{\cdot}S$, $\rho\circ S$, etc. We write $\cA^i$ for $\cA{\cdot}\{i\}$. 

    Given a sort $\rho\colon d\to \sfU$, $\cA\in P(\rho)$, $S\subseteq d$, and $\cB\in P(\rho\circ S)$, then letting $f\colon d\rightharpoonup |S|$ denote the partial inverse of the increasing enumeration of $S$, we write $\cA|_{(S, \cB)}:= \cA\cap \cB{\cdot}f$, thinking of this as ``$\cA$ restricted at $S$ to $\cB$." If $S = \{i\}$, we write $\cA|_{(i, \cB)}$.
    \end{notation}

\begin{defin}
    \label{Def:Consecutive}
    Fix a sort $\rho\colon d\to \sfU$. We say that $\cA\supseteq \cB\in P(\rho)$ are \emph{consecutive} if there is no $\cC\in P(\rho)$ with $\cA\supsetneq \cC\supsetneq \cB$. Write $\con(\rho) = \{(\cA, \cB)\in P(\rho)^2: \cA\supseteq \cB\text{ are consecutive}\}$. We say that $(\cA, \cB)\in \con(\rho)$ is \emph{essential} on some $S\subseteq d$ if there is a minimal-under-inclusion $\bC\in \cA\setminus \cB$ with underlying set $C$ such that $S = \{V^\bC(x): x\in C\}$. We say that $(\cA, \cB)\in \con(\rho)$ is \emph{essential} if it is essential on $d$, and we write $\econ(\rho) = \{(\cA, \cB)\in \con(\rho): (\cA, \cB) \text{ is essential}\}$. \qed  
\end{defin}

By arguments very similar to those of Proposition~\ref{Prop:Finite_Posets}, such a $\bC$ must satisfy $|\bC|< \|\cF\|$, so in particular, $|S|<\|\cF\|$. 

\begin{prop}
    \label{Prop:Consecutive}
    With notation as in Definition~\ref{Def:Consecutive}, if $(\cA, \cB)\in \con(\rho)$, then there is a unique $S\subseteq d$ on which $(\cA, \cB)$ is essential. Fixing this $S$, if $S'\subseteq d$ and $S\subseteq S'$, then $(\cA{\cdot}S', \cB{\cdot}S')\in \con(\rho\circ S')$, and if $S\not\subseteq S'$, we have $\cA{\cdot}S' = \cB{\cdot}S'$.
\end{prop}

\begin{proof}
    This is Proposition 2.2.6 from \cite{BCDHKVZ}.
\end{proof}

\subsection{Paths, free paths, and path sorts}

Given $i< \omega$, let $\iota_i\colon 1\to \omega$ be the function with $\iota_i(0) = i$.  Sometimes the intended codomain of $\iota_i$ is some $d< \omega$ (for the purposes of forming the partial function $\iota_i^{-1})$, and this will be clear from context. If $i< \sfU$, then $\iota_i$ is also a sort, and we write $P_i$ for $P(\iota_i)$. 

\begin{defin}
    \label{Def:Paths}
    A \emph{path} through $P_i$ is any
    subset $\frak{p}\subseteq P_i$ linearly ordered by inclusion. Write $\Path_i$ for the set of paths through $P_i$ and $\Path := \bigsqcup_{i< \sfU} \Path_i$. A \emph{maximal} path is maximal under inclusion; write $\sf{MP}_i$ for the set of maximal paths through $P_i$ and  $\sf{MP} := \bigsqcup_{i< \sfU}\sf{MP}_i$.   We equip $\sf{MP}$ with an arbitrary linear order $\leq_{\MP}$. We write $\sf{u}\colon \sf{MP}\to \sfU$ for the map sending $\frak{p}\in \sf{MP}$ to the $i< \sfU$ with $\frak{p}\in \sf{MP}_i$.
     
    A \emph{full} path is an initial segment of a maximal path; write $\sf{Full}_i$ for the set of full paths through $P_i$ and $\full := \bigsqcup_{i< \sfU} \full_i$. Given $\frak{p}\in \sf{Full}_i$, $\max(\frak{p})$ and $\min(\frak{p})$ denote the maximal and minimal members of $\frak{p}$ under inclusion. We note that $\min(\frak{p}) = \min(P_i) = \bigcap_{\cA\in P_i} \cA$ (Proposition~\ref{Prop:Intersections}). When $\frak{p}\in \MP_i$, then $\max(\frak{p}) = \max(P_i) = \cK{\cdot}\tilde{\iota}_i$, where we recall that $\tilde{\iota}_i$ is the gluing $(\emptyset, \iota_i, \emptyset)$.
    Up to bi-interpretation, $\cK{\cdot}\tilde{\iota}_i$ is the class of all structures in $\mathcal{K}$ with all vertices having unary relation $i$. Given $\frak{p}\in \sf{Full}_i$, we write $\frak{p}'\in \sf{Full}_i$ for the full path $\frak{p}\setminus \{\max(\frak{p})\}$.  \qed 
\end{defin}

Some of the members of $\sf{MP}$ 
will be \emph{free} while others will be \emph{non-free}. This will only depend on $\sf{u}(\frak{p})$.

\begin{defin}
    \label{Def:Free}
    We say that $i< \sfU$ is \emph{free} if there are $j< \sfU$ and $1\leq q< k$ such that, letting $\gamma_{j, i, q} = (\bX, \iota_i, \eta)$ denote the gluing with $X = \{0\}$, $U^\bX(0) = j$, and $\eta(0, 0) = q$, then $\cK{\cdot}\gamma_{j, i, q} = \cK{\cdot}\tilde{\iota}_i = \max(P_i)$. We call such $(j, q)$ a \emph{free pair} for $i$.

    Otherwise, we call $i$ a \emph{non-free} unary. Let $\sfU_{fr}\subseteq \sfU$ denote the set of free unaries and $\sf{U}_{non} = \sfU\setminus \sf{U}_{fr}$ the set of non-free unary predicates. 

    We call $\frak{p}\in \sf{MP}$ a \emph{free path} or a \emph{non-free path} depending on whether $\sf{u}(\frak{p})$ is free or non-free. Write $\sf{MP}_{fr}$ for the free paths and $\sf{MP}_{non}$ for the non-free paths. \qed
\end{defin}

We also generalize the notion of sort as follows.

\begin{defin}
    \label{Def:Path_Sorts}
     A \emph{path sort} is a function $\rho\colon d\to \sf{MP}$ for some $d< \omega$. Given a path sort $\rho$, we set 
     $$P(\rho) = \{\cA\in P(\sf{u}\circ \rho): \forall\, i< d \, (\cA^i\in \rho(i))\}$$
     Note that two elements of $P(\rho)$ are consecutive in $P(\rho)$, iff they are also consecutive in $P(\sf{u}\circ\rho)$. Thus we have $\con(\rho) = \con(\sf{u}\circ\rho)\cap P(\rho)^2$ and $\econ(\rho)= \econ(\sf{u}\circ \rho)\cap P(\rho)^2$.  
\end{defin}

\subsection{Abstract splitting events}

Given a path sort $\rho\colon d\to \sf{MP}$, a class $\cA\in P(\rho)$, and some $i <  d$, we discuss an abstract notion of what it should mean to ``split" at position $i$. This is the operation that will occur at ``splitting levels" of diaries. First, let $\sp(d, i)\colon (d+1)\to d$ be the unique non-decreasing surjection with $\sp(d, i)^{-1}(\{i\}) = \{i, i+1\}$. Using $\sp(d, i)$, we obtain a new sort $\sp(\rho, i) := \rho\circ \sp(d, i)$. 

We then take $\cA\in P(\rho)$ and form $\sp(\cA, i)$ (read ``split $\cA$ at $i$"). The definition of $\cA$ will depend on whether or not $\rho(i)\in \sf{MP}_{fr}$. 
\begin{defin}
    \label{Def:SplittingClass}
    With notation as above, the class $\sp(\cA, i)\in P(\sp(\rho, i))$ is defined as follows.
    \begin{enumerate}
        \item 
        If either $\rho(i)\in \sf{MP}_{fr}$ or $\cA^i\neq \max(\frak{p})$, we set $\sp(\cA, i) = \cA{\cdot}\sp(d, i)$.
        \item
        If $\rho(i)\in \sf{MP}_{non}$ and $\cA^i = \max(\frak{p})$, then we set
        $\sp(\cA, i) = (\cA{\cdot}\sp(d, i))|_{(i+1, \max(\frak{p}'))}$.
    \end{enumerate}
\end{defin}
In case $2$, notice that $\sp(\cA, i)^{i+1} = \max(\frak{p}')$. We remark that this is one of the main ways that the difference between free and non-free paths arises.

\subsection{Controlled coding}

We now review the notion of controlled coding triples from \cite{BCDHKVZ}, which is the last ingredient we need to give the definition of a diary.

\begin{defin}
    \label{Def:CodingClass}

    Given $j< d$, a function $\phi\colon (d-1)\to k$, and $\bB$ an $\cL_{d-1}$-structure with underlying set $B$, we define $\rm{Add}_{j, \phi}(\bB)$ an $\cL_d$-structure with underlying set $B\cup \{x\}$, where $x\not\in B$ is some new point. On $B$, $\rm{Add}_{j, \phi}(\bB)$ induces the structure $(d{\setminus}\{j\}){\cdot}\bB$. As for the new point $x$, we set $V^{(j,\phi)(\bB)}(x) = j$ and given $b\in B$, we set $R^{\rm{Add}_{j, \phi}(\bB)}(b, x) = \phi(V^\bB(b))$. 
    Thus $\rm{Add}_{j, \phi}(\bB)$ contains exactly one point with unary $j$, namely $x$.

    Given a class $\cA$ of $\cL_d$-structures, we define the class 
    $$\cA{\cdot}\rm{Add}_{j, \phi} = \{\bB\in \fin(\cL_{d-1}): \rm{Add}_{j, \phi}(\bB)\in \cA\}.$$ 
    If $\cA = \cK{\cdot}\gamma$ for some gluing $\gamma$, then also $\cA{\cdot}\rm{Add}_{j, \phi} = \cK{\cdot}\gamma_{j, \phi}$ for some gluing $\gamma_{j, \phi}$.
    
    Given a path sort $\rho\colon d\to \sf{MP}$, $\cA\in P(\rho)$, $j< d$, and $\phi\colon (d-1)\to k$, we call $(\cA, j, \phi)$ a \emph{controlled coding triple} if the following all hold.
    \begin{enumerate}
        \item 
        $\cA{\cdot}\rm{Add}_{j, \phi} = \cA{\cdot}(d{\setminus}\{j\})$.
        \item
        Writing $i = \sf{u}\circ \rho(j) < \sfU$, we have $\cA^j = \min(P_i)$ (note that $P_i$ is finite and closed under intersections, hence has a minimum element).
        \item
        If $\cB\in P(\rho)$ satisfies $\cB\subseteq \cA$ and $\cB{\cdot}(d{\setminus}\{j\}) = \cB{\cdot}\rm{Add}_{j, \phi} = \cA{\cdot}(d{\setminus}\{j\})$, then $\cB = \cA$. \qed
    \end{enumerate}
\end{defin}

Controlled coding triples exist in abundance. Namely, given a path sort $\rho\colon d\to \sf{MP}$, $\cA\in P(\rho)$, $j< d$, and $\phi\colon (d-1)\to k$, then if $\cA{\cdot}\rm{Add}_{j, \phi}$ contains all unaries, then there is a unique $\cB\in P(\rho)$ with $\cB\subseteq \cA$, $\cB{\cdot}\rm{Add}_{j, \phi} = \cA{\cdot}\rm{Add}_{j, \phi}$, and with $(\cB, j, \phi)$ a controlled coding triple. This $\cB$ can be formed by intersecting all $\cD\in P(\rho)$ with $\cD\subseteq \cA$ and $\cD{\cdot}\rm{Add}_{j, \phi} = \cA{\cdot}\rm{Add}_{j, \phi}$ (for more details, see Lemma~3.2.4 of \cite{BCDHKVZ}). We denote this unique $\cB$ by $\la \cA, j, \phi\ra$. 

\section{Diaries}
\label{Sec:DiagonalDiaries}

Diaries
are abstractions of the  information  responsible for exact big Ramsey degrees 
and arise  as follows.
Given $\mathbf{K}$, we may enumerate its universe, or equivalently, assume that its universe is $\omega$.
This induces the tree $S:=S(\mathbf{K})$ of (quantifier-free) $1$-types over the finite initial segments of the structure. 
For each $n\in\omega$ we let $c(n)$ denote  the $1$-type of the vertex $n$ in $\mathbf{K}$ over the initial substructure $\mathbf{K}|n$; thus, $c(n)$ is a node of length $n$.
Then $S$ along with the function $c:\omega\rightarrow S$ is a {\em coding tree} of $1$-types (see \cite{CDP_SDAP1} for a precise definition). Upon enriching the coding tree of $1$-types with information on each level about what configurations of coding nodes are possible above a given level set of nodes, we obtain the \emph{aged coding tree of $1$-types} (see \cite{Zucker_BR_Upper_Bound}, \cite{BCDHKVZ}). Given a finite $\bA\in \cK$, one can produce a (possibly infinite) coloring of $\binom{\bK}{\bA}$ by coloring $\bA'\in \binom{\bK}{\bA}$ based on the shape of the aged subtree induced by $\{c(a): a\in A'\}$.
One can find $\bK'\in \binom{\bK}{\bK}$ so that on $\binom{\bK'}{\bA}$, the coloring takes exactly $\rm{BRD}^{copy}(\bA, \bK)$ many values and is unavoidable (in fact recurrent). Diaries are then an abstraction  of this information, describing all of the necessary, ``unavoidable" behavior of subtrees coding subcopies of $\bK$, and from any such diary, one can build a recurrent big Ramsey structure.


\subsection{Conventions about trees}
\label{SubSec:Trees}
    Given $t\in \MP\times k^{<\omega}$, we write $t = (t^\sf{p}, t^\sf{seq})$ with $t^\sf{p}\in \MP$ and $t^\sf{seq}\in k^{<\omega}$, and we write $t^\sf{u}:= \sf{u}(t^\sf{p})$ (here $\sf{p}$ refers to ``path" and $\sf{u}$ refers to ``unary"). Write $\ell(t) = \dom(t^\sf{seq})$, which we call the \emph{level} of $t$, and given $m< \ell(t)$, we often abuse notation and write $t(m)$ for $t^\sf{seq}(m)$. Similarly abusing notation, if $q< k$, we write $t^\frown q$ for $(t^\sf{p}, (t^\sf{seq})^\frown q)$.

\begin{defin}
    \label{Def:Binary_tree_rels}
    We define the partial orders $\leq_\ell$, $\sqsubseteq$, and $\lex$ and the partial binary operation $\wedge$ on $\MP\times k^{<\omega}$; let $s, t\in \MP\times k^{<\omega}$.
    \begin{enumerate}
        \item 
        We write $s\leq_\ell t$ iff $\ell(s)\leq \ell(t)$.
        \item 
        If $m< \ell(s)$, we write $s|_m$ or $\pi_m(s)$ for the initial segment of $s$ at level $m$. We write $s\sqsubseteq t$ if $s^\sf{p} = t^\sf{p}$ and $s^\sf{seq} = t^\sf{seq}|_{\ell(s)}$.
        \item 
        We set $s\lex t$ iff $s^\sf{p}<_\MP t^\sf{p}$ or $s^\sf{p} = t^\sf{p}$ and $s^\sf{seq}\lex t^\sf{seq}$ (where the latter $\lex$ is the usual lexicographic order on $k^{<\omega}$).
        \item 
        If $s^\sf{p} = t^\sf{p}$, we write $s\wedge t$ for the largest common initial segment of $s$ and $t$. \qed
    \end{enumerate}
\end{defin}

A \emph{level subset} of $\MP\times k^{<\omega}$ is simply a subset of $\MP\times k^m$ for some $m< \omega$; given a level subset $T\subseteq \MP\times k^{<\omega}$, write $\ell(\Delta)$ for the common level of every $t\in T$. If $S = \{s_0\lex\cdots\lex s_{d-1}\}\subseteq \MP\times k^{<\omega}$ is a level subset, we let $\sort(S)\colon d\to \MP$ be the path sort given by $\sort(S)(i) = s_i^\sf{p}$. If $X\subseteq \MP\times k^{<\omega}$, then the $\sqsubseteq$-downwards closure of $X$ is denoted by $X{\downarrow}$.

A \emph{subtree} $\Delta\subseteq \MP\times k^{<\omega}$ is any $\sqsubseteq$-downwards-closed subset; thus necessarily subtrees are $\wedge$-closed as well.   The \emph{height} of $\Delta$ is $\rm{ht}(\Delta):= \{\ell(t): t\in \Delta\}\leq \omega$. Given $m< \rm{ht}(\Delta)$, we put $\Delta(m) = \{t\in \Delta: \ell(t) = m\}$. Given $x\in \MP$, we write $\Delta_x = \{t\in \Delta: t^\sf{p} = x\}$ and $\Delta_x(m) = \Delta_x\cap \Delta(m)$, and given $i< \sfU$, write $\Delta_i = \{t\in \Delta: t^\sf{u} = i\}$ and $\Delta_i(m) = \Delta_i\cap \Delta(m)$. Given $m< n< \rm{ht}(\Delta)$ and $S\subseteq \Delta(m)$, we put $\succ_\Delta(S, n) = \{t\in \Delta(n): t\sqsupset s\text{ for some }s\in S\}$; if $n = m+1$, we write $\is_\Delta(S)$ for $\succ_\Delta(S, m+1)$;  these are the \emph{immediate successors} of $S$ in $\Delta$. When $S = \{s\}$, we simply write $\succ(s, n)$ or $\is_\Delta(s)$. A \emph{splitting node} of $\Delta$ is any $t\in \Delta$ with $|\is_\Delta(t)|> 1$; write $\spnd(\Delta) = \{t\in \Delta: t\text{ a splitting node of }\Delta\}$. Given $s\in \Delta$ and $n> \ell(s)$, we put $\Left_\Delta(s, n)$ to be the $\lex$-least $t\in \Delta(n)$ with $t\sqsupset s$. Similarly, if $S\subseteq \Delta$ and $n> \ell(s)$ for each $s\in S$, we can write $\Left_\Delta(S, n) = \{\Left_\Delta(s, n): s\in S\}$. If the $\Delta$ subscript is omitted in various notation, we intend $\Delta = \MP\times k^{<\omega}$. In a similar way, one can define $\rm{Right}_\Delta(s, n)$, $\rm{Right}_\Delta(S, n)$, etc.

\begin{defin}
    \label{Def:Aged_Coding_Tree}
An \emph{aged coding tree} is a subtree $\Delta\subseteq \MP\times k^{<\omega}$ where additionally:
\begin{itemize}
    \item 
    We designate a subset $\cdnd(\Delta)\subseteq \Delta$ of \emph{coding nodes} of $\Delta$ with $|\cdnd(\Delta)\cap \Delta(m)|\leq 1$ for every $m< \rm{ht}(\Delta)$. We write $c^\Delta\colon |\cdnd(\Delta)|\to \cdnd(\Delta)$ for the function which enumerates the coding nodes in increasing height and $\ell^\Delta(n)$ for $\ell(c^\Delta(n))$.
    \item 
    We assign to each $m< \rm{ht}(\Delta)$ a class of $\cL_{|\Delta(m)|}$-structures in $P(\sort(\Delta(m)))$ which we denote by $\age_\Delta(m)$. 
\end{itemize} 
If $S\subseteq \Delta(m) = \{s_0\lex\cdots\lex s_{d-1}\}$ and $I = \{i< d: s_i\in S\}$, we can write $\age_\Delta(S)$ for $\age_\Delta(m){\cdot}I$, and if $d_0<\omega$ and $e\colon d_0\to \Delta(m)$ is any function, we define $\age_\Delta(e) = \age_\Delta(m){\cdot}\ol{e}$, where $\ol{e}$ is defined to satisfy $e(i) = s_{\ol{e}(i)}$. If $s\in \Delta$, we write $\age_\Delta(s)$ for $\age_\Delta(\{s\})$. We remark that $\age_\Delta(s)\in s^\sf{p}$. We let $\argpath_\Delta(s) = \{\age_\Delta(s|_m): m\leq \ell(s)\}$; for the aged coding trees we consider in this paper, we will always have $\argpath_\Delta(s)\in \Path_i$. \qed
\end{defin}

\begin{defin}
    \label{Def:Aged_Emb}
Given aged coding trees $\Theta, \Delta\subseteq \MP\times k^{<\omega}$, an \emph{aged embedding} $\phi\colon \Theta\to \Delta$ is an injection satisfying all of the following.
\begin{enumerate}
    \item 
    $\phi$ preserves $\sqsubseteq$, $\wedge$, $\lex$, and $\leq_\ell$ (in both the positive and negative sense). Write $\wt{\phi}\colon \rm{ht}(\Theta)\to \rm{ht}(\Delta)$ for the function satisfying $\phi[\Theta(m)]\subseteq \Delta(\wt{\phi}(m))$.
    \item 
    $\phi[\cdnd(\Theta)]\subseteq \cdnd(\Delta)$.
    \item 
    For each $x\in \MP$, $\phi[\Theta_x] \subseteq \Delta_x$.
    \item 
    For each $t\in \Theta$ and $m< \ell(t)$, we have $t(m) = \phi(t)(\wt{\phi}(m))$.
    \item 
    For each $m< \rm{ht}(\Theta)$, $\age_\Theta(m) = \age_\Delta(\phi[\Theta(m)])$.
\end{enumerate}
Write $\aemb(\Theta, \Delta)$ for the set of aged embeddings from $\Theta$ to $\Delta$. \qed
\end{defin}

\subsection{Diaries and their embeddings}
\begin{defin}
    \label{Def:DiagonalDiary}    
    A \emph{diary} is an aged coding tree $\Delta\subseteq \sf{MP}\times k^{<\omega}$ satisfying the following.
    \begin{enumerate}
        \item 
        Every non-empty level of $\Delta(m)$ contains at most one node $t$ with $|\is_\Delta(t)|\neq 1$. 
        \begin{itemize}
            \item
            If there is $t\in \Delta(m)$ with $\is_\Delta(t) = \emptyset$, we call $m$ a \emph{coding level} of $\Delta$.
            \item
            If there is $t\in \Delta(m)$ with $|\is_\Delta(t)|> 1$, then $|\is_\Delta(t)| = 2$, and we call $m$ a \emph{splitting level} of $\Delta$.
            \item
            If every $t\in \Delta(m)$ satisfies $|\is_\Delta(t)| = 1$, we call $m$ an \emph{age-change} level of $\Delta$.
        \end{itemize}
        We write $\sp(\Delta)$, $\cd(\Delta)$, and $\ac(\Delta)$ for the splitting, coding, and age-change levels of $\Delta$, respectively. We let $\cdnd(\Delta)\subseteq \Delta$ be the set of terminal nodes of $\Delta$.
        \item
        $\cdnd(\Delta)$ is $\sqsubseteq$-upwards cofinal in $\Delta$.

        \item
        Writing $\rho = \sort(\Delta(0))$, we have $\age_\Delta(0) = \cK{\cdot}\tilde{\rho} = \max(P(\rho))$.
        \item
        If $\Delta(m) = \{t_0\lex\cdots\lex t_{d-1}\}$ and $t_i$ is a splitting node, then we have $\Delta(m+1) = \{t_j^\frown 0: j< d\}\cup \{t_i^\frown 1\}$
        and $\age_\Delta(m+1) = \sp(\age_\Delta(m), i)$ (see Definition~\ref{Def:SplittingClass}).
        \item
        If $\Delta(m) = \{t_0\lex\cdots\lex t_{d-1}\}$ and $t_j$ is a coding node, then writing $\Delta(m+1) = \{u_0\lex\cdots\lex u_{d-2}\}$ and defining $\phi\colon (d-1)\to k$ via $\phi(i) = u_i(m)$, then $(\age_\Delta(m), j, \phi)$ is a controlled coding triple (Definition~\ref{Def:CodingClass}) and $\age_\Delta(m+1) = \age_\Delta(m){\cdot}\rm{Add}_{j, \phi} = \age_\Delta(m){\cdot}(d{\setminus}\{j\})$.
        \item
        If $m\in \ac(\Delta)$, then we have $\Delta(m+1) = \{t^\frown 0: t\in \Delta(m)\}$, and writing $\rho = \sort(\Delta(m))$, we have $(\age_\Delta(m), \age_\Delta(m+1))\in \con(\rho)$.
    \end{enumerate}
\end{defin}

Every level subset of a diary has a distinguished subset of \emph{critical nodes}. If $\Delta$ is a diary and $m< \rm{ht}(\Delta)$, we define $\crit_\Delta(m)\subseteq \Delta(m)$ as follows. If $m\in \sp(\Delta)$, we set $\crit_\Delta(m) = \spnd(\Delta)\cap \Delta(m)$. If $m\in \cd(\Delta)$, we set $\crit_\Delta(m) = \cdnd(\Delta)\cap \Delta(m)$. If $m\in \ac(\Delta)$, then by Proposition~\ref{Prop:Consecutive},  $(\age_\Delta(m), \age_\Delta(m+1))$ is essential on a unique $S\subseteq \Delta(m)$, and we set $\crit_\Delta(m) = S$.

Given a diary $\Delta$, the \emph{structure coded by $\Delta$}, denoted $\str(\Delta)$, is the $\cL$-structure on underlying set $\cdnd(\Delta)$ such that given $s, t\in \cdnd(\Delta)$, we set $U^{\str(\Delta)}(t) = t^\sf{u}$, and if $\ell(t)> \ell(s)$, we put $R^{\str(\Delta)}(t, s) = t(\ell(s))$. A key property of diaries (Proposition~3.4.3 of \cite{BCDHKVZ}) is that we always have $\rm{Age}(\str(\Delta))\subseteq \cK$. 

Sometimes, it is desirable to change the underlying set of the above structures. An \emph{enumerated structure} is a structure whose underlying set is a cardinal, so in this paper always a natural number or $\omega$. Given a diary $\Delta$, write $\str^\#(\Delta)$ for the enumerated structure isomorphic to $\str(\Delta)$ via the map that lists the coding nodes in order of increasing height.

\begin{defin}
    \label{Def:EmbDiary}
    Let $\Theta$ and $\Delta$ be diaries. A \emph{diary embedding} of $\Theta$ into $\Delta$ is any $\phi\in \aemb(\Theta, \Delta)$ which additionally satisfies:
    \begin{itemize}
        \item
        If $m\in \ac(\Theta)$, then $\wt{\phi}(m)\in \ac(\Delta)$. Defining $\phi^+\colon \Theta(m+1)\to \Delta(\wt{\phi}(m)+1)$ via $\phi^+(s^\frown 0) = \phi(s)^\frown 0$, we have $\age_\Theta(m+1) = \age_\Delta(\phi^+[\Theta(m+1)])$. 
    \end{itemize}
    Write $\demb(\Theta, \Delta)$ for the set of diary embeddings of $\Theta$ into $\Delta$. \qed
\end{defin}

\begin{rem}
      Given $\phi\in \demb(\Theta, \Delta)$, we can define $\phi^+$ on every level. Given any $m< \rm{ht}(\Theta)-1$, define $\phi^+\colon \Theta(m+1)\to \Delta(\wt{\phi}(m)+1)$, where given $s\in \Theta(m)$ and $i< k$ with $s^\frown i\in \Theta(m+1)$, we have $\phi^+(s^\frown i) = \phi(s)^\frown i$. This map is well defined by items $1$ through $5$. When $m\not\in\ac(\Delta)$, items $1$ through $5$ already imply that $\phi^+$ is an age map, and for $m\in \ac(\Delta)$, we explicitly demand this in item $6$.
\end{rem}

Note that if $\phi\colon \Theta\to \Delta$ is an embedding of diaries, then $\phi|_{\str(\Theta)}\colon \str(\Theta)\to \str(\Delta)$ is an embedding of $\cL$-structures, and given $S\subseteq \Theta(m)$, $\phi|_{\str(\Theta)/S}\colon \str(\Theta)/S\to \str(\Delta)/\phi[S]$ is an embedding of $\cL_d$-structures. We note the following important converse.

\begin{prop}
    \label{Prop:CodingNodesEmbedding}
    Given any induced substructure $\bX\subseteq \str(\Delta)$, there are a unique diary $\Theta$ and embedding $\phi\colon \Theta\to \Delta$ such that $\phi|_{\str(\Theta)}: \str(\Theta)\to \bX$ is an isomorphism. 
\end{prop}

Therefore, if $X\subseteq \cdnd(\Delta)$ is the underlying set of $\bX$, we \emph{define} $\Delta\|_X := \Theta$. 

\begin{proof}
    This is Proposition~3.4.8 from \cite{BCDHKVZ}.
\end{proof}

\subsection{Encoding diaries}

It will be useful to encode a diary $\Delta$ as a first-order relational structure in a fixed way. While our encoding has some redundancies, this will be useful when defining embeddings of finite initial segments of diaries (Definition~\ref{Def:FiniteEmbedding}) and \emph{level types} (Definition~\ref{Def:Critical_Level_Type}). 

\begin{notation}
    Fix a diary $\Delta$. If $s\in \Delta$ is neither splitting nor coding, write $\hat{s}^\Delta$ for its unique immediate successor. If $X\subseteq \Delta(m)$ is a level subset of such nodes, we can write $\wh{X}^\Delta = \{\hat{s}^\Delta: s\in X\} = \is_\Delta(X)$. When $\Delta$ is understood, we simply write $\hat{s}$ and $\wh{X}$.
\end{notation}

\begin{defin}
\label{Def:L_K*}
We simultaneously define the language $\cL_\cK^*$ and how to interpret a diary $\Delta$ as an $\cL_\cK^*$-structure as follows.
\begin{itemize}
    \item
    For each $\frak{p}\in \MP$, there is a unary predcate $U_\frak{p}$, where given $s\in \Delta$, $U_\frak{p}^\Delta(s)$ holds iff $s\in \Delta_\frak{p}$.
    \item
    We include the order $\lex$ on level subsets of $\Delta$.
    \item
    There is a unary predicate $\code$ so that $\code^\Delta(s)$ holds iff $s\in \code(\Delta)$.

    \item
    There is a ternary relation $\wedge$ so that $\wedge^\Delta(s, t, u)$ holds iff $s^\sf{p} = t^\sf{p}$ and $s\wedge t = u$.

    \item
    For every path sort $\rho\colon d\to \MP$ and every $\cA\in P(\rho)$, there is a $d$-ary relation $R_\cA$ so that given $S\subseteq \Delta$, $R_\cA^\Delta(S)$ holds iff $S$ is a level set, $\sort(S) = \rho$, and $\age_\Delta(S) = \cA$.

    \item 
    There is a unary predicate $\sp$ so that $\sp^\Delta(s)$ holds iff $s$ is a splitting node of $\Delta$.
    \item
    For each $i< k$, there is a unary predicate $\rm{Pass}_i$ so that $\rm{Pass}_i^\Delta(s)$ holds iff $|\is_\Delta(s)| =1$ and $\hat{s} = s^\frown i$. We call this $i$ the \emph{passing number} of $s$.
    \item
    For every path sort $\rho\colon d\to \MP$ and every $(\cA, \cB)\in \econ(\rho)$, there is a $d$-ary relation $\ac_{(\cA, \cB)}$ so that given $S\subseteq \Delta$, $\ac_{(\cA, \cB)}^\Delta(S)$ holds iff either of the following happens:
    \begin{itemize}
        \item 
        $S$ is a level set with $\ell(S)\in \ac(\Delta)$, $\sort(S) = \rho$, $\age_\Delta(S) = \cA$, and $\age_\Delta(\wh{S}) = \cB$.
        \item 
        $S = \{s\}$ with $s\in \spnd(\Delta_\frak{p})$ for some $\frak{p}\in \sfU_{non}$ and $\age_\Delta(s) = \cK{\cdot}\tilde{\iota}_{\sf{u}(\frak{p})}$
    \end{itemize}

\end{itemize}
\end{defin}

By Proposition~\ref{Prop:Consecutive}, there are only finitely many $\rho$ with $\econ(\rho)\neq \emptyset$. In particular, $\cL_\cK^*$ is finite, as there are only finitely many relations of the form $\ac_{(\cA, \cB)}$.

\begin{defin}
    \label{Def:FiniteEmbedding}
    Given diaries $\Theta$ and $\Delta$ and some $m\leq \rm{ht}(\Theta)$, an \emph{embedding} of $\Theta({<}m)$ into $\Delta$ is an embedding of the corresponding $\cL_\cK^*$-structures, where we view $\Theta({<}m)$ as an induced substructure of $\Theta$.
\end{defin}
In other words, if $\phi\in \demb(\Theta({<}m), \Delta)$ and $m-1 < \rm{ht}(\Theta)$, then $\phi$ needs to correctly anticipate some information about $\Theta(m)$. Thus if $m-1\in \sp(\Theta)$, $\phi$ must send the splitting node of $\Theta(m-1)$ to a splitting node of $\Delta$. If $m-1\in \cd(\Theta)$, then if $t\in \Theta(m-1)\cap \cdnd(\Theta)$, we have $\phi(t)\in \cdnd(\Delta)$, and if $s\in \Theta(m-1)$ is not the coding node, we have $\hat{s}(m) = \wh{\phi(s)}(\wt{\phi}(m))$. And if $m-1\in \ac(\Theta)$ and this age change is essential on $S\subseteq \Theta(m-1)$, then $\wt{\phi}(m-1)\in \ac(\Delta)$, and the age change is essential on $\phi[S]$ with the same change.

\begin{defin}
    \label{Def:Critical_Level_Type}
    A \emph{level type} is an $\cL_\cK^*$-structure $\tau$ for which there is a diary $\Theta$ and some level subset $X\subseteq \Theta$ such that $\tau$ and $X$ are isomorphic as $\cL_\cK^*$-structures. We set $\age(\tau) = \age_\Theta(X)$. If additionally $\crit_\Theta(\ell(X))\subseteq X$, we call $\tau$ a \emph{critical level type}. In this case, we write $\crit(\tau)\subseteq \tau$ for the set of critical nodes of $\tau$. Given level types $\tau_0$ and $\tau_1$, we write $\tau_0\cong^*\tau_1$ if they are isomorphic as $\cL_\cK^*$-structures. 
    
    If $\Delta$ is a diary, $m< \rm{ht}(\Delta)$, and $S\subseteq \Delta(m)$, we say that a critical level type $\tau$ is \emph{based on $S$} if the underlying set of $\tau$ is $S$ and $\age(\tau)\subseteq \age_\Delta(S)$.
\end{defin}

Given $\tau$ a critical level type based on $S\subseteq \Delta(m)$, we will often want to find some $T$ above $S$ with $\tau\cong^* T$. This is a key motivation for the strong diaries we construct in the next section.

We end this section with a brief discussion of how diaries give rise to recurrent big Ramsey structures. To do this, we use an encoding of diaries where the underlying set is some $\cL$-structure $\bA$ with $\age(\bA)\subseteq \cK$.

\begin{defin}
    \label{Def:Encode_Diaries_on_Structures}
The language $\cL_\cK^{\rm{diary}}\supseteq \cL$ adds a relational symbol $R_\Theta$ of arity $|\cdnd(\Theta)|$ for every finite diary $\Theta$ satisfying $|\cdnd(\Theta)| \leq \max(2\|\cF\|-2, 4)$. If $\bA$ is an $\cL$-structure with underlying set $A$ and $\age(\bA)\subseteq \cK$, $\bA^{\!*}$ is an $\cL_\cK^{\rm{diary}}$-expansion of $\bA$, and $\Gamma$ is a diary, we say that $\bA^{\!*}$ \emph{encodes the diary $\Gamma$} if there is an isomorphism $f\colon \bA\to \str(\Gamma)$ of $\cL$-structures with the property that for every $n< \omega$ and every $(a_0,...,a_{n-1})\in A^n$, we have that $R_\Theta^{\bA^{\!*}}(a_0,...,a_{n-1})$ holds iff $\rm{ht}(f(a_0))< \cdots < \rm{ht}(f(a_{n-1}))$ and there is (a unique) $g\in \demb(\Theta, \Gamma)$ with $g[\cdnd(\Theta)] = f[\{a_0,...,a_{n-1}\}]$. By (the proof of) Theorem~3.4.16 of \cite{BCDHKVZ}, there is at most one diary $\Gamma$ that $\bA^{\!*}$ encodes.

Observe that if $\bA$ is an enumerated structure and $\bA^{\!*}$ is an $\cL_\cK^{\rm{diary}}$-expansion of $\bA$ which encodes the diary $\Gamma$, then we have $\str^\#(\Gamma) = \bA$. \qed
\end{defin}

\begin{theorem}
    Let $\bK^*$ be an $\cL_\cK^{\rm{diary}}$-expansion of $\bK$ which encodes some diary. Then $\bK^*$ is a recurrent big Ramsey structure for $\bK$.
\end{theorem}

\begin{proof}
    This follows directly from Theorem~5.1.6 of \cite{BCDHKVZ} and the remark immediately after. 
\end{proof}

\section{Strong diaries}\label{Sec:Strong diaries}

In order to perform the combinatorics needed to obtain our infinite-dimensional Ramsey theorems, we will build and analyze particularly nice diaries that we call \emph{strong diaries}. Strong diaries mimic the desirable properties of the coding trees from \cite{Zucker_BR_Upper_Bound} (here referred to as ``ordinary coding trees") while still being diaries. In a strong diary, some levels are declared to be strong, and every strong level will indicate the start of a \emph{strong level gadget}, an interval of levels with some specified behavior. Roughly speaking, each such interval of a strong diary mimics a portion of one level of an ordinary coding tree; in particular, ``left" will be a safe direction (Proposition~\ref{Prop:Left_Safe}), ensuring that strong diaries have a rich supply of self-embeddings to work with. The example constructed in Subsection~\ref{Subsec:7.3} shows what can go wrong if we were to attempt to work with general diaries.

 Eventually, given a strong diary $\Delta$, we will let $\rm{SL}(\Delta)\subseteq \omega$ denote the set of strong levels of the strong diary $\Delta$. For now, assume $\Delta$ is a diary with $\rm{ht}(\Delta) = \omega$, and let $\rm{SL}(\Delta) = \{\mu_i: i< \omega\}\subseteq \omega$, $\mu_0< \mu_1 < \cdots$, be an infinite set containing $0$. 

We now describe the three main types of strong level gadget. First, we see that ``strong splitting gadgets" are rather uninteresting.

\begin{defin}
	\label{Def:StrongSplittingGadget}
	We call the interval $[\mu_i, \mu_{i+1})$ a \emph{strong splitting gadget} if $\mu_i\in \sp(\Delta)$ and $\mu_{i+1} = \mu_i+1$. If $t\in \Delta(m)$ is the splitting node, we call $[\mu_i, \mu_{i+1}) = \{\mu_i\}$ a \emph{strong splitting gadget at $t$}. 
\end{defin}

Next, we consider age-changes. First, we need two lemmas.

\begin{lemma}
    \label{Lem:MaxNonAC}
    Suppose $\rho\colon d\to \MP$ is a path sort and that $(\cA, \cB)\in \econ(\rho)$ is such that for some $i< d$, we have $\rho(i)\in \MP_{non}$ and $\cA_i = \max(\rho(i))$. Then $d = 1$, $i = 0$, and $\cA = \max(\rho(i))$. 
\end{lemma}

\begin{proof}
    We can find a gluing $\gamma = (\bX, \rho, \eta)$ and $x\in X$ so that, writing $\delta = (\bX\setminus \{x\}, \rho, \eta|_{d\times (I\setminus \{x\})})$, we have $\cK{\cdot}\delta = \cA$ and $\cK{\cdot}\gamma = \cB$. First we argue that for every $j< d$, we have $\eta(j, x)\neq 0$. If not, let $\bC\in \cA\setminus \cB$ have underlying set $C$ and witness that $(\cA, \cB)\in \econ(\rho)$. So every $\bD\subsetneq \bC$ satisfies $\bD\in \cB$, and $V^\bC[C] = d$. Thus $\delta(\bC)\in \cK$, $\gamma(\bD)\in \cK$ for every $\bD\subsetneq \bC$, but $\gamma(\bC)\not\in \cK$. Towards a contradiction, suppose $\eta(j, x) = 0$ for some $j< d$. Because $V^\bC[C] = d$, we can write $\gamma(\bC)$ as a free amalgam of two proper substructures, one including $x$ but not all of $\bC$, and one including all of $\bC$ but not $x$. As $\gamma(\bD)\in \cK$ for every $\bD\subsetneq \bC$, both pieces of the free amalgam are in $\cK$, so $\gamma(\bC)\in \cK$, a contradiction.
    
    Fix $i< d$ as in the lemma statement. Write $q := \eta(i, x)\neq 0$. Because $\rho(i)\in \MP_{non}$, find $\bF\in \cF$ and $a\in \bF$ with $U^\bF(a) = \sf{u}\circ \rho(i)$ and so that for all $b\in \bF\setminus \{a\}$, we have $R^\bF(b, a) = q$. Then one can find an $\cL_d$-structure $\bC$ (up to relabeling unaries, this will be $\bF\setminus \{a\}$) with $V^\bC[C] = \{i\}$, $\bC\in \cA$, but with $\gamma(\bC)|_{C\cup \{x\}}\cong \bF$. Thus $(\cA, \cB)$ is essential on $\{i\}$, and by Proposition~\ref{Prop:Consecutive}, we must have $d = \{i\}$. 
\end{proof}

\begin{lemma}
    \label{Lem:Essential_Change_Last}
    Suppose $\rho\colon d\to \MP$ is a path sort, $\cA\in P(\rho)$, $S\subseteq d$, and $\cB\in P(\rho\circ S)$ are such that $(\cA{\cdot}S, \cB)\in \econ(\rho\circ S)$. Then if $\cA = \cA_0\supseteq\cdots\supseteq \cA_n = \cA|_{(S, \cB)}$ are such that $(\cA_i, \cA_{i+1})\in \con(\rho)$ for every $i< n$, then we have $\cA_{n-1}{\cdot}S = \cA{\cdot}S$. 
\end{lemma}

\begin{proof}
    If not, then for some $i< n$ we have $\cA_i{\cdot}S\subsetneq \cA{\cdot}S$. Thus $\cA_i{\cdot}S = \cB$, and $\cA_i\subseteq \cA\cap \cB{\cdot}f$, a contradiction. 
\end{proof}

\begin{notation}
\label{Notation:Restrict_Class}
Elaborating on Notation~\ref{Notation:Manipulate_Classes}, if $m< \omega$, $U\subseteq S\subseteq\Delta(m)$, and $\cB\in P(\sort(U))$, then if $S$ is understood from context, we can abuse notation and write $\age_\Delta(S)|_{(U, \cB)}$. 
\end{notation}

\begin{defin}
	\label{Def:StrongACGadget}
	Suppose $i< \omega$, $S\subseteq \Delta(\mu_i)$, and $(\age_\Delta(S), \cB)\in \econ(\sort(S))$. The definition that $[\mu_i, \mu_{i+1})$ is a \emph{strong age-change gadget at $S$ to $\cB$} depends on some cases.
\begin{enumerate}
    \item 
    If there are $t\in S$ and $\frak{p}\in \MP_{non}$ with both
   $t^\sf{p} = \frak{p}$ and $\age_\Delta(t) = \max(\frak{p})$, then by Lemma~\ref{Lem:MaxNonAC}, we must have $S = \{t\}$. In this case, $[\mu_i, \mu_{i+1})$ is a strong age-change gadget at $\{t\}$ to $\cB$ iff it is a strong splitting gadget at $t$. Note that we must have $\cB = \max(\frak{p}')$.
    \item
   Otherwise, we call $[\mu_i, \mu_{i+1})$ a strong age-change gadget at $(S, \cB)$ if the following all hold. 
      \begin{itemize}
        \item 
        $\mu_i+|S| < \mu_{i+1}$, and for each $s\in S$, there is a splitting node $t\sqsupseteq s$ with $\ell(t)< \mu_i+|S|$.
        \item
        Set $U = \rm{Right}(S, \mu_i+|S|)$. Because there is no $t\in S$ and $\frak{p}\in \MP_{non}$ with $t^\sf{p}$ and $\age_\Delta(t) = \max(\frak{p})$, we have that $\age_\Delta(U) = \age_\Delta(S)$. The remaining levels of the gadget are age-change levels from $\age_\Delta(\mu_i+|S|)$ to $\age_\Delta(\mu_i+|S|)\big|_{(U, \cB)} = \age_\Delta(\mu_{i+1})$.
    \end{itemize}
\end{enumerate}
\end{defin}

The idea is that upon extending the members of $S$ and following all splitting \emph{to the right}, we obtain the desired essential age change to $\c{B}$, and by Lemma~\ref{Lem:Essential_Change_Last}, this happens at level $\mu_{i+1}-1$. This allows left to be a safe direction. Note that although the two cases are defined differently, they function in roughly the same way, in the sense that left is age-preserving and right is an age change.

Strong coding gadgets (Definition~\ref{Def:StrongCodingGadget}) have the most involved definition of the three types of gadgets. 

\begin{defin}
	\label{Def:ValidPassingNumbers}
    Given a path sort $\rho\colon 2\to \MP$, $\cA\in P(\rho)$, $j< 2$, and $q< k$, we say that $q$ is a \emph{valid passing number} for $(\rho, \cA, j)$ if $\cA{\cdot}\rm{Add}_{j, \iota_q}\in \rho(1-j)$. In particular, $0$ is always a valid passing number for $(\rho, \cA, j)$. \qed
\end{defin}

\begin{lemma}
    \label{Lem:Valid_Passing}
    Suppose $\rho$, $\cA$, and $j< 2$ are as in Definition~\ref{Def:ValidPassingNumbers} and that $q< k$ is a valid passing number for $(\rho, \cA, j)$. Suppose $\rho(1-j) = \frak{p}\in \MP_{non}$ and $\cA^{1-j} = \max(\frak{p})$. Then $q$ is a valid passing number for $(\rho, \cA|_{(1-j, \max(\frak{p}'))}, j)$. 
\end{lemma}

\begin{proof}
    Without loss of generality suppose $j = 1$ and $q\neq 0$. Then $\cA{\cdot}\rm{Add}_{1, \iota_q}\in \frak{p}$, and in particular $\cA{\cdot}\rm{Add}_{1, \iota_q}\subseteq \cA^0 = \max(\frak{p})$. Since $\cA\in P(\rho)$ as witnessed by some gluing, we must have $\cA\subseteq \cK{\cdot}\wt{\rho}$. By the definition of $\cA{\cdot}\rm{Add}_{1, \iota_q}$ and since $\frak{p}\in \MP_{non}$, we must have $\cA{\cdot}\rm{Add}_{1, \iota_q}\subseteq \max(\frak{p}')$. Hence we have $(\cA|_{(0, \max(\frak{p}'))}){\cdot}\rm{Add}_{1, \iota_q} = \cA{\cdot}\rm{Add}_{1, \iota_q}$.
\end{proof}

\begin{defin}
	\label{Def:StrongCodingGadget}
	Fix $i< \omega$ and $t\in \Delta(\mu_i)$. We call $[\mu_i, \mu_{i+1})$ a \emph{strong coding gadget at $t$} if the following all happen.
    \begin{itemize}
        \item 
        $t$ is a splitting node. Write $u = t^\frown 1$.
        \item
        For every $s\in \Delta(\mu_i+1)\setminus \{u\}$, let $I_s$ denote the set of valid passing numbers for $(s, u)$. The next several levels of the gadget are splitting levels. For each $s\in \Delta(\mu_i+1)\setminus\{u\}$, $|I_s|-1$ of these splitting levels will have a splitting node extending $s$. If $m>\mu_i$ is the level we are at after these splitting levels, note that for $s\in \Delta(\mu_i+1)\setminus \{u\}$, we have $|\succ_\Delta(s, m)| = |I_s|$, and we have $\succ_\Delta(u, m) = \{v\}$ a singleton.  
        
        Write $\phi\colon \Delta(m)\setminus\{v\}\to k$ for the function such that for each $s\in \Delta(\mu_i+1)\setminus\{u\}$, $\phi|_{\succ_\Delta(s, m)}$ is the $\lex$-increasing bijection onto $I_s$. 
        \item
        Now notice that by the definition of valid passing number, we have that $\age_\Delta(m)(v, \phi)$ contains all unaries. Levels from $m$ to $\mu_{i+1}-1$ are age-change levels, with $\age_\Delta(\mu_{i+1}-1) = \la \age_\Delta(m), v, \phi\ra$.
        \item
        $\mu_{i+1}-1\in \cd(\Delta)$, and writing $w:= \Left(v, \mu_{i+1}-1)$, the $\cL_\cK^*$-structure on $\Delta(\mu_{i+1}-1)$ is described by the controlled coding triple $(\age_\Delta(\mu_{i+1}-1), w, \phi\circ \pi_m)$.
    \end{itemize}
    Given $s\in \Delta(\mu_i)$ and $q\in I_{s^\frown 0}$, there is a unique $u\in \succ_\Delta(s, \mu_{i+1})$ with $u(\mu_{i+1}) = q$; we denote this $u$ by $s{*}q$. Note that $s{*}0 = \Left_\Delta(s, \mu_{i+1})$.
 \end{defin}

\begin{figure}
\begin{center}
\begin{tikzpicture}[scale=.18]

\node at (0,0) {};

\foreach \x in {-15,0,17,30}{
\foreach \y in {0}{
\node at (\x,\y) {};
}}

\draw[dotted] (-17,0) -- (31,0);

\node[circle, fill=black,inner sep=0pt, minimum size=2pt] at (-15,0) {};
\node[circle, fill=black,inner sep=0pt, minimum size=2pt] at (0,0) {};
\node[circle, fill=black,inner sep=0pt, minimum size=2pt] at (17,0) {};
\node[circle, fill=black,inner sep=0pt, minimum size=2pt] at (30,0) {};

\node at (-15, -1.5) {$r$};
\node at (0,-1.5) {$s$};
\node at (17,-1.5) {$t$};
\node at (30,-1.5) {$u$};
\node at (35,-0.3) {$\Delta(\mu_i)$};

\draw (0,0) to [out= -30, in=-150] (17,0);

\draw[dotted] (-19,4) -- (29,4);
\node at (35,3.7) {$\Delta(\mu_i +1)$};

\node[circle, fill=black,inner sep=0pt, minimum size=2pt] at (-17,4) {};
\node[circle, fill=black,inner sep=0pt, minimum size=2pt] at (-2,4) {};
\node[circle, fill=black,inner sep=0pt, minimum size=2pt] at (15,4) {};
\node[circle, fill=black,inner sep=0pt, minimum size=2pt] at (19,4) {};
\node at (20,4) {$u$};
\node[circle, fill=black,inner sep=0pt, minimum size=2pt] at (28,4) {};

\draw[thick] (-15,0)--(-17,4);
\draw[thick] (0,0)--(-2,4);
\draw[thick] (15,4)--(17,0)--(19,4);
\draw[thick] (30,0)--(28,4);

\draw[dotted] (-21,8) -- (28,8);

\node[circle, fill=black,inner sep=0pt, minimum size=2pt] at (-19,8) {};
\node[circle, fill=black,inner sep=0pt, minimum size=2pt] at (-4,8) {};
\node[circle, fill=black,inner sep=0pt, minimum size=2pt] at (0,8) {};
\node[circle, fill=black,inner sep=0pt, minimum size=2pt] at (13,8) {};
\node[circle, fill=black,inner sep=0pt, minimum size=2pt] at (17,8) {};
\node[circle, fill=black,inner sep=0pt, minimum size=2pt] at (26,8) {};

\draw[thick] (-19,8)--(-17,4);
\draw[thick] (-4,8)--(-2,4)--(0,8);
\draw[thick] (15,4)--(13,8);
\draw[thick]  (17,8)--(19,4);
\draw[thick] (26,8)--(28,4);

\draw[dotted] (-23,12) -- (28,12);

\node[circle, fill=black,inner sep=0pt, minimum size=2pt] at (-21,12) {};
\node[circle, fill=black,inner sep=0pt, minimum size=2pt] at (-17,12) {};
\node[circle, fill=black,inner sep=0pt, minimum size=2pt] at (-6,12) {};
\node[circle, fill=black,inner sep=0pt, minimum size=2pt] at (-2,12) {};
\node[circle, fill=black,inner sep=0pt, minimum size=2pt] at (11,12) {};
\node[circle, fill=black,inner sep=0pt, minimum size=2pt] at (15,12) {};
\node[circle, fill=black,inner sep=0pt, minimum size=2pt] at (24,12) {};

\draw[thick] (-21,12)--(-19,8)--(-17,12);
\draw[thick] (-4,8)--(-6,12);
\draw[thick]  (0,8)--(-2,12);
\draw[thick] (11,12)--(13,8);
\draw[thick]  (17,8)--(15,12);
\draw[thick] (26,8)--(24,12);

\draw[dotted] (-25,16) -- (28,16);

\node[circle, fill=black,inner sep=0pt, minimum size=2pt] at (-23,16) {};
\node[circle, fill=black,inner sep=0pt, minimum size=2pt] at (-15,16) {};
\node[circle, fill=black,inner sep=0pt, minimum size=2pt] at (-8,16) {};
\node[circle, fill=black,inner sep=0pt, minimum size=2pt] at (0,16) {};
\node[circle, fill=black,inner sep=0pt, minimum size=2pt] at (9,16) {};
\node[circle, fill=black,inner sep=0pt, minimum size=2pt] at (13,16) {};
\node[circle, fill=black,inner sep=0pt, minimum size=2pt] at (22,16) {};

\draw[thick] (-21,12)--(-23,16);
\draw[thick] (-17,12)--(-15,16);
\draw[thick] (-8,16)--(-6,12);
\draw[thick]  (0,16)--(-2,12);
\draw[thick] (11,12)--(9,16);
\draw[thick]  (13,16)--(15,12);
\draw[thick] (22,16)--(24,12);

\draw[dotted] (-27,20) -- (28,20);
\node at (35,19.7) {$\Delta(\mu_{i+1}-1)$};

\node[circle, fill=black,inner sep=0pt, minimum size=2pt] at (-25,20) {};
\node[circle, fill=black,inner sep=0pt, minimum size=2pt] at (-17,20) {};
\node[circle, fill=black,inner sep=0pt, minimum size=2pt] at (-6,20) {};
\node[circle, fill=black,inner sep=0pt, minimum size=2pt] at (-2,20) {};
\node[circle, fill=black,inner sep=0pt, minimum size=2pt] at (7,20) {};
\node[circle, fill=black,inner sep=0pt, minimum size=5pt] at (15,20) {};
\node[circle, fill=black,inner sep=0pt, minimum size=2pt] at (20,20) {};

\draw[thick] (-25,20)--(-23,16);
\draw[thick] (-17,20)--(-15,16);
\draw[thick] (-6,20)--(-8,16);
\draw[thick]  (-2,20)-- (0,16);
\draw[thick] (7,20)--(9,16);
\draw[thick]  (15,20)--(13,16);
\draw[thick] (20,20)--(22,16);

\draw[dotted] (-29,24) -- (28,24);
\node at (35,23.7) {$\Delta(\mu_{i+1}-1)$};

\node[circle, fill=black,inner sep=0pt, minimum size=2pt] at (-27,24) {};
\node[circle, fill=black,inner sep=0pt, minimum size=2pt] at (-15,24) {};
\node[circle, fill=black,inner sep=0pt, minimum size=2pt] at (-8,24) {};
\node[circle, fill=black,inner sep=0pt, minimum size=2pt] at (0,24) {};
\node[circle, fill=black,inner sep=0pt, minimum size=2pt] at (5,24) {};
\node[circle, fill=black,inner sep=0pt, minimum size=2pt] at (18,24) {};

\draw[thick] (-25,20)--(-27,24);
\draw[thick] (-17,20)--(-15,24);
\draw[thick] (-6,20)--(-8,24);
\draw[thick]  (-2,20)-- (0,24);
\draw[thick] (7,20)--(5,24);
\draw[thick] (18,24)--(20,20);

\end{tikzpicture}
\caption{Strong coding gadget at $t$ in the class of triangle-free graphs. Here $\age_\Delta(\{r, s, t, u\})$ is the class of $\cL_4$-structures (graphs with vertex set partitioned into $4$ pieces $V_0$ through $V_3$) which are triangle free, with each of $V_1$, $V_2$, and $V_3$ an anti-clique, with no edges allowed between $V_1$ and $V_3$, and with no edges allowed between $V_2$ and $V_3$. The curved line between $s$ and $t$ represents that edges are allowed between $V_1$ and $V_2$. We do not draw horizontal lines connecting to $r$ since $\age_\Delta(\{r\})$ is the class of all triangle free graphs (with all nodes in unary $V_0$), which automatically implies that edges are allowed between $V_0$ and $V_i$ for $i\in \{1, 2, 3\}$.}
\end{center}
\end{figure}
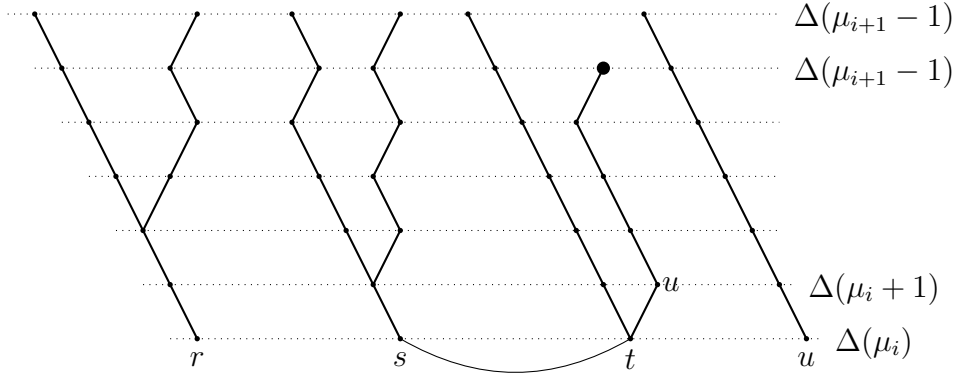

\begin{defin}
	\label{Def:StrongDiary}
	A \emph{strong diary} is any diary $\Delta$ built in the following fashion. We set $\Delta(0) = \MP\times \{\emptyset\}$; writing $\sort(\Delta(0)) = \rho$, we have $\age_\Delta(0) = \c{K}{\cdot}\tilde{\rho}$. We set $\mu_0 = 0$. If for some $n< \omega$, $\mu_n$ has been determined and $\Delta$ has been constructed up to level $\mu_n$, we build the interval from $[\mu_n, \mu_{n+1})$ by adding a strong splitting, age-change, or coding gadget. We ensure that the following all hold. 
	\begin{itemize}
		\item 
		For every $t\in \Delta$, there is some $n< \omega$ with $\mu_n > \ell(t)$ so that $[\mu_n, \mu_{n+1})$ is a strong splitting gadget at $\Left(t, \mu_n)$.
		\item
		Same as above for strong coding gadgets.
		\item
		Suppose for some $\mu_\ell< \omega$, $S\subseteq \Delta(\mu_\ell)$ and $\cB\in P(\sort(S))$ we have $(\age_\Delta(S), \cB)\in \econ(\sort(S))$. Then there is $n > \ell$ so that $[\mu_n, \mu_{n+1})$ is a strong age-change gadget at $\Left(S, \mu_n)$ to $\cB$.
	\end{itemize}
	It is routine to inductively construct strong diaries. Whenever $\Delta$ is a strong diary, we write $\rm{SL}(\Delta) = \{\mu_n: n< \omega\}$. We note that $[\mu_n, \mu_{n+1})$ is a strong X gadget, where X is one of coding, age change, or splitting, iff $\mu_{n+1}-1$ is a level of type X.  Given $\ell< \omega$, let $\pi_s(\ell)$ be the largest member of $\rm{SL}(\Delta)$ with $\pi_s(\ell)\leq \ell$, and given $t\in \Delta$, let $\pi_s(t)$ denote the restriction of $t$ to $\pi_s(\ell(t))$. The map $\pi_s$ plays a very similar role to the notion of ``splitting predecessor'' from \cite{Dobrinen_SDAP, Dobrinen_3Free, Dobrinen_Henson}. 
\end{defin}

Using strong diaries, we obtain a simple proof that there exist diaries coding $\flim(\cK)$. In fact, the proof of Proposition~3.4.11 from \cite{BCDHKVZ} is inspired from the proof of the following proposition. 

\begin{notation}
    \label{Notation:ValidPassingNumbers}
    Elaborating on Definition~\ref{Def:ValidPassingNumbers}, suppose $\Delta\subseteq \MP\times k^{<\omega}$ is an aged coding tree, $m< \rm{ht}(\Delta)$,  and $s\neq t\in \Delta(m)$, writing $\{s, t\} = \{s_0, s_1\}$ with $s_0\lex s_1$ and letting $j< 2$ satisfy $t = s_j$. Then given $q< k$, we write $\age_\Delta(s; t, q):= \age_\Delta(\{s, t\}){\cdot}\rm{Add}_{j, \iota_q}$, and we say that $q< k$ is a \emph{valid passing number} for $(\Delta, s, t)$ if  $q$ is a valid passing number for $(\sort(\{s, t\}), \age_\Delta(\{s, t\}), j)$, i.e.\ if $\age_\Delta(s; t, q)\in s^\sf{p}$.
\end{notation}

\begin{prop}
    \label{Prop:Strong_Diaries}
    If $\Delta$ is a strong diary, then $\str(\Delta)\cong \flim(\cK)$. 
\end{prop}

\begin{proof}
    We show that $\str(\Delta)$ satisfies the extension property for one-element extensions. Suppose $n < \omega$ and that $c_0,...,c_{n-1}\in \cdnd(\Delta)$ are the first $n$ coding nodes in $\Delta$. Let $\bX$ be an enumerated structure such that the map $x\to c_x$ is an isomorphism from $\bX$ to $\str(\Delta)|_{\{c_0,...,c_{n-1}\}}$. Suppose $j< \sfU$ and $\gamma_n:= (\bX, \iota_j, \eta)$ is a rank $1$ gluing with $\cK{\cdot}\gamma_n\neq\emptyset$, thus representing an instance of an extension problem. Given $i< n$, set $\gamma_i = (\bX|_i, \rho, \eta|_{1\times i})$. Thus $\max(P_j) = \cK{\cdot}\gamma_0\supseteq\cdots\supseteq \cK{\cdot}\gamma_n \in P_j$. Let $\frak{p}\in \MP_j$ satisfy $\cK{\cdot}\gamma_i\in \frak{p}$ for every $i< n$. 

    Set $s_0 = (\frak{p}, \emptyset)\in \Delta$. Suppose $i< n$ and that $s_i$ has been determined with $\age_\Delta(s_i) = \cK{\cdot}\gamma_i$ and so that for $i> 0$ we have $s_i\in \Delta(\ell(c_{i-1})+1)$. Write $\mu = \pi_s(\ell(c_i))\in \SL(\Delta)$, and set $t = c_i|_\mu$ and $u = \Left_\Delta(t_i, \mu)$. Then $q_i:= \eta(1, i)< k$ is a valid passing number for $(\Delta, u, t)$ since $\age_\Delta(u; t, q_i) = \cK{\cdot}\gamma_{i+1}\in \frak{p}$. We then set $s_{i+1} = u{*}q_i$. 

    Upon defining $s_n$, any $c\in \cdnd(\Delta)\cap \succ_\Delta(s_n)$ will witness the desired instance of the extension property.
\end{proof}

We \textbf{fix once and for all} a strong diary $\Delta$ coding $\flim(\cK)$. The key advantage to working with strong diaries is that when restricting to strong levels, they are almost as flexible to work with as the coding trees from \cite{Zucker_BR_Upper_Bound}. The next few propositions give us the key properties we will need in the proofs of the main theorems.

\begin{notation}
    \label{Notation:End_Age_Exts}
     If $S\subseteq \Delta(m)$, then an \emph{end-extension} of $S$ is a set $U\subseteq \Delta(n)$ for some $n\geq m$ with the property that each $s\in S$ extends to exactly one $u\in U$. Note that this implies $\age_\Delta(U)\subseteq \age_\Delta(S)$. If we have equality, then we call $U$ an \emph{age-extension} of $S$. We write $S\sqsubseteq_{end} U$ if $U$ end-extends $S$ and $S\sqsubseteq_{age} U$ if $U$ age-extends $S$.  
\end{notation}

\begin{prop}
    \label{Prop:Left_Safe}
    If $\mu< \nu \in \SL(\Delta)$, then $\Left_\Delta(\Delta(\mu), \nu)\sqsupseteq_{age} \Delta(\mu)$. More generally, if $S_0\subseteq \Delta(\mu)$, $T_0\subseteq \Delta(\nu)$ with $T_0\sqsupseteq_{age} S_0$, and we set $T = T_0\cup \Left_\Delta(\Delta(\mu)\setminus S_0, \nu)$, then $T\sqsupseteq_{age} \Delta(\mu)$.
\end{prop}

\begin{proof}
    By considering $S_0 = \emptyset$, it suffices to prove the second claim. We may assume for some $i< \omega$ that $\mu = \mu_i$ and $\nu = \mu_{i+1}$. There are now cases depending on what type of strong gadget $[\mu_i, \mu_{i+1})$ is.
    \begin{itemize}
        \item 
        Splitting: This case is clear. 
        \item
        Age-change: We can assume we are in case $2$ of Definition~\ref{Def:StrongACGadget}, and we adopt notation as in that definition. So suppose $S\subseteq \Delta(\mu_i)$ and that $[\mu_i, \mu_{i+1})$ is a strong age-change gadget at $S$ to $\cB$. If $S\not\subseteq S_0$, then the result is clear. If $S\subseteq S_0$, then we cannot have $T_0 = \rm{Right}(S_0, \mu_{i+1})$. In particular, $\succ_\Delta(U, \mu_{i+1})\not\subseteq T$, and the result follows. 
        \item
        Coding: We adopt notation as in Definition~\ref{Def:StrongCodingGadget}. If $\Delta(m) = \{s_0\lex\cdots\lex s_{d-1}\}$ and $v = s_j$, then if $S\subseteq d\setminus \{j\}$ is minimal under inclusion with $\age_\Delta(m){\cdot}S \supsetneq  \la\age_\Delta(m), v, \phi\ra{\cdot}S$, then we must have $\phi(s_\alpha)\neq 0$ for every $\alpha\in S$. However, if  $s_\beta = \Left_\Delta(s, m)$ for some $s\in \Delta(\mu_i)$, then we have $\phi(\beta) = 0$. The result follows.  \qedhere
    \end{itemize}
\end{proof}

\begin{prop}
    \label{Prop:Find_Level_Type}
    Fix $\mu\in \SL(\Delta)$, $S\subseteq \Delta(\mu)$, and $\tau$ a critical level type based on $S$.
    \begin{enumerate}
        \item 
        There is $T\sqsupseteq_{end} S$ with $\ell(T)+1\in \SL(\Delta)$ and $T\cong^* \tau$.
        \item
        If $\age(\tau) = \age_\Delta(S)$, $S_0\subseteq S$ and $\tau_0= \tau|_{S_0}$, then if $\tau_0$ is critical and $T_0\sqsupseteq_{end} S_0$ is such that $\ell(T_0)+1\in \SL(\Delta)$ and $T_0\cong^* \tau_0$, then there is  $T\sqsupseteq_{age} S$ with $T\supseteq T_0$ and $T\cong^* \tau$.
    \end{enumerate}
\end{prop}

\begin{proof}
    By passing through some age-change gadgets, we can find $S'\sqsupseteq_{end} S$ with $\age(\tau) = \age_\Delta(S')$. So we will simply assume that $\age(\tau) = \age_\Delta(S)$. By Definition~\ref{Def:StrongDiary}, the first part of the result is clear if $\tau$ is a critical level type describing a splitting or age-change level, and the second part follows by simply moving nodes in $S\setminus S_0$ up leftmost. 
    
    So we assume that $\tau$ describes a coding level and use notation as in Definition~\ref{Def:StrongCodingGadget}. For the first result, let $x\in S$ be the node that $\tau$ thinks is a coding node. Find $\mu_i\in \SL(\Delta)$ with $\mu_i> \mu$ so that $[\mu_i, \mu_{i+1})$ is a strong coding gadget at $t:= \Left(x, \mu_i)$. If $s\in S\setminus \{x\}$ and $\tau$ thinks that $s$ has passing number $q< k$, then $q$ will be a valid passing number for $(\Left_\Delta(s, \mu_i+1), u)$. For the second result, suppose $\ell(T_0)+1 = \mu_{i+1}\in \SL(\Delta)$. Set $U = \pi_{\mu_i}[T_0]\cup \Left_\Delta(S\setminus S_0, \mu_i)$. By Proposition~\ref{Prop:Left_Safe}, $U\sqsupseteq_{age} S$. It follows that $U$ has all the same valid passing numbers as the corresponding members of $S$ for the given position of the coding node, and we can define $T\supseteq T_0$ with $T\sqsupseteq_{age} U$ according to the passing numbers we need. The reason that we can achieve these passing numbers without passing through an unwanted age change is that $\tau$ is assumed to be a critical level type based on $S$, so in particular already corresponds to a controlled coding triple.   
\end{proof}

\section{Ramsey spaces from a weaker axiom}\label{Sec:Ramsey spaces}

Ramsey spaces  generalize topological Ramsey spaces to 
pairs of spaces $(\mathcal{R},\mathcal{S})$   where  the field of $\mathcal{S}$-Ramsey subsets of $\mathcal{R}$ is closed under the Souslin operation and coincides with the field of $\mathcal{S}$-Baire subsets of $\mathcal{R}$. 
Building on prior work of Carlson and Simpson,
Todorcevic 
distilled four axioms which guarantee a Ramsey space.
The main result of this section shows  that the conclusion of his  Abstract Ramsey Theorem  still holds when 
axiom \axiom\
is replaced by the weaker condition of   an \axiom-ideal.  For further background, the reader is referred to Sections 4.2 and 4.3 of \cite{stevo_book}.

\subsection{Ramsey spaces and Todorcevic's Four Axioms}

This subsection recalls notions from Section 4.2 of \cite{stevo_book};
some of the notation is modified here.
We work with  structures of the form
$(\mathcal{R},\mathcal{S},\le_{\mathcal{R}},\le)$ with the following properties.
The sets $\mathcal{R}$ and $\mathcal{S}$ are families of infinite sequences of objects.
We will use
$\leq\ \subseteq \cS\times \cS$ to denote a pre-order on $\cS$ and use
$\leq_\cR\ \subseteq \cR\times \cS$
to refer to a relation with the property that for any
$X\leq Y$ in $\mathcal{S}$ and any $A\in\mathcal{R}$ with $A\leq_\cR  X$, we have $A\leq_\cR  Y$.

Given $n< \omega$ and $M\in \cR\cup \cS$, we will simply write $M|_n$ for the restriction operation giving the $n$-th finite approximation to $M$.
Write $\cal{AR}_n = \{M|_n: M\in \cR\}$, and similarly for $\cal{AS}_n$. Write $\cal{AR} = \bigcup_n \cal{AR}_n$, and similarly for $\cal{AS}$.
The {\em basic sets} are defined as follows:
For $a\in\mathcal{AR}_n$, $s\in\mathcal{AS}_n$, and $Y\in\mathcal{S}$,
set
\begin{align*}
[a,Y]&=\{A\in\mathcal{R}: A\le_\cR Y\mathrm{\ and\ } A|_n=a\}\cr
    [s,Y]&=\{X\in\mathcal{S}: X\le Y\mathrm{\ and\ }  X|_n=s\}\cr
    [a] &= \{A\in \cR: A|_{n} = a\}\cr
    [s] &= \{X\in \cS: X|_n = s\}
\end{align*}
We set $\cal{AR}^Y = \{b\in \cal{AR}: [b, Y]\neq \emptyset\}$, and similarly for $\cal{AR}^Y_n$. With this notation in place, we can now state the four axioms of Todorcevic.

\begin{enumerate}
    \item[\bf A.1] (Sequencing)
    For any choice of $\mathcal{P}\in\{\cR,\cS\}$,
    \begin{enumerate}
        \item[(1)] 
        $M|_0=N|_0$ for all $M,N\in\mathcal{P}$,
         \item[(2)] 
         $M\ne N$ implies that $M|_n\ne N|_n$ for some $n$,
          \item[(3)] 
          $M|_m=N|_n$ implies $m=n$ and $M|_k=N|_k$ for all $k\le m$.
    \end{enumerate}
\end{enumerate}

From this axiom, we can identify an object $M$ from $\mathcal{R}\cup\mathcal{S}$ with the infinite sequence 
$(M|_k)_{k<\omega}$.
Similarly, $x\in\mathcal{AR}\cup\mathcal{AS}$
is identified with the sequence $(x|_k)_{k<n}$, where $n$ is the unique integer such that $x=M|_n$ for some $M\in\mathcal{R}\cup\mathcal{S}$.
The notion of end-extension is defined as follows:
For $\mathcal{P}\in\{\mathcal{R},\mathcal{S}\}$ and $x,y\in\mathcal{AP}$, 
write $x\sqsubseteq y$  if and only if there  are  $M\in\mathcal{P}$ and  $m\le n$ such that $x=M|_m$ and $y=M|_n$.
We write $x\sqsubset y$ when $x\sqsubseteq y$ and $x\ne y$.

\begin{enumerate}
    \item[{\bf A.2}] (Finitization)
    There is a   pre-order $\le_{\mathrm{fin}}\ \subseteq \mathcal{AS}\times\mathcal{AS}$
    and  a relation $\le^{\mathcal{R}}_{\mathrm{fin}}\, \subseteq \mathcal{AR}\times\mathcal{AR}$
    which are finitizations of the relations $\le$ and $\le_{\cR}$, meaning that the following hold:
    \begin{enumerate}
        \item[(1)]
        $\{a: a \le^{\mathcal{R}}_{\mathrm{fin}} x\}$ and $\{y:y\le_{\mathrm{fin}} x\}$ are finite for all $x\in\cS$,
        
        \item[(2)]
        $X\le Y$ iff $\forall m\ \exists n$ such that 
        $X|_m\le_{\mathrm{fin}} Y|_n$,
        
        \item[(3)]
        $A\le_{\cR} X$ iff $\forall m\ \exists n$ such that $A|_m \le^{\mathcal{R}}_{\mathrm{fin}} X|_n$,
        
        \item[(4)]
        $\forall a\in\mathcal{AR}$\  $\forall x,y\in\mathcal{AS}$ [$a\le^{\mathcal{R}}_{\mathrm{fin}} x\le_{\mathrm{fin}} y \, \rightarrow\,  a\le^{\mathcal{R}}_{\mathrm{fin}} y$],
        
        \item[(5)]
        $\forall a,b\in\mathcal{AR}\ \forall x\in\mathcal{AS}$
        [$a\sqsubseteq b$ and $b\le^{\mathcal{R}}_{\mathrm{fin}} x \, \rightarrow\, \exists y\sqsubseteq x\ a \le^{\mathcal{R}}_{\mathrm{fin}} y$].
    \end{enumerate}
\end{enumerate}

For $a\in\mathcal{AR}$ and $Y\in\cS$, depth$_Y(a)$ denotes  the minimum $k$ such that $a \le^{\mathcal{R}}_{\mathrm{fin}} Y|_k$ if such a $k$ exists;
otherwise $\rm{depth}_Y(a)=\infty$.
If $d< \omega$, then $[d,Y]$ denotes $[Y|_d,Y]$.

\begin{enumerate}
    \item[{\bf A.3}] (Amalgamation)
    \begin{enumerate}
        \item[(1)]
        $ \forall Y\in\cS\ \forall a\in\cal{AR}^Y\ \forall X\in [\rm{depth}_Y(a), Y]\ ([a,X]\ne\emptyset).$
         \item[(2)]
 $\forall X\leq Y\in\cS\ \forall a\in \cal{AR}^X\ \exists Y'\in 
         [\rm{depth}_Y(a), Y]\ ([a,Y']\subseteq [a,X]).
        $
    \end{enumerate}
\end{enumerate}

\begin{enumerate}
    \item[{\bf A.4}] (Pigeonhole)
    $\forall Y\in \cS\ \forall k< \omega\ \forall a\in \cal{AR}^Y_k\ \forall \cO\subseteq \cal{AR}_{k+1}\ \exists X\in [\rm{depth}_Y(a), Y]$
    $$\{b\in \cal{AR}^X_{k+1}: a\sqsubseteq b\}\subseteq \cO \text{ or } \{b\in \cal{AR}^X_{k+1}: a\sqsubseteq b\}\cap \cO = \emptyset.$$
\end{enumerate}

We isolate a subset of the above axioms in the following definition.

\begin{defin}
    \label{Def:ContainmentSpace}
    $(\cR, \cS, \leq, \leq_\cR)$ is called a \emph{containment space} if it satisfies axioms {\bf A.1}, {\bf A.2}, and {\bf A.3(1)}. If $\cR = \cS$ and $\leq\, =\, \leq_\cR$, we call it a \emph{topological containment space}. 
\end{defin}

The Abstract Ramsey Theorem (Theorem~\ref{thm.RS}) concerns the following notions of definability for subsets of $\cR$.

\begin{defin}\label{Def:SRamseyBaire}
Given $\frakX, \frakY\subseteq \cR$, we say that $\frakY$ is \emph{$\frakX$-invariant} if either $\frakY\subseteq \frakX$ or $\frakY\cap \frakX = \emptyset$. 

A set $\mathfrak{X}\subseteq\cR$ is {\em $\cS$-Ramsey} if for every nonempty basic set $[a,Y]$ there is $X\in[\mathrm{depth}_Y(a),Y]$ such that $[a,X]$ is $\frakX$-invariant.
If for every $[a,Y]\ne\emptyset$ we can find $X\in[\mathrm{depth}_Y(a),Y]$ such that $[a,X]\cap\mathfrak{X}=\emptyset$, then we call $\mathfrak{X}$ an {\em $\cS$-Ramsey null set of $\cR$}.

A set $\mathfrak{X}\subseteq\mathcal{R}$ is {\em $\cS$-Baire} if 
for every nonempty basic set $[a,Y]$ there exists $X\le Y$ and $a\sqsubseteq b\in\cal{AR}^X$ such that  $[b,X]$ is $\frakX$-invariant.
If for every $[a,Y]\ne\emptyset$ we can find $X\le Y$ and $a\sqsubseteq b\in \cal{AR}^X$ with $[b,X]\cap \mathfrak{X} = \emptyset$, then $\mathfrak{X}$  is said to be  {\em $\cS$-meager}.
\end{defin}

It is clear that $\cS$-Ramsey implies $\cS$-Baire, and that $\cS$-Ramsey null implies $\cS$ meager. 
The following  theorem of Todorcevic
shows that axioms {\bf A.1}--{\bf A.4} suffice to conclude the converse.

\begin{theorem}[Abstract Ramsey Theorem, Todorcevic \cite{stevo_book}]\label{thm.RS}
If  $(\mathcal{R},\mathcal{S},\le,\le_{\cR})$ is a structure satisfying axioms {\bf A.1} to {\bf A.4} and that $\mathcal{S} = \varprojlim \cal{AS}_n$,
then the field of $\mathcal{S}$-Ramsey subsets of $\mathcal{R}$ is closed under the Souslin operation and it coincides with the field of $\mathcal{S}$-Baire subsets of $\mathcal{R}$.
Moreover the ideals of $\mathcal{S}$-Ramsey null subsets of $\mathcal{R}$ and $\mathcal{S}$-meager subsets of $\mathcal{R}$ are $\sigma$-ideals and they also coincide.
\end{theorem}

\noindent Any containment space $(\mathcal{R},\mathcal{S},\le,\le_{\cR})$ satisfying the conclusion of Theorem~\ref{thm.RS} is a {\em Ramsey space}.

Part (2) of axiom {\bf A.3}  has presented a fundamental obstacle to   infinite-dimensional Ramsey theorems for  the majority of homogeneous structures  investigated so far.
A strengthening of {\bf A.4}  was developed by the first author to obtain analogues of the Galvin-Prikry theorem
without using  \axiom\
for the Rado graph in \cite{Dobrinen_Rado} and for diagonal antichains (diaries in the terminology of this paper) in \cite{Dobrinen_SDAP} for 
a  class of homogeneous structures which includes  the unconstrained structures in this paper.
The results in \cite{Dobrinen_SDAP}, however, 
do not apply  to structures of the form Forb$(\mathcal{F})$ when $\mathcal{F}$ has structures of size greater than two.
While  the so-called $\mathbb{Q}$-like structures were shown  in \cite{Dobrinen_SDAP} to carry analogues of the Ellentuck theorem,
the free amalgamation classes in \cite{Dobrinen_SDAP} were only shown to carry analogues of the Galvin-Prikry Theorem.

Motivated to obtain a stronger infinite-dimensional Ramsey theorem which is moreover applicable to the class of finitely constrained free amalgamation classes with finitely many binary relations,
we will prove
Theorem \ref{thm.ARTIdeal}:
the conclusion of 
the Abstract Ramsey Theorem
still holds when \axiom\ is replaced by a weaker version.
This involves the new notion of \axiom-ideals.

\subsection{\axiom-ideals}

\textbf{Throughout this subsection, we fix} a containment space $\Omega:= (\cR, \cS,  \leq, \leq_\cR)$. 

\begin{defin}
    \label{Def:Branching_Approx}
    \begin{enumerate}
        \item 
        Given $X\in \cS$, we call $a\in \cal{AR}^X$ \emph{$X$-branching} if whenever $b\in \cal{AR}^X$ satisfies $a\sqsubseteq b$ and $[a, X] = [b, X]$, then $a = b$. Equivalently, if $a\in \cal{AR}^X_\ell$, then $a$ is $X$-branching iff $|\{b\in \cal{AR}^X_{\ell+1}: a\sqsubseteq b\}|\geq 2$. Write $\cal{BAR}^X := \{a\in \cal{AR}^X: a\text{ is $X$-branching}\}$. 
        \item 
        We say that $\Omega$ \emph{respects branching} if whenever $X\in \cS$, $a\in \cal{AR}^X$, and $Y\in [\rm{depth}_X(a), X]$, we have that $a\in \cal{BAR}^X$ iff $a\in \cal{BAR}^Y$.
        \item 
        We call $\Omega$ \emph{perfect} if whenever $X\in \cS$ and $a\in \cal{AR}^X$, there is $b\in \cal{BAR}^X$ with $a\sqsubseteq b$; equivalently, this happens iff we always have $[a, X]\subseteq \cR$ infinite. \qed
    \end{enumerate}
\end{defin}

\begin{defin}
	\label{Def:A32Ideal}
	A subset $\cI\subseteq \cS\times \cS$ is called an \emph{ideal} if
	\begin{itemize}
		\item 
		$(X, Y)\in \cI\Rightarrow X\leq Y$.
		\item
		$(X, Y)\in \cI\text{ and }Z\leq X \Rightarrow (Z, Y)\in \cI$.
	\end{itemize} 
		
	We call an ideal $\cI\subseteq \cS\times\cS$ an \emph{\axiom-ideal} for $\cR$ if additionally
	\begin{itemize}
		\item
		$\forall Y{\in}\cS\,\, \forall n{<}\omega\,\, \exists Y'{\in}\cS$ with $(Y', Y)\in \cI$ and $Y'|_n = Y|_n$.
        \item 
        If $(X, Y)\in \cI$ and $a\in \cal{BAR}^X$, then there is $Y'\in \cS$ with $Y'\in [\rm{depth}_Y(a), Y]$, $(Y', Y)\in \cI$, and $[a, Y']\subseteq [a, X]$. \qed         
	\end{itemize}
    $\Omega$ \emph{admits an \axiom-ideal} if there is $\cI\subseteq \cS\times \cS$ which is an \axiom-ideal for $\cR$.
\end{defin}

\textbf{For the rest of the subsection}, we additionally assume that the containment space $\Omega$ satisfies {\bf A.4}, respects branching, is perfect, and that $\cI\subseteq \cS\times \cS$ is an \axiom-ideal for $\cR$. We now proceed to follow chapter 4 of \cite{stevo_book} under these hypotheses. First we relativize the combinatorial forcing to $\cI$.

\begin{defin}
	\label{Def:ComboForcing}
	Fix $Y\in \cS$, $a\in \cal{AR}^Y$, and $\frakX\subseteq \cR$. We say that $Y$ \emph{$\frakX$-accepts} $a$ if $[a, Y]\subseteq \frakX$. We say that $Y$ \emph{$\frakX$-rejects} $a$ if for every $Y'\in [\rm{depth}_Y(a), Y]$ with $(Y', Y)\in \cI$, we have that $Y'$ does not $\frakX$-accept $a$. We say that $Y$ \emph{$\frakX$-decides} $a$ if it either $\frakX$-accepts $a$ or $\frakX$-rejects $a$.
\end{defin}

For the time being, $\frakX\subseteq \cR$ will be fixed, so we simply say \emph{accepts}, etc. We note the following key use of \textbf{A.3(2)} which now is performed by the \axiom-ideal.
	
\begin{prop}
	\label{Prop:Rejects}
	Suppose $(X, Y)\in \cI$, and fix $a\in \cal{BAR}^X$. 
	If $Y$ decides $a$, then so does $X$ with the same decision.
\end{prop}

\begin{proof}
	As $X\leq Y$, we have $[a, X]\subseteq [a, Y]$, so if $Y$ accepts $a$, then so does $X$. 
		
	Now suppose $Y$ rejects $a$. Towards a contradiction, suppose that $X$ does not reject $a$. Then there is $X'\in [\rm{depth}_X(a), X]$ with $(X', X)\in \cI$ and $[a, X']\subseteq \frakX$. We note that $a\in \cal{BAR}^{X'}$ since $\Omega$ respects branching. As $(X, Y)\in \cI$, also $(X', Y)\in \cI$. Since
	$\cI$ is an \axiom-ideal for $\cR$ and $a\in \cal{BAR}^{X'}$, find $Y'\in [\rm{depth}_Y(a), Y]$ with $(Y', Y)\in \cI$ and $[a, Y']\subseteq [a, X']$. So $Y'$ accepts $a$, a contradiction.  
\end{proof}

We also have the following easy Lemma.

\begin{lemma}
	\label{Lem:ExistsDecide}
	If $a\in \cal{AR}^Y$, then there is $Y'\in [\rm{depth}_Y(a), Y]$ with $(Y', Y)\in \cI$ and so that $Y'$ decides $a$.
\end{lemma}

\begin{proof}
	If there is such a $Y'$ that accepts $a$, this $Y'$ works. Otherwise, $Y$ rejects $a$, and so will any $Y'\in [\rm{depth}_Y(a), Y]$ with $(Y', Y)\in \cI$.  
\end{proof}

\begin{prop}
	\label{Prop:Decides}
	For every $Y\in \cS$ and $n< \omega$, there is $X\in \cS$ with $(X, Y)\in \cI$, $X|_n = Y|_n$, and so that $X$ decides every $a\in \cal{AR}^X$ with $\rm{depth}_X(a)\geq n$.
\end{prop}
	
\begin{proof}
	The proof is almost identical to that of Lemma 4.33 from \cite{stevo_book}. The only difference is that we ensure that $(Y_{k+1}, Y_k)\in \cI$ for each $k< \omega$. To do this, when the $X_i$ are defined with $Y_k = X_0\geq\cdots \geq X_p = Y_{k+1}$, we ensure that $(X_{i+1}, X_i)\in \cI$. This is possible by Lemma~\ref{Lem:ExistsDecide}. Upon setting $X = \lim_k Y_k$, we have $X\leq Y_k$ for every $k$, so in particular $(X, Y)\in \cI$.
\end{proof}

We note that when working with \axiom-ideals, we need the following stronger-looking version of \textbf{A.4}.
	
\begin{lemma}
	\label{Lem:StrongA4}
	Suppose $Y\in \cS$, $\ell< \omega$, and $a\in \cal{AR}^Y_\ell$. Write $n = \rm{depth}_Y(a)$, and write $\Psi = \{b\in \cal{AR}^Y_{\ell+1}: a\sqsubseteq b\}$. Then for any partition $\Psi = \Psi_0\sqcup \Psi_1$, there is $X\in [n, Y]$ with $(X, Y)\in \cI$ and so that $\{b\in \cal{AR}^X_{\ell+1}: a\sqsubseteq b\}\subseteq \Psi_i$ for some $i< 2$.
\end{lemma}

\begin{proof}
	First find any $Y'\in [n, Y]$ with $(Y', Y)\in \cI$. Then consider $\Psi':= \{b\in \cal{AR}^{Y'}_{\ell+1}: a\sqsubseteq b\}$ and $\Psi'_i:= \Psi'\cap \Psi_i$ for $i< 2$. Use ordinary \textbf{A.4} to find $X\in [n, Y']\subseteq [n, Y]$ with $\{b\in \cal{AR}^X_{\ell+1}: a\sqsubseteq b\}\subseteq \Psi'_i\subseteq \Psi_i$ for some $i< 2$. Then $(X, Y)\in \cI$ as desired.
\end{proof}

For Lemma~\ref{Lem:GrowRejects}, we find $X$ which rejects every $b\in \cal{AR}^X_{\ell+1}$ which is $X$-branching and extends $a$. Similarly for Lemma~\ref{Lem:RejectAll}.

\begin{lemma}
	\label{Lem:GrowRejects}
	Suppose $Y\in \cS$, $\ell< \omega$, and $a\in \cal{AR}^Y_\ell$ are such that $Y$ rejects $a$. Then there is $X\in [\rm{depth}_Y(a), Y]$ with $(X, Y)\in \cI$ and such that $X$ rejects every member of $\cal{BAR}^X_{\ell+1}$ extending $a$. 
\end{lemma}

\begin{proof}
	Write $n = \rm{depth}_Y(a)$. Use Proposition~\ref{Prop:Decides} to find $Z\in [n, Y]$ with $(Z, Y)\in \cI$ and so that $Z$ decides every $b\in \cal{AR}^Z$ with $\rm{depth}_Z(b)\geq n$. Note that $[a, Z]\neq \emptyset$ by \textbf{A.3(1)}, and since $(Z', Z), (Z, Y)\in \cI$ implies $(Z', Y)\in \cI$, we have that $Z$ rejects $a$. Now consider the partition of $\Psi:= \{b\in \cal{AR}^Z_{\ell+1}: a\sqsubseteq b\}$ into $\Psi_{acc}\sqcup\Psi_{rej}$ depending on how $Z$ decides. By Lemma~\ref{Lem:StrongA4}, find $X\in [n, Z]$ with $(X, Z)\in \cI$ and with $\Theta:= \{b\in \cal{AR}^X_{\ell+1}: a\sqsubseteq b\}\subseteq \Psi_i$ for some $i\in \{acc, rej\}$. Note that as $(Z, Y)\in \cI$, we have $(X, Y)\in \cI$. 
		
	Towards a contradiction, suppose $i = acc$. Then $[b, X]\subseteq [b, Z]\subseteq \frakX$ for every $b\in \cal{AR}^X_{\ell+1}$ extending $a$. But this implies that $[a, X]\subseteq \frakX$, i.e.\ that $X$ accepts $a$, a contradiction as $Y$ rejects $a$ and $X\leq Y$.
		
	Thus $Z$ rejects every $b\in \Theta$. As $(X, Z)\in \cI$ and $\Theta\subseteq \cal{AR}^X$, it follows by Proposition~\ref{Prop:Rejects} that $X$ rejects every $b\in \Theta\cap \cal{BAR}^X$.
\end{proof}

\begin{lemma}
	\label{Lem:RejectAll}
	Suppose $Y\in \cS$, $\ell< \omega$, and $a\in \cal{AR}^Y_\ell$. Assume that $Y$ decides every $b\in \cal{AR}^Y$ with $\rm{depth}_Y(b)\geq \rm{depth}_Y(a)$. If $Y$ rejects $a$, then there is $X\in [\rm{depth}_Y(a), Y]$ with $(X, Y)\in \cI$ and so that $X$ rejects every $b\in \cal{BAR}^X$ with $a\sqsubseteq b$. 
\end{lemma}

\begin{proof}
	The argument is almost identical to that of Lemma 4.35 from \cite{stevo_book}. Similar to Proposition~\ref{Prop:Decides}, the main difference is ensuring that $(Y_{k+1}, Y_k)\in \cI$ for every $k< \omega$. To do this, when the $X_i$ are defined with $Y_k = X_0\geq\cdots \geq X_p = Y_{k+1}$, we ensure that $(X_{i+1}, X_i)\in \cI$, and this is possible by Lemma~\ref{Lem:GrowRejects}.
\end{proof}

When working with \axiom-ideals, we 
note that  
the following stronger-looking versions of being $\cS$-Baire and $\cS$-Ramsey are equivalent to the original definition.
	
\begin{fact}
    \label{Fact:SRamseyBaire}
    In Definition~\ref{Def:SRamseyBaire}, when defining what it means for $\frakX\subseteq \cR$ to be $\cS$-Ramsey, $\cS$-Ramsey null, $\cS$-Baire, or $\cS$-meager, it is enough to only consider those non-empty basic sets $[a, Y]$ with $a\in \cal{BAR}^Y$. To see why for $\cS$-Ramsey, fix $\frakX\subseteq \cR$ which is $\cS$-Ramsey in this a priori weaker sense. Consider some non-empty basic set $[a, Y]$. As $[a, Y]$ is infinite, find $b\in \cal{AR}^Y$ with $a\sqsubseteq b$, $[a, Y] = [b, Y]$, and with $b\in \cal{BAR}^Y$. By assumption, we can find $X\in [\rm{depth}_Y(b), Y]\subseteq [\rm{depth}_Y(a), Y]$ with $[b, X]$ $\frakX$-invariant. Lastly, we note that $[a, X] = [b, X]$.

    When considering $\cS$-Baire, given $[a, Y]$, the witness $[b, X]$ may be chosen with $b\in \cal{BAR}^X$ and with $(X, Y)\in \cI$. To see this, first find $Y'\in [\depth_Y(a), Y]$ with $(Y', Y)\in \cI$. Then use the ordinary definition of $\cS$-Baire to find $X\leq Y'$ with $a\in \cal{AR}^X$ and $b'\in \cal{AR}^X$ which is $\frakX$-invariant,  Then $(X, Y)\in \cI$ as desired, and then extend $b'$ to $b\in \cal{BAR}^X$. By a similar argument, for $\cS$-Ramsey, we may choose the witness $X\in [\rm{depth}_Y(a), Y]$ with $(X, Y)\in \cI$.
\end{fact}

As each $\mathcal{S}$-Ramsey set is $\mathcal{S}$-Baire,  the following theorem consists of proving the converse.

\begin{theorem}
	\label{Thm:BaireEqualsRamsey}
	Given $\frakX\subseteq \cR$, we have that $\frakX$ is $\cS$-Baire iff it is $\cS$-Ramsey. 
\end{theorem}

\begin{proof}
	Assume $\frakX$ is $\cS$-Baire. Towards showing $\cS$-Ramsey, fix $Y\in \cS$ and $a\in \cal{AR}^Y$. Write $n = \depth_Y(a)$. Use Proposition~\ref{Prop:Decides} to find $X\in [n, Y]$ with $(X, Y)\in \cI$ and so that $X$ $\frakX$-decides and $\frakX^c$-decides every $b\in \cal{AR}^X$ with $\depth_Y(b)\geq n$. 
		
	If $X$ $\frakX$-accepts or $\frakX^c$-accepts $a$, we are done, so towards a contradiction assume reject for both. By Lemma~\ref{Lem:RejectAll}, find $Z\in [n, X]$ which $\frakX$-and-$\frakX^c$-rejects every $b\in \cal{AR}^Z$ with $a\sqsubseteq b$.
		
	By Fact~\ref{Fact:SRamseyBaire}, find $Z_0\in \cS$ with $(Z_0, Z)\in \cI$ along with $b\in \cal{BAR}^{Z_0}$ with $a\sqsubseteq b$ so that $[b, Z_0]$ is $\frakX$-invariant. As $\cI$ is an \axiom-ideal, find $Z_1\in [\depth_{Z}(b), Z]$ with $(Z_1, Z)\in \cI$ and $[b, Z_1]\subseteq [b, Z_0]$. Because $(Z_1, Z)\in \cI$ and $Z$ $\frakX$-and-$\frakX^c$-rejects $b$, we cannot have $[b, Z_1]$ $\frakX$-invariant, a contradiction. 
\end{proof}

\begin{theorem}
	\label{Thm:CountablyAdditive}
	The collection of $\cS$-Ramsey subsets of $\cR$ is countably additive. 
\end{theorem}

\begin{proof}
	Suppose $\{\frakX_k: k< \omega\}$ are $\cS$-Ramsey subsets of $\cR$, and set $\frakX = \bigcup_k \frakX_k$. Fix $Y\in \cS$ and $a\in\cal{AR}^Y$. Write $n = \depth_Y(a)$. By applications of Proposition~\ref{Prop:Decides} and Lemma~\ref{Lem:RejectAll}, we may assume that $Y$ $\frakX$-rejects every $b\in \cal{BAR}^Y$ with $a\sqsubseteq b$. 
		
	Setting $Y = Y_0$, we now run a fusion construction. If $Y_k$ is given, use the assumption that every $\frakX_i$ is $\cS$-Ramsey to find $Y_{k+1}\in [n+k, Y_k]$ such that for every $b\in \cal{AR}^{Y_k}$ with $a\sqsubseteq b$ and $\depth_{Y_k}(b) = n+k$, we have that $[b, Y_{k+1}]$ is $\frakX_i$-invariant for all $i\leq k$. The number of such $b$ is finite by \textbf{A.2}. By Fact~\ref{Fact:SRamseyBaire}, we can assume that $(Y_{k+1}, Y_k)\in \cI$ for every $k< \omega$.
		
	Let $X = \lim_k Y_k$, and note that $(X, Y)\in \cI$; we show that $[a, X]\cap \frakX = \emptyset$. Fix $A\in [a, X]$ and $i< \omega$, towards showing that $A\not\in \frakX_i$. Fix $\ell< \omega$ so that, writing $b = A|_\ell$, we have $b\in \cal{BAR}^X$ and $\depth_X(b)\geq n+i$. For some $k\geq i$, we have that $\depth_X(b) = n+k$. As $X\in [n+k, Y_k]$, also $\depth_{Y_k}(b) = n+k$. It follows by the construction of $Y_{k+1}$ that $[b, X]$ is $\cX_i$-invariant. As $(X, Y)\in \cI$ and $Y$ $\frakX$-rejects $b$, also $X$ $\frakX$-rejects $b$ by Proposition~\ref{Prop:Rejects}. Thus we cannot have $[b, X]\subseteq \frakX$, so $[b, X]\cap \frakX_i = \emptyset$, and in particular, $A\not\in \frakX_i$. 
\end{proof}

\begin{theorem}
	\label{Thm:RamseyNull}
	Given $\frakX\subseteq \cR$, we have that $\frakX$ is $\cS$-Ramsey-null iff $\frakX$ is $\cS$-meager.
\end{theorem}

\begin{proof}
	Identical to that of Lemma 4.38 of \cite{stevo_book}.
\end{proof}

\begin{theorem}
	\label{Thm:Souslin}
	The collection of $\cS$-Ramsey subsets of $\cR$ is closed under the Souslin operation.
\end{theorem}

\begin{proof}
	Fix a Souslin scheme $\{\frakX_s: s\in \omega^{<\omega}\}$ with each $\frakX_s$ $\cS$-Ramsey; we may assume that $\frakX_s\supseteq \frakX_t$ whenever $s\sqsubseteq t$. Write $\displaystyle\frakX = \bigcup_{f\in \omega^\omega} \bigcap_{n< \omega} \frakX_{f|_n}$ for the result of the Souslin operation, and given $s\in \omega^{<\omega}$, write $\wt{\frakX}_s$ for the result of the Souslin operation applied to the subtree $N_s := \{t\in \omega^{<\omega}: t\sqsubseteq s\text{ or }s\sqsubseteq t\}$. Note that $\wt{\frakX}_s\subseteq \frakX_s$ and $\wt{\frakX}_\emptyset = \frakX$. 
		
	Towards showing that $\frakX$ is $\cS$-Ramsey, fix $Y\in \cS$ and $a\in \cal{AR}^Y$. Write $n = \depth_Y(a)$. Enumerate $\omega^{<\omega}$ as $\{s_i: i< \omega\}$ so that the enumeration is non-$\sqsubseteq$-decreasing. Setting $Y_0 = Y$, we perform a fusion construction. If $Y_k$ is given, choose $Y_{k+1}\in [n+k, Y_k]$ with $(Y_{k+1}, Y_k)\in \cI$ and so that for any $b\in \cal{AR}^{Y_k}$ with $\depth_{Y_k}(b) = n+k$, $Y_{k+1}$ $(\wt{\frakX}_{s_i})^c$-decides $b$ for every $i\leq k$. Let $X = \lim_k Y_k$. Note that $(X, Y_k)\in \cI$ for every $k< \omega$ and that $\depth_X(a) = n$. 
		
	For the time being, fix $i< \omega$, and let $s = s_i\in \omega^{<\omega}$. Write $\Psi = \{b\in \cal{AR}^X: [b, X]\cap \wt{\frakX}_s = \emptyset\}$. Then set $\Phi(\wt{\frakX}_s) = ([a, X]\cap \frakX_s)\setminus \bigcup\{[b, X]: b\in \Psi\}$. So $([a, X]\cap \wt{\frakX}_s)\subseteq \Phi(\wt{\frakX}_s)\subseteq ([a, X]\cap \frakX_s)$. We note that metrically closed subsets of $\cR$ are $\cS$-Baire, so in particular, each $[b, X]$ is $\cS$-Baire, so also $\Phi(\wt{\frakX}_s)$ is $\cS$-Baire, hence $\cS$-Ramsey and also $\cS({\leq}X)$-Baire, where $\cS({\leq}X):= \{Y\in \cS: Y\leq X\}$, and we can discuss subsets of $\cR({\leq}X):= \{M\in \cR: M\leq_{\cR} X\}$ as being $\cS({\leq}X)$-Baire, meager, Ramsey, or Ramsey null. We remark that the containment space $(\cR({\leq}X), \cS({\leq}X), \leq, \leq_\cR)$, with respective sets of approximations $\cal{AR}^X$ and $\cal{AS}^X$, is perfect, respects branching, satisfies Axiom \textbf{A.4}, and that $\cI\cap (\cS({\leq}X)\times \cS({\leq}X))$ is an \axiom-ideal for $\cR({\leq}X)$. Hence we may freely use Theorems~\ref{Thm:BaireEqualsRamsey}, \ref{Thm:CountablyAdditive}, and \ref{Thm:RamseyNull} for $(\cR({\leq}X), \cS({\leq}X), \leq, \leq_\cR)$.
    \vspace{2 mm}

	\begin{claim}
		Any $\frakY\subseteq \Phi(\wt{\frakX}_s)\setminus \wt{\frakX}_s$ which is $\cS({\leq}X)$-Baire is $\cS({\leq}X)$-meager. 
	\end{claim}
	
	\begin{proof}[Proof of claim]
		Fix $X'\leq X$ and $b\in \cal{AR}^{X'}$. By growing $b$ if needed so that $|a|\leq |b|$, we may assume $a\sqsubseteq b$, as otherwise $[b, X']\cap \frakY = \emptyset$. Find $b'\in \cal{AR}^{X'}$ with $b\sqsubseteq b'$ and so that $\depth_X(b')\geq n+i$. As $\frakY$ is $\cS({\leq}X)$-Baire, find $Z\leq X'$ and $c\in \cal{BAR}^Z$ with $b\sqsubseteq c$ and $[c, Z]$ $\frakY$-invariant. We want $[c, Z]\cap \frakY = \emptyset$, so towards a contradiction, suppose $[c, Z]\subseteq \frakY$.
		
		Write $n+k = \depth_X(c)$, and note that $k\geq i$. By the construction of $X$, also $n+k = \depth_{Y_k}(c)$, so $Y_{k+1}$ $(\wt{\frakX}_s)^c$-decides $c$. As $[c, Z]\subseteq \frakY$ and $\frakY \cap \wt{\frakX}_s = \emptyset$, we have that $Z$ $(\wt{\frakX}_s)^c$-accepts $c$. As $(X, Y_{k+1})\in \cI$ and $Z\leq X$, we have $(Z, Y_{k+1})\in \cI$, so by Proposition~\ref{Prop:Rejects}, we must have $Y_{k+1}$ $(\wt{\frakX}_s)^c$-accepts $c$. But then $c\in \Psi$, so by the definition of $\Phi(\wt{\frakX}_s)$, we have $[c, X]\cap \Phi(\wt{\frakX}_s) = \emptyset$, contradicting that $[c, Z]\subseteq [c, X]$ and $[c, Z]\subseteq \frakY$. 
	\end{proof}
	
	Now given $t\in \omega^{<\omega}$, consider $\frakY_t = \Phi(\wt{\frakX}_t)\setminus \bigcup\{\Phi(\wt{\frakX}_{t^\frown m}): m< \omega\}$. As each $\Phi(\wt{\frakX}_{t^\frown m})$ contains $\wt{\frakX}_{t^\frown m}$ and $\wt{\frakX}_t = \bigcup_m \wt{\frakX}_{t^\frown m}$, we have that $\frakY_t \subseteq \Phi(\wt{\frakX}_t)\setminus \wt{\frakX}_t$. As $\frakY_t$ is $\cS({\leq}X)$-Baire by Theorem~\ref{Thm:CountablyAdditive}, it is $\cS({\leq}X)$-meager by the claim, hence also $\cS({\leq}X)$-Ramsey-null. 
 \vspace{2 mm}
		
	\begin{claim}
		Writing $\Phi(\frakX)$ for $\Phi(\wt{\frakX}_\emptyset)$, we have $(\Phi(\frakX)\setminus \frakX)\subseteq \bigcup_{t\in \omega^{<\omega}} \frakY_t$. 
	\end{claim}
		
	\begin{proof}[Proof of claim]
		Fix $A\in \Phi(\frakX) = \Phi(\wt{\frakX}_\emptyset)$. If it exists, find a $\sqsubseteq$-maximal $t\in \omega^{\omega}$ with $A\in \Phi(\wt{\frakX}_t)$. Then $A\in \frakY_t$. If no such $t$ exists, then let $f\in \omega^\omega$ be such that $A\in \Phi(\wt{\frakX}_{f|_m})$ for every $m< \omega$. As $\Phi(\wt{\frakX}_{f|_m})\subseteq \frakX_{f|_m}$, we conclude that $A\in \frakX$.  
	\end{proof}
		
	Again using Theorem~\ref{Thm:CountablyAdditive}, we can find $Z\in [n, X]$ with $[a, Z]\cap \frakY_t = \emptyset$ for every $t\in \omega^{<\omega}$. It follows that $[a, Z]\cap \Phi(\frakX) = [a, Z]\cap \frakX$.  As $\Phi(\frakX)$ is $\cS({\leq}X)$-Ramsey, find $Z_0\in [n, Z]$ so that $[a, Z_0]$ is $\Phi(\frakX)$-invariant. Then also $[a, Z_0]$ is $\frakX$-invariant as desired.
\end{proof}

From the previous theorems, we obtain the conclusion of the Abstract Ramsey Theorem with  Axiom \axiom\ replaced with the existence of an \axiom-ideal.

\begin{theorem}\label{thm.ARTIdeal}
Suppose $\Omega = (\cR,\cS,\le,\le_{\cR})$ is a containment space which is perfect, respects branching, satisfies axiom {\bf A.4}, and admits an {\bf A.3(2)}-ideal for $\cR$.
Then the conclusion of the Abstract Ramsey Theorem holds.
\end{theorem}

\begin{rem}
    If in Definition~\ref{Def:A32Ideal}, one strengthens the last bullet to hold for any $a\in \cal{AR}^X$, then one can obtain the conclusion of Theorem~\ref{thm.ARTIdeal} without the assumption that $\Omega$ is perfect or respects branching. The original \axiom-axiom states that $\{(X, Y): X\leq Y\}\subseteq \cS\times \cS$ is a \axiom-ideal in this stronger sense. 
\end{rem}

For many containment spaces, including the main one of consideration in this paper, we can simplify the definition of \axiom-ideal.
	
\begin{defin}
	\label{Def:MonoidBased}
	We call a containment space $(\cR, \cS, \leq_\cR, \leq)$ \emph{monoid-based} if 
	\begin{itemize}
		\item
		$\cS$ is an infinite left-cancellative monoid, and given $\phi, \sigma\in \cS$, $\sigma\leq \phi$ iff $\sigma\in \phi\cdot\cS$.
		\item
		For every $n< \omega$, there is a left-cancellative action $\cS\times \cal{AS}_n\to \cal{AS}_n$ such that the restriction map $\phi\to \phi|_n$ is \emph{$\cS$-equivariant}, i.e.\ for any $\phi, \sigma\in \cS$ and any $n< \omega$, we have $\phi\cdot (\sigma|_n) = (\phi\cdot \sigma)|_n$. 
		\item
		There are left-cancellative left actions of $\cS$ on $\cR$ and $\cal{AR}_n$ for each $n< \omega$ such that the restriction maps are $\cS$-equivariant. Additionally, given $(\eta, \phi)\in \cR\times \cS$, we have $\eta\leq_\cR \phi$ iff $\eta\in \phi\cdot \cR$. \qed
	\end{itemize}
\end{defin}

When working with monoid-based containment spaces, when given $h\in \cal{AR}$, we write $\depth(h)$ for $\depth_{\rm{id}}(h)$, and similarly for members of $\cal{AS}$. Given $h\in \cal{AR}$ and $\phi\in \cS$, we note that $h\in \cal{AR}^\phi$ iff there is $g\in \cal{AR}$ with $h = \phi\cdot g$, in which case $[h, \phi] = \phi\circ [g]$ and $\rm{depth}_\phi(h) = \rm{depth}(g)$. We put $\cal{BAR} := \cal{BAR}^{\rm{id}}$, and likewise for $\cal{BAR}_\ell$. Monoid-based containment spaces always respect branching and are always perfect.

In this case, we call $\cI\subseteq \cS$ an \emph{\axiom-ideal for $\cR$} if $\{(\sigma, \phi)\in \cS\times \cS: \sigma\in \phi\cdot \cI\}$ is an \axiom-ideal as in Definition~\ref{Def:A32Ideal}. This means that $\cI$ is a right ideal satisfying 
\begin{itemize}
	\item
	$\forall n{<}\omega\,\, (\cI\cap [\rm{id}_n])\neq \emptyset$.
	\item
	For every $h\in \cal{BAR}$ and $\psi\in \cI$, writing $p = \depth(\psi\circ h)$, there is $\sigma\in (\cI\cap [\rm{id}_p])$ with $\sigma\circ [\psi\circ h]\subseteq \psi\circ [h]$.
\end{itemize}

For the rest of the paper, this is the sense in which we will be thinking of \axiom-ideals.

\section{Infinite-dimensional Ramsey theorems}\label{Sec:Main Theorem}

Recall that $\Delta$ is the fixed strong diary constructed in Section~\ref{Sec:Strong diaries}.

\begin{defin}
    \label{Def:StrongEmbs}
    Let $\Theta$ be a diary. A \emph{strong embedding} of $\Theta$ into $\Delta$ is an embedding $\phi\in \demb(\Theta, \Delta)$ with the property that $\wt{\phi}(m)+1 \in \rm{SL}(\Delta)$ for every $m< \omega$, and similarly for strong embeddings of $\Theta({<}n)$ into $\Delta$ for some $n< \rm{ht}(\Theta)$. Write $\StEmb(\Theta, \Delta)$ for the set of strong embeddings of $\Theta$ into $\Delta$, and likewise for $\Theta({<}n)$. Given $\phi\in \StEmb(\Theta, \Delta)$ and $n< \rm{ht}(\Theta)$, we write $\phi|_n$ for $\phi|_{\Theta(<n)}$. 

    Strong embeddings from a diary $\Theta$ into $\Delta$ will be the objects we turn into a Ramsey space. \textbf{For the remainder of this section, fix} a diary $\Theta$ with $\rm{ht}(\Theta) = \omega$. In following the notation from chapter 4 of  \cite{stevo_book}, we write 
    \begin{align*}
        \cR &:= \StEmb(\Theta, \Delta),\\
        \cal{AR}_n &:= \StEmb(\Theta({<}n), \Delta) = \{\phi|_n: \phi\in \cR\},\\
        \cal{AR} &:= \bigcup_{n< \omega} \cal{AR}_n.
    \end{align*}
    The two descriptions of $\cal{AR}_n$ are equivalent by repeated application of Proposition~\ref{Prop:Find_Level_Type}. Given $f\in \c{AR}_n$, let $\rm{depth}(f) = \tilde{f}(n-1)+1\in \rm{SL}(\Delta)$, and let $[f] = \{\phi\in \cR: \phi|_n = f\}$.

    When considering strong embeddings of $\Delta$ into itself, there is a natural variation. Given $\mu\in \rm{SL}(\Delta)\cup \{\omega\}$, we call $\psi\in \demb(\Delta, \Delta)$ \emph{$\mu$-strong} if $\psi$ is the identity on $\psi({<}\mu)$, and for $m\geq \mu$, we have $\wt{\psi}(m)+1\in \rm{SL}(\Delta)$. In particular, $\phi\in \demb(\Delta, \Delta)$ is $0$-strong iff $\phi\in \StEmb(\Delta, \Delta)$, and the identity embedding is the unique $\omega$-strong member of $\demb(\Delta, \Delta)$. Write $\cS_\mu$ for the set of all $\mu$-strong members of $\demb(\Delta, \Delta)$, and put 
    \begin{align*}
        \cS &= \bigcup_{\mu\in \SL(\Delta)\cup \{\omega\}} \cS_\mu,\\
        \cal{AS}_n &= \{\phi|_{\Delta({<}n)}: \phi\in \cS\},\\
        \cal{AS} &= \bigcup_{n<\omega} \cal{AS}_n.
    \end{align*}
    Given $s\in \cal{AS}_n$, write $[s] = \{\psi\in \cS: \psi|_n = s\}$. On $\cS$, let $\leq$ denote the partial order where $\sigma\leq \phi$ iff $\sigma\in \phi\cdot\cS$, and let $\leq_\cR\,\subseteq \cR\times\cS$ be defined so that $\eta \leq_\cR \phi$ iff $\eta\in \phi\cdot \cR$. Then $(\cR, \cS, \leq_\cR, \leq)$ is a monoid-based 
    containment space. 
\end{defin}

One of the main theorems of this paper is the following.

\begin{theorem}
    \label{Thm:Ramsey_Space}
    With notation as in Definition~\ref{Def:StrongEmbs}, $(\cR, \cS, \leq_\cR, \leq)$ is a  Ramsey 
    space. More precisely, $(\cR, \cS, \leq_\cR, \leq)$ satisfies \textbf{A.4} and $\cS$ contains an \axiom-ideal for $\cR$. 
\end{theorem}

In the case when $\Theta = \Delta$, we can prove a stronger result. 

\begin{theorem}
    \label{Thm:TRS}
    With notation as in Definition~\ref{Def:StrongEmbs}, $(\cS, \leq)$ is a topological Ramsey space.
\end{theorem}
This does not directly follow from Theorem~\ref{Thm:Ramsey_Space}, since in the case that $\Theta = \Delta$, we have $\cR = \cS_0\subsetneq \cS$. 

Before undertaking the proof of Theorem~\ref{Thm:Ramsey_Space}, let us show how this implies theorems analogous to those of Galvin-Prikry \cite{GP} and Silver \cite{Silver} for \emph{all} embeddings, not just strong ones. 

\begin{theorem}
    \label{Thm:General_GP_Silver}
    Let $\Gamma$ and $\Theta$ be any diaries coding $\bK$. 
        Then for any finite Souslin-measurable coloring of $\demb(\Theta, \Gamma)$, there is $h\in \demb(\Gamma, \Gamma)$ with $h\circ \demb(\Theta, \Gamma)$ monochromatic. 
\end{theorem}

\begin{proof}
    Fix a Souslin-measurable coloring $\chi\colon \demb(\Theta, \Gamma)\to r$. By Theorem 6.2.14 of \cite{BCDHKVZ}, any two diaries which code $\flim(\cK)$ are bi-embeddable. So fix $\psi\in \demb(\Delta, \Gamma)$.  Then $\chi\circ \psi\colon \demb(\Theta, \Delta)\to r$ is also Souslin-measurable. So also $\chi\circ \psi|_\cR$ is Souslin-measurable.
    By Theorem~\ref{Thm:Ramsey_Space}, there is $\sigma\in \cS$ with $\chi\circ \psi\circ \sigma|_\cR$ monochromatic. Now let $g\in \cS$ be chosen so that $\wt{g}(m)+1\in \SL(\Delta)$ for every $m< \omega$. Then $g\circ \demb(\Theta, \Delta)\subseteq \cR$, so in particular, $\chi\circ \psi\circ \sigma\circ g\colon \demb(\Theta, \Delta)\to r$ is constant. Now pick any $\beta\in \demb(\Gamma, \Delta)$, and observe that $\chi\circ \psi\circ \sigma\circ g\circ \beta\colon \demb(\Theta, \Gamma)\to r$ is constant. Thus $h:= \psi\circ \sigma\circ g \circ \beta \in \demb(\Gamma, \Gamma)$ is as desired. 
\end{proof}

\begin{rem}
The conclusion of Theorem~\ref{Thm:General_GP_Silver} goes well beyond just Souslin measurable colorings. Indeed, all that is needed for the proof is that there is some $\psi\in \demb(\Delta, \Gamma)$ so that $\chi\circ \psi|_\cR$ is $\cS$-Baire. 
\end{rem} 

\begin{rem}
Theorem~\ref{Thm:General_GP_Silver} cannot be strengthened to allow countably many colors. Fix any diary $\Gamma$. We view $S:= \demb(\Gamma, \Gamma)$ as a submonoid of $\emb(\bK, \bK)$ as in Section~\ref{sec.optimality}. For a given finite $\bA\subseteq \bK$ and $f\in \emb(\bA, \bK)$, the forward orbit $S{\cdot}f\subseteq \emb(\bA, \bK)$ is countably infinite. 
Then we can regard the function $\gamma\colon S\to \emb(\bA, \bK)$ given by $\gamma(s) = s|_{\bA}$ as a countable coloring of $S$. For this coloring, it is impossible to find $\eta\in \demb(\Gamma, \Gamma)$ with $\gamma\cdot \eta$ constant. 
\end{rem}

The remainder of this section is split into three subsections. The first discusses various notions of Ellentuck topology and how they compare. The second proves that $\cS$ contains an \axiom-ideal for $\cR$ and for $\cS$; in fact, the $\cI\subseteq \cS$ will be the same for all cases. The third uses a forcing argument to show that axiom \textbf{A.4} holds.

\subsection{The Ellentuck topology}
\label{Subsec:Ellentuck}

\begin{defin}
    \label{Def:Ellentuck}
    Fix a monoid $M$ equipped with a left-invariant metric $d$. The \emph{Ellentuck topology} on $M$ is the coarsest left-invariant topology containing the metric topology. We write $s_i\xrightarrow{\rmE_M} s$ to denote Ellentuck convergence of a net from $M$. Explicitly, given $(s_i)_{i\in I}$ and $s$ in $M$, we have that $s_i\xrightarrow{\rmE_M} s$ iff $s_i\xrightarrow{d} s$ \emph{and} eventually $s_i\in sM$.
\end{defin}

We will want to consider the above definition in three settings.
\begin{itemize}
    \item
    Ellentuck's original definition \cite{Ellentuck} was of a finer topology on $[\bbN]^{\infty}$. Upon identifying $[\bbN]^\infty$ with the monoid of all increasing injections of $\bbN$ to itself equipped with the pointwise convergence topology, Definition~\ref{Def:Ellentuck} is equivalent to the original.
    \item 
    For any diary $\Gamma$ with $\str(\Gamma) \cong \bK$, we can consider the monoid $M = \demb(\Gamma, \Gamma)$ equipped with its standard metric topology. 
    \item
    In the case $\Gamma = \Delta$, we also want to consider the Ellentuck topology as computed in the submonoid $\cS\subseteq \demb(\Delta, \Delta)$.
\end{itemize} 
First, we discuss $M= \demb(\Gamma, \Gamma)$. In principle, there is another way we could put an Ellentuck-like topology on $M$. As a metric space, we can identify $M$ with a subspace of $\binom{\str^\#(\Gamma)}{\str^\#(\Gamma)} \subseteq [\bbN]^\infty$. Thus we can equip $M$ with the topology it inherits as a subspace of $[\bbN]^\infty$ equipped with its Ellentuck topology. Let us observe that this topology in fact coincides with the Ellentuck topology on $M$ as follows. Given $s\in M$, let $\str^\#(s)$ denote the substructure of $\str^\#(\Gamma)$ corresponding to the image of $s$. Then if $s, t\in M$ are such that $\str^\#(t)\subseteq \str^\#(s)$, it follows that for some $u\in M$ we have $t = su$. Thus if $(s_i)_{i\in I}$ and $s$ are in $M$ and $s_i\to s$ metrically, we have that eventually $s_i\in sM$ iff eventually $\str^\#(s_i)\subseteq \str^\#(s_i)$. 

However, this is no longer true when considering the submonoid $\cS\subseteq \demb(\Delta, \Delta)$. Let us write $\cS$-Ellentuck for the Ellentuck topology on $\cS$ as computed by $\cS$, and $\demb(\Delta, \Delta)$-Ellentuck for the topology on $\cS$ inherited as a subspace of the Ellentuck topology on $\demb(\Delta, \Delta)$. Abstractly, we have that the $\cS$-Ellentuck topology is at least as fine as the $\demb(\Delta, \Delta)$-Ellentuck topology. To see that it is strictly finer, first note that if $s\in \cS_0$, then $s\circ \demb(\Delta, \Delta)\subseteq \cS_0$. This fact can be used to create sequences from $\cS$ converging in the $\demb(\Delta, \Delta)$-Ellentuck topology to another member of $\cS$, but which do not converge in the $\cS$-Ellentuck topology.

Theorem~\ref{Thm:TRS} tells us that for any finite $\cS$-Ellentuck-BP coloring of $\cS$, we can find $u\in \cS$ with $u{\cdot}\cS$ monochromatic. In particular, if $s\in \cS_0$, we have that $us{\cdot}\demb(\Delta, \Delta)$ is monochromatic. Hence if we want to instead consider a finite coloring $\chi$ of $\demb(\Delta, \Delta)$, it suffices that $\chi|_\cS$ is $\cS$-Ellentuck-BP. However, knowing that $\chi$ is $\demb(\Delta, \Delta)$-Ellentuck-BP is neither necessary nor sufficient to conclude this. This is simply because given a set $X$ and topologies $\tau_0\subseteq \tau_1$ on $X$, there isn't a clear comparison between the two corresponding $\sigma$-fields of BP sets. However, we can say the following. Let us call a subset of a topological space $(X, \tau)$ \emph{$\tau$-Souslin-measurable} if it belongs to the smallest $\sigma$-algebra containing $\tau$ and closed under the Souslin operation.

\begin{cor}
    \label{Cor:Ellentuck}
    For any finite Ellentuck-Souslin-measurable $\chi\colon \demb(\Delta, \Delta)\to r$, there is $h\in \demb(\Delta, \Delta)$ with $h\circ \demb(\Delta, \Delta)$ monochromatic.
\end{cor}

\begin{proof}
    Such a $\chi$ satisfies that $\chi|_\cS$ is $\cS$-Ellentuck-Souslin-measurable, so in particular $\cS$-Ellentuck-BP. 
\end{proof}

It makes sense to ask if the conclusion of Corollary~\ref{Cor:Ellentuck} holds for $\demb(\Gamma, \Gamma)$ whenever $\Gamma$ is a diary coding $\bK$. We will see in Subsection~\ref{Subsec:7.3} that it does not.

\subsection{\axiom-generic embeddings}
\label{Subsec:A32GenericEmbs} 

This subsection constructs a single right ideal $\cI\subseteq \cS$ which is simultaneously an \axiom-ideal for $\cR$ and for $\cS$. 

\begin{rem}
    There is no $\cI\subseteq \cS$ which is a ``strong" \axiom-ideal for $\cS$ as in the remark after Theorem~\ref{thm.ARTIdeal}. Towards a contradiction, suppose $\cI\subseteq \cS$ were a ``strong" \axiom-ideal for $\cS$. Note that $\cI\cap \cS_0\neq \emptyset$. Fix $n\in \omega\setminus \SL(\Delta)$, and consider $h:= \rm{id}_n\in \cal{AS}_n$ (note that $h\not\in \cal{BAS}$).    Let $\mu\in \SL(\Delta)$ be the least strong level above $n$, and consider $\psi\in \cI\cap \cS_0$. Let us now compare the sets $[\psi\circ h]$ and $\psi\circ [h]$. For any $\phi\in \psi\circ [h]$ we have $\phi|_\mu = \psi|_\mu$, whereas since $\wt{\psi}(n)+1\in \SL(\Delta)$, we have that $\{\phi|_\mu: \phi\in [\psi\circ h]\}$ is infinite. In particular, writing $p = \rm{depth}(\psi\circ h)$, we have that for any $\sigma\in (\cI\cap [\rm{id}_p])$, $\sigma\circ [\psi\circ h]\not\subseteq \psi\circ [h]$.

    As for $\cR$, we simply note that $\cal{AR} = \cal{BAR}$. 
\end{rem}

\begin{notation}
	If $T\subseteq \Delta$ is level and $s\in T{\downarrow}$, we put $T(s)\in T$ to be the unique
	$t\in T$ with $s\sqsubseteq t$. For $X\subseteq T{\downarrow}$, $T(X) := T\cap \succ_\Delta(X)$. 
\end{notation}
	
\begin{defin}
	\label{Def:A32Generic}
	Given $\phi\in \cS$ and $\mu\in \SL(\Delta)\setminus\{0\}$, we call $\phi$ \emph{\axiom-generic below $\wt{\phi}(\mu)$} if whenever $X\subseteq \phi^+[\Delta(\mu)]$ and, writing $m = \wt{\phi}(\mu-1)+1$ (in particular $\phi^+[\Delta(\mu)]\subseteq \Delta(m)$), if $\tau$ is a critical level type based on $S\subseteq\Delta(m)$ with $X\subseteq S$ and  $\rm{Crit}(\tau)\not\subseteq X$, there is $T\sqsupseteq_{age} S$ with $T\cong^* \tau$, $\ell(T)< \wt{\phi}(\mu)$, $\ell(T)+1\in \SL(\Delta)$, and $T(X)\subseteq \im(\phi){\downarrow}$.

	We define $\cI\subseteq \cS$ to be those $\phi\in \cS$ which are \axiom-generic below $\wt{\phi}(\mu)$ for all but finitely many $\mu\in \SL(\Delta)\setminus\{0\}$.
\end{defin}

\begin{figure}
\begin{center}
\begin{tikzpicture}[scale=.2]

\node at (0,0) {};

\foreach \x in {-15,-7,0,6, 12}{
\foreach \y in {0}{
\node at (\x,\y) {};
}}

\draw[dotted] (-17,0) -- (14,0);

\node at (17,0) {$\Delta(m)$};

\node[circle, fill=blue,inner sep=0pt, minimum size=2pt] at (-15,0) {};
\node[circle, fill=black,inner sep=0pt, minimum size=2pt] at (-7,0) {};
\node[circle, fill=black,inner sep=0pt, minimum size=2pt] at (0,0) {};
\node[circle, fill=black,inner sep=0pt, minimum size=2pt] at (6,0) {};
\node[circle, fill=blue,inner sep=0pt, minimum size=2pt] at (12,0) {};

\node at (-15,-1.5) {\color{blue}$x$};
\node at (-7,-1.5) {$r$};
\node at (0,-1.5) {$s$};
\node at (6,-1.5) {$t$};
\node at (12,-1.5) {\color{blue}$y$};

\node[circle, fill=black,inner sep=0pt, minimum size=2pt] at (0,4) {};
\node[circle, fill=blue,inner sep=0pt, minimum size=2pt] at (2,7) {};
\node[circle, fill=black,inner sep=0pt, minimum size=2pt] at (6.6,7) {};

\draw[thick] (0,0) to [out= 100, in=-80] (0,4);
\draw[blue] (0,4)--(1,5);
\draw [blue](1,5)  to [out= 55, in=-80] (2,7);
\draw [thick](6,0) to [out= 100, in=-80] (6.6,7);

\draw[dotted] (1,7) -- (8,7);

\draw [thick](6.6,7)--(7.6,9);
\draw [blue](2,7)--(3,9);

\node at (5,7.8) {$\tau_0$};

\node[circle, fill=blue,inner sep=0pt, minimum size=5pt] at (-15,14) {};
\node[circle, fill=black,inner sep=0pt, minimum size=2pt] at (-7,14) {};
\node[circle, fill=black,inner sep=0pt, minimum size=2pt] at (-1,14) {};
\node[circle, fill=black,inner sep=0pt, minimum size=2pt] at (7,14) {};

\draw[dotted] (-17,14) -- (9,14);

\draw[blue] (-15,0) to [out= 105, in=-80] (-15,14);
\draw [thick](-7,0)--(-7,14);
\draw [thick](0,4) to [out= 100, in=-80] (-1,14);
\draw [thick](7.6,9) to [out=70, in=-110] (7,14);

\draw[thick] (-7,14)--(-8,16);
\draw[thick](-1,14)--(-2,16);
\draw [thick](7,14)--(6,16);

\node at (-9,14.8) {$\tau_1$};

\node[circle, fill=black,inner sep=0pt, minimum size=2pt] at (-8.8,20) {};
\node[circle, fill=black,inner sep=0pt, minimum size=2pt] at (-1.8,20) {};
\node[circle, fill=black,inner sep=0pt, minimum size=2pt] at (5.8,20) {};
\node[circle, fill=blue,inner sep=0pt, minimum size=5pt] at (13,20) {};

\draw[dotted] (-10,20) -- (15,20);

\draw[thick] (-8,16) to [out= 120, in=-80] (-8.8,20);
\draw[thick] (-2,16) to [out= 100, in=-85] (-1.8,20);
\draw[thick](6,16) to [out=120, in=-100] (5.8,20);
\draw [blue](12,0) to [out=105, in=-90] (13,20);

\node at (7.5,20.8) {$\tau_2$};

\draw [thick] (-8.8,20)--(-7.8,22);
\draw [thick](-1.8,20)--(-2.8,22);
\draw[thick] (5.8,20)--(4.8,22);

\node[circle, fill=black,inner sep=0pt, minimum size=2pt] at (-7.5,27) {};
\node[circle, fill=black,inner sep=0pt, minimum size=2pt] at (-2.8,27) {};
\node[circle, fill=black,inner sep=0pt, minimum size=2pt] at (4.5,27) {};

\draw[dotted] (-9,27) -- (7,27);

\node at (10,27) {$\tilde{\varphi}(\mu)$};

\draw [thick] (-7.8,22) to [out= 70, in=-110] (-7.5,27);
\draw [thick](-2.8,22) to [out= 120, in=-80] (-2.8,27);
\draw[thick] (4.8,22) to [out= 120, in=-85] (4.5,27);

\end{tikzpicture}
\caption{Partial example of a $\varphi$ which is  {\bf A.3(2)}-generic below $\tilde{\varphi}(\mu)$ for the class of triangle-free graphs. $\im(\phi){\downarrow}$ is drawn in black. The nodes $r,s,t$ are in $\varphi^+[\Delta(\mu)]$, while the blue nodes $x,y$ are just in $\Delta(m)$. 
For  $X_0=\{t\}$ and  $S_0=\{s,t\}$, $\tau_0$ is
an essential age change for $S_0$.
For  $X_1=\{r,s,t\}$ and  $S_1=\{x,r,s,t\}$, $\tau_1$ is  the critical level type with a coding node at $x$ and all passing numbers $0$.  For $X_2=\{r,s,t\}$ and $S_2=\{r,s,t,y\}$, $\tau_2$ is the critical level type with a coding node at $y$, the passing number at $r$ is $1$, and the other passing numbers $0$.}
\end{center}
\end{figure}
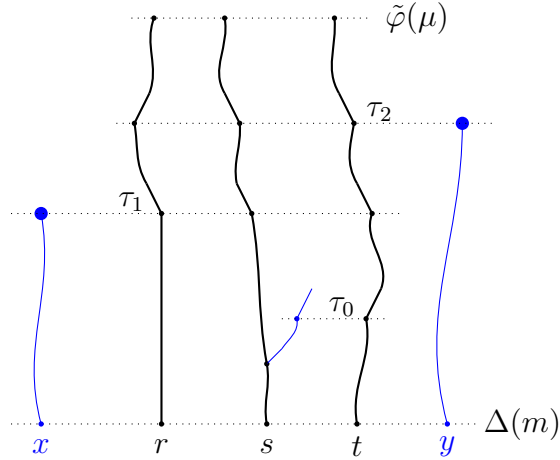

\begin{lemma}
	\label{Lem:GenericIdeal}
	$\cI\subseteq\cS$ is a right ideal. 
\end{lemma}

\begin{proof}
	Suppose $\phi = \sigma\circ \psi$ with $\sigma\in \cI$ and $\psi\in \cS$. Towards showing that $\phi\in \cI$, fix a suitably large $\mu\in \SL(\Delta)$. Note that $\wt{\psi}(\mu-1) = \nu-1$ with $\nu\in \SL(\Delta)$. Write $m = \wt{\phi}(\mu-1)+1 = \wt{\sigma}(\wt{\psi}(\mu-1))+1$.  As $\sigma\in \cI$, $\sigma$ is \axiom-generic below $\wt{\sigma}(\nu)$ so long as $\mu$ was suitably large. So given $X\subseteq \phi^+[\Delta(\mu)]\subseteq \sigma^+[\Delta(\nu)]$ and a critical level type $\tau$ based on $S\subseteq \Delta(m)$ with $X\subseteq X$ and $\crit(\tau)\not\subseteq X$, we can find $T\sqsupseteq_{age}S$ with $\ell(T)< \wt{\sigma}(\nu) \leq \wt{\phi}(\mu)$, $\ell(T)+1\in \SL(\Delta)$, and $T(X)\subseteq \im(\sigma){\downarrow}$. But consider some $s\in X$. Since $\ell(T)< \wt{\sigma}(\nu)$, it follows that $\im(\sigma){\downarrow}$ cannot split above $s$ between levels $m$ and $\ell(T)$. So as $s\in \im(\phi){\downarrow}$, we must also have $T(s)\in \im(\phi){\downarrow}$.  
\end{proof}

\begin{prop}
	\label{Prop:IdealExists}
	$\cI\cap [\rm{id}_n]\neq \emptyset$ for every $n< \omega$.
\end{prop} 
	
\begin{proof}
	We may assume that $n\in \SL(\Delta)$, and we build $\phi\in \cI\cap [\rm{id}_n]$ which is \axiom-generic below $\wt{\phi}(\mu)$ for every $\mu\in \SL(\Delta)$ with $\mu\geq n$. Suppose for some $\mu\geq n$ we have built $\phi|_{\Delta({<}\mu)}$, and write $m := \wt{\phi}(\mu-1)+1\in \SL(\Delta)$. If $\mu\not\in \SL(\Delta)$, then  any extension of $\phi|_{\Delta({<}\mu)}$ to $\Delta(\mu)$ will do. If $\mu\in \SL(\Delta)$, let $\{(\tau_j, S_j, X_j): j< N\}$ list all triples with $X_j\subseteq \phi^+[\Delta(\mu)]$ and $\tau_j$ a critical level type based on $S_j\supseteq X_j$ with $\crit(\tau_j)\not\subseteq X_j$. Writing $Y_0 = \phi^+[\Delta(\mu)]$, and suppose for some $j< N$ that $Y_j\sqsupseteq_{age} Y_0$ with $\ell(Y_j)\in \SL(\Delta)$ has been defined. Set $Z_j = Y_j\cup \Left_\Delta(\Delta(m)\setminus Y_0, \ell(Y_j))$. By Proposition~\ref{Prop:Left_Safe}, $Z_j\sqsupseteq_{age} \Delta(m)$, so by Proposition~\ref{Prop:Find_Level_Type}, we can find $T_j\sqsupseteq_{age} Z_j(S_j)$ with $\ell(T_j)+1\in \SL(\Delta)$ and $T_j\cong^* \tau_j$. Set $Y_{j+1} = \wh{T_j(X_j)}\cup \Left_\Delta(Y_j(Y_0{\setminus}X_j), \ell(T_j)+1)$.  Once $Y_N$ has been constructed, we use Proposition~\ref{Prop:Find_Level_Type} once more to find $U\sqsupseteq_{age}Y_N$ with $U\cong^* \Delta(\mu)$, and we set $\phi[\Delta(\mu)] = U$. 
\end{proof}

\begin{prop}
	\label{Prop:IdealIsA32}
	$\cI\subseteq \cS$ is an \axiom-ideal for $\cR$ and for $\cS$.
\end{prop}

\begin{proof}
	Towards proving that $\cI$ is an \axiom-ideal for $\cR$, fix $h\in \cal{AR}_\ell = \cal{BAR}_\ell$ and $\psi\in \cI$. Set $n = \depth(h) = \wt{h}(\ell-1)+1\in \SL(\Delta)$ and $p = \depth(\psi\circ h) = \wt{\psi}(n-1)+1\in \SL(\Delta)$. Write $W = \psi^+[\Delta(n)]$.  We construct $\sigma\in \cI\cap [\rm{id}_p]$ with the following properties
	\begin{itemize}
		\item 
		$\sigma[\succ_\Delta(W)]\subseteq \im(\psi){\downarrow}$.
		\item
		For any $q\geq p$, there is $r\in \SL(\Delta)$ with $\wt{\psi}(r-1)\leq \tilde{\sigma}(q-1)< \wt{\psi}(r)$, and if $\crit_\Delta(q)\subseteq \succ_\Delta(W)$, then $\tilde{\sigma}(q) = \wt{\psi}(r-1)$.
	\end{itemize}
		
	To see that this gives what we want, fix $\eta\in [\psi\circ h]$, and consider $\sigma\circ \eta$. Since $\eta[\Theta({\geq}\ell)]\subseteq \succ_\Delta(W)$, the first demand on $\sigma$ gives $\sigma\circ \eta[\Theta({\geq}\ell)]\subseteq \im(\psi){\downarrow}$, which implies that $\im(\sigma\circ \eta)\subseteq \im(\psi){\downarrow}$. This can only happen if $\im(\sigma\circ \eta)\subseteq \im(\psi)$. Thus there is $\alpha\in \demb(\Theta, \Delta)$ with $\sigma\circ \eta = \psi\circ \alpha$, and by the last demand on $\sigma$, we have $\alpha\in \cR$, and necessarily $\alpha\in [h]$. 

    If instead we want to show that $\cI$ is an \axiom-ideal for $\cS$, we fix $h\in \cal{BAS}$. Note that since $h$ is branching, we have $\rm{depth}(h)\in \SL(\Delta)$. We build $\sigma$ satisfying exactly the same two bullets as above. The justification that this gives what we want is the same, simply replacing mentions of $\Theta$ with $\Delta$.
    
    Given $x\in \Delta$ and $n\geq \ell(x)$, we write
    \begin{align*}
        \Left_\psi(x, n) = \begin{cases}
        \Left_\Delta(x, n) \quad &\text{if } x\not\in \im(\psi){\downarrow}\\[1 mm]
        \text{The $\lex$-least } y\in \succ_\Delta(x, n)\cap \im(\psi){\downarrow} \quad &\text{if } x\in \im(\psi){\downarrow}.
        \end{cases}
    \end{align*}
    and likewise for level subsets of $\Delta$. Observe that if $\mu\in \SL(\Delta)$, $T\subseteq \Delta$ is a level set with $\ell(T)\in (\wt{\psi}(\mu-1), \wt{\psi}(\mu)]$, and $n> \ell(T)$, then $\Left_\psi(T, n)\sqsupseteq_{age} T$. Therefore, let us call levels $\ell< \omega$ satisfying $\ell\in (\wt{\psi}(\mu-1), \wt{\psi}(\mu)]$ for some $\mu\in \SL(\Delta)$ \emph{$\psi$-safe}.
        
    Given $x\in \im(\psi){\downarrow}$, we set $\psi^*(x)\in \Delta$ be the $\sqsubseteq$-least node with $\psi(\psi^*(x))\sqsupseteq x$. Think of $\psi^*$ as an extension of $\psi^{-1}$ to all of $\im(\psi){\downarrow}$. Note that $\psi^*$ is an age map on any level subset of $\im(\psi){\downarrow}$. Similarly, we write $\wt{\psi}^*\colon \omega\to \omega$ for the induced map on levels, though note that $\wt{\psi}^*$ need not be injective.

	Suppose we have constructed $\sigma|_{\Delta({<}q)}$ for some $q\geq p$. Write $Q = \succ_\Delta(W, q)$ and $m = \tilde{\sigma}(q-1)+1 = \ell(\sigma^+[\Delta(q)])$. Suppose we have maintained the following inductive assumptions:
	\begin{itemize}
		\item
		$\sigma^+[Q]\subseteq \im(\psi){\downarrow}$. 
		\item
		$m = \ell(\sigma^+[Q]) \in (\wt{\psi}(\mu-1), \wt{\psi}(\mu))$ for some $\mu\in \SL(\Delta)$, which we now fix.
	\end{itemize} 
    Note that when $q = p$, we have that $\sigma^+$ is the identity on $\Delta(p)$ and that $\tilde{\sigma}(p-1)+1 = p = \wt{\psi}(n-1)+1$, so the assumptions hold in our base case. Also note that the second inductive assumption implies that $\wt{\psi}^*(m) := \mu\in \SL(\Delta)$, which will be useful for when we need to perform various parts of the construction inside $\im(\psi){\downarrow}$.
		
	If $q\not\in \SL(\Delta)$, we can immediately proceed to the construction of $\sigma[\Delta(q)]$.  There are two cases.
    \begin{itemize}
        \item 
        If $\crit_\Delta(q)\subseteq Q$, consider $\psi^*[\sigma^+[Q]]\subseteq \Delta(\mu)$. By Proposition~\ref{Prop:Find_Level_Type}, we can find $U\sqsupseteq_{age} \psi^*[\sigma^+[Q]]$ with $\ell(U)+1:= r\in \SL(\Delta)$ and $U\cong^* Q$. So also $\psi[U]\cong^* Q$, and again using Proposition~\ref{Prop:Find_Level_Type}, we can find $V\sqsupseteq_{age} \sigma^+[\Delta(q)]$ with $V\cong^* \Delta(q)$ and $V(Q) = \psi[U]$. We set $\sigma[\Delta(q)] = V$, and we note that $\wt{\sigma}(q) = \wt{\psi}(r-1)$ as required.
        \item
        If $\crit_\Delta(q)\not\subseteq Q$, find a suitably large $r\in \SL(\Delta)$ with $\wt{\psi}(r-1)\geq m$. Write $S = \Left_\psi(\sigma^+[\Delta(q)], \wt{\psi}(r-1))$.  Because $\psi$ is \axiom-generic below $\wt{\psi}(r)$, we can find $T\sqsupseteq_{age} S$ with $T\cong^* \Delta(q)$,  $\ell(T)+1\in \SL(\Delta)$, and $T(Q)\subseteq \im(\psi){\downarrow}$. We set $\sigma[\Delta(q)] = T$.    
    \end{itemize}

    If $q\in \SL(\Delta)$, then before defining $\sigma[\Delta(q)]$, we need to ensure that $\sigma$ is \axiom-generic below $\wt{\sigma}(q)$. The argument is very similar to that of Proposition~\ref{Prop:IdealExists}, but certain parts of the construction need to stay inside of $\im(\psi){\downarrow}$. Let $\{(\tau_j, S_j, X_j): j< N\}$ list all triples with $X_j\subseteq \sigma^+[\Delta(q)]$ and $\tau_j$ a critical level type based on $S_j\supseteq X_j$ with $\crit(\tau_j)\not\subseteq X_j$. Set $Y_0 = \sigma^+[\Delta(q)]$, and suppose for some $j< N$ that $Y_j\sqsupseteq_{age} Y_0$ with $\ell(Y_j)\in \SL(\Delta)$, $Y_j(\sigma^+[Q])\subseteq \im(\psi){\downarrow}$, and $\ell(Y_j)$ a $\psi$-safe level. Find a suitably large $\nu_j\in \SL(\Delta)$ with $\wt{\psi}(\nu_j-1)\geq \ell(Y_j)$, write $m_j = \wt{\psi}(\nu_j-1)+1$, and  set $Z_j = \Left_\psi(Y_j, m_j) \cup \Left_\psi(\Delta(m)\setminus Y_0, m_j)$. Let $\tau_j'$ be the critical level type isomorphic to $\tau_j$ and based on $Z_j$. Then $\crit(\tau_j')\not\subseteq Z_j(X_j\cap \sigma^+[Q])$, so since $\psi$ is \axiom-generic below $\wt{\psi}(\nu_j)$, we can find $T_j\sqsupseteq_{age} Z_j(S_j)$ with $\ell(T_j)< \wt{\psi}(\nu_j)$, $T_j\cong \tau_j'\cong \tau_j$, $\ell(T_j)+1\in \SL(\Delta)$, and $T_j(X_j\cap \sigma^+[Q])\subseteq \im(\psi){\downarrow}$. We set $Y_{j+1} = \wh{T_j(X_j)}\cup \Left_\psi(Y_j(Y_0\setminus X_j), \ell(T_j)+1)$. After forming $Y_N$, we can proceed to the construction of $\sigma[\Delta(q)]$, which occurs much in the same way as in the case $q\not\in \SL(\Delta)$, just starting from $Y_N$ instead of $\sigma^+[\Delta(q)]$. 
\end{proof}
	  
\subsection{The proof of \textbf{A.4}}
\label{Subsec:A4}

This subsection proves that Axiom \textbf{A.4} holds in both $(\cR, \cS, \leq, \leq_\cR)$ and $(\cS, \leq)$. The proof for both is nearly identical. We state \textbf{A.4} below for $(\cR, \cS, \leq, \leq_\cR)$.

\begin{theorem}
	\label{Thm:StrongHL}
	Let $f\in \cal{AR}_m$, and write 
	$$F = \{g\in \cal{AR}_{m+1}: g|_m = f\}$$
	Let $\chi\colon F\to 2$ be a coloring. Then writing $\mu = \wt{f}(m-1)+1\in \SL(\Delta)$, there is $h\in \cS\cap [\rm{id}_\mu]$ with $h\circ F$ monochromatic for $\chi$. 
\end{theorem}

\begin{rem}
    If instead we wish to prove \textbf{A.4} for $(\cS, \leq)$, we fix some $f\in \cal{AS}_m$ and set $F = \{g\in \cal{AS}_{m+1}: g|_m = f\}$. If $f$ is non-branching, then $|F| = 1$ and the desired conclusion clearly holds. If $f$ is branching, then $\wt{f}(m-1)+1\in \SL(\Delta)$, and the proof is identical to that given below, replacing mentions of $\Theta$ with $\Delta$ and $\cR$ with $\cS$.
\end{rem}

\begin{proof}
Fix $\kappa$ a suitably large uncountable cardinal. We define a poset $\bb{P}$ as follows:
\begin{itemize}
	\item 
	Members of $\bb{P}$ are functions $p\colon B(p)\times \Theta(m)\to \Delta(\ell(p))$, where $B(p)\subseteq \kappa$ is finite, $\ell(p)\geq \mu$ with $\ell(p)+1\in \rm{SL}(\Delta)$, and $p(\alpha, t)\sqsupseteq \phi\circ f^+(t)$ for each $t\in \Theta(m)$.
	\item
	Given $p, q\in \bb{P}$, we declare that $q\sqsupseteq_{\bb{P}} p$ iff $B(q)\supseteq B(p)$, $\ell(q)\geq \ell(p)$, and for each $t\in \Theta(m)$ and $\alpha\in B(p)$, we have $q(\alpha, t)\sqsupseteq \pi_s(p(\alpha, t))$. Note that $q(\alpha, t)\sqsupseteq \pi_s(p(\alpha, t))$ iff $\pi_s(q(\alpha, t))\sqsupseteq \pi_s(p(\alpha, t))$, so in particular $\sqsupseteq_\bb{P}$ is transitive.
	\item
	If $q\sqsupseteq p$ and $S\subseteq B(p)\times \Theta(m)$ is a set on which $p$ is injective, we define $\theta(p, q, S)\colon p[S]\to q[S]$ where given $(\alpha, t)\in S$, we set $\theta(p, q, S)(p(\alpha, t)) = q(\alpha, t)$.
	\item
	We declare that $q\leq_{\bb{P}} p$ iff $q\sqsupseteq_\bb{P} p$ and for each $S\subseteq B(p)\times \Theta(m)$ on which $p$ is injective, we have that $\theta(p, q, S)$ is an $\cL_\cK^*$-embedding.
\end{itemize}

\begin{lemma}
	\label{Lem:GoodConditions}
	Any condition $p\in \bb{P}$ can be strengthened to a condition $q\in \bb{P}$ with the property that $\pi_s$ is an age map on $\im(q)$.
\end{lemma}

We will call such conditions \emph{good}.

\begin{proof}
	Write $P = \pi_s[\im(p)]$. If $\pi_s$ is not injective on $\im(p)$, pass through controlled splitting gadgets as needed to obtain $n< \omega$ with $\mu_n>\ell(p)+1$ and $S\subseteq \Delta(\mu_n)$ so that if $t\in P$ has $k_t$-many extensions to a member of $\im(p)$, then also $t$ has $k_t$-many extensions to a member of $S$. After doing this, we can apply Proposition~\ref{Prop:Find_Level_Type}. The $U\sqsupseteq_{end} S$ thus obtained will be the image of the desired condition $q$.
\end{proof}

We will identify members of $F$ with their restrictions to $\Theta(m)$, as below level $m$ every member of $F$ looks like $f$. Write $\Theta(m) = \{s_i: i< N\}$ in lex order. The notation $\vec{\alpha}$ will always refer to a tuple of ordinals $\alpha_0<\cdots < \alpha_{N-1}$. We observe that if $q\leq_{\bb{P}} p \in \bb{P}$ and $\vec{\alpha}\subseteq B(p)$ are such that the map $s_i\to p(\alpha_i, s_i)$ is in $F$, then also the map $s_i\to q(\alpha_i, s_i)$ is in $F$. By a mild abuse of notation, we write $p(\vec{\alpha}, -)\colon \Theta(m)\to \ell(p)$ for the map given by $p(\vec{\alpha}, s_i) = p(\alpha_i, s_i)$.

For each $\vec{\alpha}\subseteq \kappa$ and $\epsilon< 2$, we set
\begin{align*}
	\dot{b}(\vec{\alpha}) &:= \{(q, \ell(q)): \vec{\alpha}\subseteq B(q),\, q(\vec{\alpha}, -)\in F\}\\
	\dot{b}(\vec{\alpha}, \epsilon) &:= \{(q, \ell(q)): \vec{\alpha}\subseteq B(q),\, q(\vec{\alpha}, -)\in F,\, \chi(q(\vec{\alpha}, -)) = \epsilon\}.
\end{align*}
So $\dot{b}(\vec{\alpha})$ and $\dot{b}(\vec{\alpha}, \epsilon)$ are names for subsets of $\omega$. For each $p\in \mathbb{P}$, we define
$$\dot{L}(p) := \{(q, \ell(q)): q\leq_\bb{P} p\}.$$
So $\dot{L}(p)$ names a subset of $\omega$.
\begin{lemma}
	\label{Lem:InfiniteSubset}
	$p\Vdash ``\dot{L}(p) \text{ is infinite}"$
\end{lemma} 
\begin{proof}
	Fix $q\leq_{\bb{P}} p$ and $n< \omega$. We can assume $q$ is good. We wish to find $r\leq_{\bb{P}} q$ with $\ell(r)> n$. Write $S = \pi_s[\im(q)]$, we first find $\nu\in \SL(\Delta)$ with $\nu> n$. Write $T = \Left(S, \nu)$, then apply Proposition~\ref{Prop:Find_Level_Type} to find $U\sqsupseteq_{end} T$ with $U\cong^* \im(q)$. Then define $r\in \bbP$ with $\dom(r) = \dom(q)$ and $\im(r) = U$ in the obvious way to obtain $r\leq_\bbP q$.
\end{proof}
We then set 
$$\dot{\c{G}} := \{(p, \dot{L}(p)): p\in \mathbb{P}\}.$$
So $\dot{\mathcal{G}}$ names a collection of infinite subsets of $\omega$. Observe that if we are given $p_0,...,p_{n-1}\in \bb{P}$ and $q\in \bb{P}$ with $q\leq_\bb{P} p_i$ for each $i< n$, then $q \Vdash ``\dot{L}(q)\subseteq \dot{L}(p_i)"$ for each $i< n$. Therefore, we can find $\dot{\mathcal{U}}$ a name for some non-principal ultrafilter extending $\dot{\mathcal{G}}$. Fix this $\dot{\cU}$.
\vspace{2 mm}

\begin{lemma}
	\label{Lem:FindConditionsSpl}
	For each $\vec{\alpha}$, there are good $q_{\vec{\alpha}}\in \mathbb{P}$ and $\epsilon_{\vec{\alpha}}\in \im{\chi}$ so that:
	\begin{enumerate}
		\item 
		$\vec{\alpha}\subseteq B(q_{\vec{\alpha}})$,
		\item 
		$q_{\vec{\alpha}}\Vdash ``\dot{b}(\vec{\alpha}, \epsilon_{\vec{\alpha}})\in \dot{\mathcal{U}}".$
	\end{enumerate}
\end{lemma}

\begin{proof}
	We first define a condition $p_{\vec{\alpha}}$ by first fixing some $g\in F$, the same $g$ for all $\vec{\alpha}$. We set $B(p) = \vec{\alpha}$, and for $i, j <N$, we set $p_{\vec{\alpha}}(\alpha_i, s_j) = g(s_j)$.
	Notice in particular that $p_{\vec{\alpha}}(\vec{\alpha}, -)\in F$. This implies that $p_{\vec{\alpha}}\Vdash ``\dot{L}(p_{\vec{\alpha}})\subseteq \dot{b}(\vec{\alpha}),"$ so in particular $p_{\vec{\alpha}}\Vdash ``\dot{b}(\vec{\alpha})\in \dot{\c{U}}."$ Then find good $q_{\vec{\alpha}}\leq_\bb{P} p_{\vec{\alpha}}$ and $\epsilon_{\vec{\alpha}}\in \im{\chi}$ with $q_{\vec{\alpha}}\Vdash ``\dot{b}(\vec{\alpha}, \epsilon_{\vec{\alpha}})\in \dot{\c{U}}."$
\end{proof}
\vspace{2 mm}

We are now prepared to apply the Erd\H{o}s-Rado theorem. The proof of the next lemma is almost verbatim the same as the proof of Lemma~3.3 from \cite{Zucker_BR_Upper_Bound}.
\vspace{2 mm}

\begin{lemma}
	\label{Lem:ErdosRado}
	There are countably infinite subsets $K_0 < \cdots < K_{N-1}$ of $\kappa$ so that the following hold.
	
	\begin{enumerate}
		\item 
		There are $\ell^*< \omega$ and $\epsilon^* \in \im(\chi)$ with $\ell(q_{\vec{\alpha}}) = \ell^*$ and $\epsilon_{\vec{\alpha}} = \epsilon^*$ for every $\vec{\alpha}\in \prod_{i< N} K_i$.
		
		\item 
		There are $t_i\in \Delta(\ell^*)$ with $q_{\vec{\alpha}}(\vec{\alpha}, s_i) = t_i$ for every $\vec{\alpha}\in \prod_{i< N} K_i$.
		
		\item 
		If $J\subseteq \prod_{i< N} K_i$ is finite, then $\bigcup_{\vec{\alpha}\in J} q_{\vec{\alpha}}\in \mathbb{P}$.
	\end{enumerate}
\end{lemma}

We now turn towards the construction of $h\in\cS\cap [\rm{id}_\mu]$ from the statement of Theorem~\ref{Thm:StrongHL}, which proceeds inductively by describing $h[\Delta(n)]$ for each $n\geq \mu$. For each $n\geq \mu$, we also build $S_n\sqsupseteq_{age} h^+[\Delta(n)]$ to aid the construction. In fact, for every $n> \mu$ we will simply have $S_n = h^+[\Delta(n)]$. So we start by building $S_\mu$, noting that $h^+[\Delta(\mu)] = \Delta(\mu)$. We first observe that $f^+[\Theta(m)]\sqsubseteq_{age}\{\pi_s(t_i): i< N\}$, so by Proposition~\ref{Prop:Left_Safe}, if we set $S_\mu = \Left_\Delta(\Delta(\mu)\setminus f^+[\Theta(m)], \pi_s(\ell^*))\cup \{\pi_s(t_i): i< N\}$, we have $S_\mu\sqsupseteq_{age} \Delta(\mu)$ as desired.

Assume for some $n\geq \mu$ that $h[\Delta(n-1)]$ and $S_n$ have been defined. Set $F(m):= \{\tilde{g}(m): g\in F\}$. If $n\not\in F(m)$, choose any extension of $h$ through $S_n$ to $\Delta(n)$, and set $S_{n+1} = h^+[\Delta(n+1)]$. 

If $n\in F(m)$, let $F_n = \{g\in F: g(m) = n\}$. For each $i< N$, let $\rho_i\colon \{h^+(g(s_i)): g\in F_n\}\to K_i$ be any injection. Given $g\in F_n$ and $i< N$, write $\alpha_g(i) = \rho_i\circ h^+\circ g(s_i)$, and set $\vec{\alpha}_g := \{\rho_i\circ h^+\circ g(s_i): i< N\}$. Observe that if $\alpha_g(i) = \alpha_h(j)$, then we must have $i = j$ and $g(s_i) = h(s_i)$. Also set $Y = \{(\alpha_g(i), i): g\in F_n, i< N\}$. 
Write $q_g := q_{\vec{\alpha}_g}\in \bb{P}$. Set $q = \bigcup_{g\in F_n} q_g$.  So by Lemma~\ref{Lem:ErdosRado}, we have $q\in \mathbb{P}$. Note that $\im(q) = \im(q_g)$ for any $g\in F_n$, so since each $q_g$ is good, also $q$ is good. 

We now define a function $\sigma\colon \dom(q)\to \Delta(\tilde{h}(n-1)+1)$ which we will modify to a condition in $\bb{P}$ (Lemma~\ref{Lem:FindConditionExtending}). We do this as follows:
\begin{enumerate}
	\item 
	If $\alpha\in B(q)$ is such that $\alpha = \alpha_g(i)$ for some $i< N$ and $g\in F_n$, set $\sigma(\alpha, s_i) = h^+\circ g(s_i)$.
	\item 
	Otherwise, set $\sigma(\alpha, s_i) = \Left(\pi_s(q(\alpha, s_i)), \tilde{h}(n-1)+1)$. 
\end{enumerate}
We note the following key property of $\sigma$: for any $g\in F_n$, we have $\sigma[\dom(q_g)]\sqsupseteq_{age} \pi_s[\im(q_g)]$, as can be seen using Proposition~\ref{Prop:Left_Safe}.

\begin{lemma}
\label{Lem:FindConditionExtending}
There is $r\in \bb{P}$ which end-extends $\sigma$ so that, writing $X = \bigcup_{g\in F_n} \im(g)\subseteq \Delta(n)$, we have $\{r(\vec{\alpha}_g, s_i): i< N, g\in F_n\}\cong^* X$. In particular, $r\leq_\bbP q_g$ for every $g\in F_n$. 
\end{lemma}

\begin{proof}
In order to use Proposition~\ref{Prop:Find_Level_Type} to produce the desired $r\in \bbP$, we need to build a level type $\tau$ based on $\im(\sigma)$ with $\tau|_{\sigma[Y]}\cong^* X$.

In the splitting case, if $s_j\in \Theta(m)$ is the splitting node, then the ordinal $\beta = \rho_j\circ h^+\circ g(s_j)$ does not depend on $g$, and we define $\tau$ to have $\age(\tau) = \age_\Delta(\im(\sigma))$ and a splitting node at $\sigma(\beta, s_j)$. 
 
In the age-change case, suppose $S\subseteq N$ is such that $\Theta(m)$ has an essential age change at $S$ to $\cB\in P(\sort(\Theta(m)){\cdot}S)$. Then for each $j\in S$, the ordinal $\beta_j = \rho_j\circ h^+\circ g(s_j)$ does not depend on $g$, and we define $\tau$ so that $\age(\tau)\subseteq \age(\im(\sigma))$ is given by Lemma~\ref{Lem:Essential_Change_Last} (i.e.\ the last class before the desired essential age change) and so that $\tau$ has an essential age change at $\{\sigma(\beta_j, s_j): j\in S\}$ to $\cB$. 
 
In the coding case, suppose $s_j\in \Theta(m)$ is the coding node, and set $\beta = \rho_j\circ h^+\circ g(s_j)$ (once again independent of $g$). To define $\tau$, we first discuss passing numbers. Given $\alpha\in B(q)$ and $i\in N\setminus \{j\}$, we declare that the passing number of $\sigma(\alpha, s_i)$ in $\tau$ is equal to the passing number of $q(\alpha, s_i)$. Now we define $\age(\tau)$. Write $\im(\sigma) = \{u_0\lex\cdots\lex u_{d-1}\}$, and suppose $\sigma(\beta, s_j) = u_c$. Let $\phi\colon (d-1)\to k$ be such that $\phi(i)$ is our choice of passing number for $u_{(d{\setminus}\{c\}){\cdot}i}$. We set $\age(\tau) = \la\age_\Delta(\im(\sigma)), c, \phi\ra$ (defined immediately before Section~\ref{Sec:DiagonalDiaries}).

In all three cases, we have $\age(\tau|_{\sigma[Y]}) = \age_\Delta(\sigma[Y]) = \age_\Delta(X)$, which by our construction implies that $\tau|_{\sigma[\dom(q_g)]}\cong^* \im(q_g)$ as desired.
\end{proof}

Fix $r\in \bbP$ as in Lemma~\ref{Lem:FindConditionExtending}. Since $r\leq_\bb{P} q_g$ for each $g\in F_n$, we have $r\Vdash ``\dot{b}(\vec{\alpha}_g, \epsilon^*) \in \dot{\c{U}}."$ Since we also have $r\Vdash ``\dot{L}(r)\in \dot{\c{U}},"$ we may find $y_0\leq_\bb{P} r$ and some fixed $M> \tilde{\psi}_n(n)$ so that $y_0\Vdash ``M\in \dot{L}(r)"$ and $y_0\Vdash ``M\in \dot{b}(\vec{\alpha}_g, \epsilon^*)"$ for each $g\in F_n$.
Then strengthen to  $y_1\leq_\bb{P} y_0$ in order to find $y\geq_\bb{P} y_1$ and $y_g\geq_\bb{P} y_1$ with $(y, M)\in \dot{L}(r)$ and $(y_g, M)\in \dot{b}(\vec{\alpha}_g, \epsilon^*)$. In particular, we have $\ell(y) = \ell(y_g) = M$, $y\leq_\bb{P} r$, and $\chi(y_g(\vec{\alpha}_g)) = \epsilon^*$. But since $y_1$ strengthens both $y$ and $y_g$, we must also have $y(\vec{\alpha}_g) = y_g(\vec{\alpha}_g)$. So also $\chi(y(\vec{\alpha}_g)) = \epsilon^*$.

To finish the argument, we have $h^+[X]\subseteq h^+[\Delta(n)]$ and $y[Y]\sqsupseteq_{age} h^+[X]$ with $y[Y]\cong^* X$. Use Proposition~\ref{Prop:Find_Level_Type} to find $T\supseteq y[Y]$ with $T\sqsupseteq_{age} h^+[\Delta(n)]$ and $T\cong^* \Delta(n)$. We set $h[\Delta(n)] = T$.
\end{proof}

\section{Optimality}\label{sec.optimality}

\subsection{Maximality of partition-regular subsemigroups}\label{Subsec:7.1}
Given a semigroup $S$ whose underlying set is a Polish space, and $\rmP$ is a definability property of finite colorings (i.e.\ clopen, Borel, Souslin, etc.), let us say that $S$ is \emph{$\rmP$-partition-regular}, or \emph{$\rmP$-PR}, if for any finite $\rmP$-coloring of $S$, then some right ideal of $S$ is monochromatic.

If $\cK$ is a \fr class with limit $\bK$ and a finite substructure $\bA\subseteq \bK$ has big Ramsey degree greater than $1$ (by a result of Hjorth \cite{Hjorth}, such an $\bA$ always exists so long as $|\aut(\bK)|> 1$), then we can show that the semigroup $\emb(\bK, \bK)$ is not clopen-PR as follows. Indeed, fix $\gamma\colon \emb(\bA, \bK)\to 2$ an unavoidable coloring, and define $\chi\colon \emb(\bK, \bK)\to 2$ via $\chi(\eta) = \gamma(\eta|_\bA)$. Then $\chi$ is clopen, and each color class is unavoidable.

Returning to our setting where $\cK = \mathrm{Forb}(\cF)$, fix a diary $\Gamma$, and set $\bK = \str(\Gamma)$. The semigroup $S := \demb(\Gamma, \Gamma)$ can then be seen as a subsemigroup of $\emb(\bK, \bK)$, and Theorem~\ref{Thm:General_GP_Silver} shows that $S$ is Souslin-PR. Moreover, $S\subseteq \emb(\bK, \bK)$ is a ``large" submonoid in the following sense: building on Definition~\ref{Def:Recurrent}, let us call $S\subseteq \emb(\bK, \bK)$ \emph{recurrent} if $S$ meets every right ideal of $\emb(\bK, \bK)$. Then $\demb(\Gamma, \Gamma)\subseteq \emb(\bK, \bK)$ is recurrent since any two diaries coding the \fr limit are bi-embeddable.
We remark that in \cite{Dobrinen_SDAP}, the submonoids of $\emb(\bK, \bK)$ which are shown to be Borel-PR (even topological Ramsey spaces for certain $\bK$) are also recurrent,
whereas the submonoid of self-embeddings of the Rado graph considered in \cite{Dobrinen_Rado} is not.
The main result of this section shows that Theorem~\ref{Thm:General_GP_Silver} is optimal in the following sense.

\begin{theorem}
    \label{Thm:LargestSemigroup}
    If $T\subseteq \emb(\bK, \bK)$ is a Polish subsemigroup of $\emb(\bK, \bK)$ with $S\subsetneq T$, then $T$ is not clopen-PR. 
\end{theorem}

To prove this, we first fix $\bA\in \cK$ and investigate the left action of $S$ on $\emb(\bA, \bK)$ by post-composition. Call $X\subseteq \emb(\bA, \bK)$ \emph{large} if there is $\eta\in \emb(\bK, \bK)$ with $X\supseteq \{\eta\circ f: f\in \emb(\bA, \bK)\}$, and call $X$ \emph{unavoidable} if $\emb(\bA, \bK)\setminus X$ is not large. Notice that since $S\subseteq \emb(\bK, \bK)$ is recurrent, it follows that for any $f\in \emb(\bA, \bK)$, the set $S{\cdot}f$ is unavoidable.

\begin{lemma}
    \label{Lem:S_Action}
    If $\bA\in \cK$ has big Ramsey degree $\ell< \omega$ and $\{f_i: i< \ell\}\subseteq \emb(\bA, \bK)$ are such that $S{\cdot}f_i\cap S{\cdot}f_j = \emptyset$ for $i\neq j< \ell$, then $\bigcup_{i< \ell} S{\cdot}f_i$ is large.
\end{lemma}

\begin{proof}
    Towards a contradiction, suppose not. Write $X = \emb(\bA, \bK)\setminus \left(\bigcup_{i< \ell} S{\cdot}f_i\right)$. Then $\emb(\bA, \bK) = S{\cdot}f_0\sqcup\cdots\sqcup S{\cdot}f_{\ell-1}\sqcup X$ is a partition of $\emb(\bA, \bK)$ into $\ell+1$-many unavoidable sets, a contradiction to our assumption that the big Ramsey degree of $\bA\in \cK$ is $\ell$. 
\end{proof}

The last observation that we need is that given $\bA\in \cK$ with big Ramsey degree $\ell< \omega$, we can find $\{f_i: i< \ell\}\subseteq \emb(\bA, \bK)$ as in Lemma~\ref{Lem:S_Action}. To see this, given $f\in \emb(\bA, \bK)$, let $\Gamma{\cdot}f$ denote the unique diary with coding nodes labelled by the members of $\bA$ such that the subdiary of $\Gamma$ induced by $f[\bA]$ (see the discussion after Definition~\ref{Def:EmbDiary}) is isomorphic to $\Gamma{\cdot}f$ via an isomorphism which respects the labels on coding nodes. Because diaries coding $\bK$ are big Ramsey structures for $\bK$, the map $f\to \Gamma{\cdot}f$ partitions $\emb(\bA, \bK)$ into $\ell$ unavoidable parts. Form $\{f_i: i< \ell\}$ by selecting one embedding per part.

\begin{proof}[Proof of Theorem~\ref{Thm:LargestSemigroup}]
    Fix a Polish subsemigroup $T\subseteq \emb(\bK, \bK)$ with $S\subsetneq T$. As $S\subseteq \emb(\bK, \bK)$ is closed, this implies that for some $\bA\in \cK$, $f\in \emb(\bA, \bK)$, and $t\in T$, we have $\Gamma{\cdot}tf\neq \Gamma{\cdot}f$. Let $\ell$ denote the big Ramsey degree of $\bA\in \cK$, and find $\{f_i: i< \ell\}\subseteq \emb(\bA, \bK)$ as indicated above. We may suppose that $f_0 = f$ and $f_{\ell-1} = tf$, implying that $T{\cdot}f_0\supseteq S{\cdot}f_0\cup S{\cdot}f_{\ell-1}$. Assuming that $T$ is Borel-PR for clopen colorings, we will show that the big Ramsey degree of $\bA\in \cK$ is at most $\ell-1$, obtaining a contradiction. So fix $r< \omega$ and a coloring $\gamma\colon \emb(\bA, \bK)\to r$. We define for each $i< \ell-1$ elements $\eta_i\in T$ and clopen colorings $\gamma_i\colon T\to r$ by induction. Set $\gamma_0(x) = \gamma(x{\cdot}f_0)$. If the clopen coloring $\gamma_i\colon T\to T$ has been defined, then by our assumption on $T$, we can find $\eta_i\in T$ with $\gamma_i{\cdot}\eta_i$ constant. If $i< \ell-2$, we set $\gamma_{i+1}(x) = \gamma(\eta_0\circ \cdots \circ \eta_i\circ x\cdot f_{i+1})$. At the end, set $\eta = \eta_0\circ\cdots \circ \eta_{\ell-2}$. Then $\gamma\circ \eta$ takes on at most $(\ell-1)$-many colors on the large set $T{\cdot}f_0\cupdots T{\cdot}f_{\ell-2}\supseteq S{\cdot}f_0\cupdots S{\cdot}f_{\ell-1}$. In particular, we can find $\theta \in \emb(\bK, \bK)$ so that $\gamma\cdot (\eta\circ \theta)$ takes on at most $\ell-1$ colors.  
\end{proof}

\subsection{A counterexample for $\emb(\bK, \bK)$.}\label{Subsec:7.2}

We have seen above that for any \fr class $\cK$ with limit $\bK$ such that $|\aut(\bK)|> 1$, we have that $\emb(\bK, \bK)$ is not Borel-PR. However, one might be tempted to hope that given a finite Borel coloring of $\emb(\bK, \bK)$ that one can zoom in to a copy to make the coloring appear much nicer. For \emph{uniformly clopen} colorings of $\emb(\bK, \bK)$, i.e.\ those which are induced from colorings of $\emb(\bA, \bK)$ for some finite $\bA\subseteq \bK$, the bi-embeddability of diaries implies that for any finite uniformly clopen coloring, we may zoom in to a copy of $\bK$ on which the coloring is recurrent. However, even upon considering non-uniformly clopen colorings, this is not in general possible.

Consider $\bK = \bbQ$; while this is not technically one of the classes treated in this paper, it is easy to work with, and the style of counterexample can be easily generalized. Enumerate $\bbQ = \{q_n: n< \omega\}$. Given $\eta\in \emb(\bbQ, \bbQ)$, let $\{\ol{\eta}(n): n< \omega\} = \{m< \omega: q_m\in \im(\eta)\}$, where $\ol{\eta}(0) < \ol{\eta}(1)<\cdots$. We define a coloring $\gamma\colon \emb(\bbQ, \bbQ)\to 2$ via
\begin{align*}
    \gamma(\eta) = \begin{cases}
        1\quad &\text{if }\ol{\eta}(0) = n \text{ and }q_{\ol{\eta}(0)}<\cdots < q_{\ol{\eta}(n)}\\[1 mm]
        0 \quad &\text{else}.
    \end{cases}
\end{align*}
We note that $\gamma$ is clopen. Towards a contradiction, suppose we could find $\eta\in \emb(\bK, \bK)$ with $\gamma\cdot \eta$ recurrent. We may assume, zooming in to a further subcopy if needed, that $\gamma(\eta) = 1$. Write $n = \ol{\eta}(0)$. Notice that if $h\in \emb(\bbQ, \bbQ)$ satisfies $h\circ \eta^{-1}(q_{\ol{\eta}(i)}) = \eta^{-1}(q_{\ol{\eta}(i)})$ for each $i\leq n$, then $\gamma(\eta\circ h) = 1$. Now suppose $\theta\in \emb(\bbQ, \bbQ)$ is such that $\ol{\theta}(0) = N > n$ and satisfies $\gamma(\theta) = 1$. We show that for any set $Y\subseteq \bbQ$ with $|Y| = n+1$, there is $h\in \emb(\bbQ, \bbQ)$ fixing $Y$ pointwise, but with $\gamma(\theta\circ h) = 0$. To do this, we first arrange so that $\theta[Y] = \theta\circ h[Y]$ are the least-enumerated elements of $\im(\theta\circ h)$, then we arrange so that the first $N+1$ enumerated elements of $\theta\circ h$ are not enumerated in correct rational order. As $\ol{\theta\circ h}(0)\geq \ol{\theta}(0)$, this is enough to conclude that $\gamma(\theta\circ h) = 0$. It follows that $\gamma\cdot \eta \neq \gamma\cdot \theta$. In particular, $\gamma$ cannot be recurrent.

A similar style of example yields a countable coloring $\delta\colon \emb(\bbQ, \bbQ)\to \omega$ so that all colors are unavoidable, i.e.\ for any $h\in \emb(\bbQ, \bbQ)$, we have $\delta[h\circ \emb(\bbQ, \bbQ)] = \omega$. Given $\eta\in \emb(\bbQ, \bbQ)$, we set
\begin{align*}
    \delta(\eta) = n
        \quad &\text{if $n$ is largest with }q_{\ol{\eta}(0)}<\cdots < q_{\ol{\eta}(n)}
\end{align*}
Because the action of $\emb(\bbQ, \bbQ)$ on the space of enumerations of $\bbQ$ is recurrent, we see that all colors are unavoidable. More generally, considering the recurrent action of $\emb(\bK, \bK)$ on the space of diaries, any countable Borel partition of this space yields a countable coloring of $\emb(\bK, \bK)$ with every color class unavoidable.

\subsection{Not all diaries behave the same}\label{Subsec:7.3}

We have seen in Corollary~\ref{Cor:Ellentuck} that \emph{there is} a diary $\Delta$ coding $\bK$ so that for any Ellentuck-Souslin-measurable finite coloring $\chi$ of $\demb(\Delta, \Delta)$, there is $h\in \demb(\Delta, \Delta)$ with $\chi\circ h$ constant. In this section, we consider the class $\cK$ of finite graphs and its \fr limit $\bK$, the \emph{Rado graph}. We construct a diary $T$ coding the Rado graph for which the Ellentuck topology on $\emb(T, T)$ is discrete. In particular, the conclusion of Corollary~\ref{Cor:Ellentuck} cannot be strengthened to hold for all diaries coding $\bK$.

The exact big Ramsey degrees for the class of finite graphs were characterized by Laflamme, Sauer, and Vuksanovic \cite{LSV}. The diaries coding objects in $\cK$ are particularly straightforward to describe, and we call them \emph{LSV-trees}. An LSV-tree is simply a subtree $T\subseteq 2^{<\omega}$ where every level contains exactly one terminal coding node or exactly one splitting node, and if $m< \rm{ht}(T)$ is a splitting level and $s\in T(m)$, then $s^\frown 0\in T(m+1)$.

Fix an enumerated Rado graph $\bK = \{r_n: n< \omega\}$. We will build an LSV-tree $T\subseteq 2^{<\omega}$ coding a Rado graph order-isomorphic to $\bK$. We will denote the coding nodes of $T$ by $\{c_n: n< \omega\}$, with $|c_0|< |c_1|<\cdots$. Given a node $s\in T$, the \emph{splitting predecessor} of $s$, denoted $\rm{SpPred}_T(s)$, denotes the maximal splitting node $t\in T$ with $t\sqsubseteq s$. The key feature of our construction is that $\rm{SpPred}_T(c_0) = \emptyset$, and for $n> 0$, we have $|\rm{SpPred}_T(c_n)|< |c_{n-1}|$. If $T$ satisfies this property and $\eta\in \emb(T, T)$ satisfies $\eta(\emptyset) = \emptyset$, then we must have $\eta= \rm{id}_T$. Indeed, if $\eta$ fixes $\emptyset$, then $\eta$ must fix $c_0$, so $\eta$ fixes all levels up to $|c_0|$, but then $\eta$ must fix $c_1$, so also all levels up to $|c_1|$, etc.

To begin our construction, set $c_0 = 100$, and define $T(3) = T(|c_0|) = \{100, 000, 010, 101\}$.
Assign all three members of $T(|c_0|)\setminus \{c_0\}$ the type string $\emptyset$, i.e.\ all of these represent the same type over the empty structure.

Now suppose that $T(|c_n|)$ has been constructed, and that furthermore, every node $p\in T(|c_n|)\setminus \{c_n\}$ has been assigned a type string $\tau_p\in 2^n$ via $\tau_p(m) = p(|c_m|)$ in such a way that the map $p\to \tau_p$ is at least $3$-to-$1$. Now, consider the type of $r_{n+1}$ over $\{r_0,...,r_{n-1}\}$, which we can encode as some $\sigma\in 2^n$. Choose some $p^*\in T(|c_n|)\setminus \{c_n\}$ with $\tau_{p^*} = \sigma$. This will be the node that we extend to $c_{n+1}$. We set $c_{n+1}(|c_n|)$ depending on the adjacency of $r_n$ and $r_{n+1}$, and set $c_{n+1}(m) = 0$ for all $|c_n|< m < |c_{n+1}|$ (once we decide what $|c_{n+1}|$ will be). For $q\in T(|c_n|)\setminus \{c_n, p^*\}$, we first extend to level $|c_n|+1$ by choosing a $0$ or $1$ in such a way that for each $\phi\in 2^n$ and $i< 2$, there is $q\in T(|c_n|)\setminus \{c_n, p^*\}$ with $\tau_q = \sigma$ and so that $q$ is extended by appending $i$. This is possible by our $3$-to-$1$ inductive assumption. Notice that these extensions to level $|c_n|+1$ determine the type string for extensions to level $|c_{n+1}|$. Thus at level $|c_n|+1$, we have ensured that every possible type over $\{r_0,...,r_n\}$ is represented, but perhaps not three times yet. In the remaining levels between $|c_n|+1$ and $|c_{n+1}|$, we simply perform enough splitting to ensure the $3$-to-$1$ condition. 

The existence of this pathological diary motivates the following question.

\begin{que}
    \label{Que:Ellentuck}
    Is there a diary $\Gamma$ coding $\bK$ with the property that for any finite Ellentuck-BP coloring $\chi$ of $\demb(\Gamma, \Gamma)$, there is $h\in \demb(\Gamma, \Gamma)$ with $\chi\circ h$ constant? For which diaries can we make this conclusion, even just for Ellentuck-Souslin-measurable colorings?
\end{que}

Another question raised by the diary constructed in this subsection concerns the possibility of proving that certain subsets or colorings of $\demb(\Gamma, \Gamma)$ are ``completely Ramsey." For instance, Theorem~\ref{Thm:TRS} has the stronger consequence that finite Ellentuck-BP colorings of $\cS$ are \emph{completely Ramsey}, namely, given such a coloring $\chi$, then for any $f\in \cal{AR}$ with $n = \rm{depth}(f)$, there is $h\in \cS$ \emph{which is the identity on $\Delta_n$} (i.e.\ $h\in \cS_\mu$ for some $\mu\in \SL(\Delta)$ with $\mu\geq n$) such that $\chi\circ h$ is monochromatic on $\{\phi\in \cR: f\sqsubseteq \phi\}$. So while we are asking for a smaller monochromatic set (namely $\{\phi\in \cR: f\sqsubseteq \phi\}$ instead of all of $\cR$), we place stronger demands on $h$, namely that $h$ fixes the first $n$ levels. To attempt to generalize this idea, fix a monoid $M$ whose underlying set is a zero-dimensional Polish space. In addition, assume that the identiy $1_M$ has a base of clopen submonoids. Given a definability property $\rmP$ of finite colorings, let us call $M$ \emph{completely $\rmP$-partition-regular} if for any clopen submonoid $U\subseteq M$ and any finite $\rmP$-coloring $\chi$ of $M$, there is $h\in U$ so that $\chi\circ h$ is a clopen coloring\footnote{The authors thank Spencer Unger for helpful discussions about possible notions of ``completely Ramsey."}. In particular, if a monoid $M$ has isolated identity and any Borel non-clopen colorings at all, then $M$ cannot be completely Borel-PR. Hence for the LSV-tree $T$ constructed in this section, $\emb(T, T)$ cannot be completely Borel-PR. 

\begin{que}
    Is the monoid $\demb(\Delta, \Delta)$ \emph{completely} Ellentuck-Souslin-PR? More generally, characterize the diaries $\Gamma$ coding $\bK$ such that $\demb(\Gamma, \Gamma)$ is completely $\rmP$-PR for $\rmP$ equal to Borel, Souslin-measurable, etc.
\end{que}

\bibliographystyle{amsplain}
\bibliography{bibzucker}

\end{document}